\documentclass[11pt, oneside]{mnthesis}
\usepackage{amssymb,amsthm, amsmath, amsfonts} 
\usepackage{graphics}                 
\usepackage{color}                    
\usepackage{hyperref}                 
\usepackage[all]{xy}
\usepackage{enumerate}

\newtheorem{theorem}{Theorem}[section]
\newtheorem{lemma}[theorem]{Lemma}
\newtheorem{proposition}[theorem]{Proposition}

\newtheorem{corollary}[theorem]{Corollary}
\newtheorem{definition}[theorem]{Definition}
\newtheorem{example}[theorem]{Example}
\newtheorem{remark}[theorem]{Remark}
\numberwithin{equation}{section}

\DeclareMathOperator{\grad}{grad}
\DeclareMathOperator{\rad}{rad}
\DeclareMathOperator{\Ob}{Ob}
\DeclareMathOperator{\pd}{pd}
\DeclareMathOperator{\gldim}{gldim}

\DeclareMathOperator{\Hom}{Hom}
\DeclareMathOperator{\End}{End}
\DeclareMathOperator{\Ext}{Ext}
\DeclareMathOperator{\ext}{ext}
\DeclareMathOperator{\Aut}{Aut}
\DeclareMathOperator{\Top}{Top}

\DeclareMathOperator{\Stab}{Stab}

\begin{document}
\bibliographystyle{hunsrt} 

\phd 

\title{\bf A generalized Koszul theory and its applications in representation theory}
\author{Liping Li}
\campus{University of Minnesota}
\program{School of Mathematics}
\director{Peter Webb, Adviser}

\submissionmonth{July} 
\submissionyear{2012} 

\abstract{
There are many structures (algebras, categories, etc) with natural gradings such that the degree 0 components are not semisimple. Particular examples include tensor algebras with non-semisimple degree 0 parts, extension algebras of standard modules of standardly stratified algebras. In this thesis we develop a generalized Koszul theory for graded algebras (categories) whose degree 0 parts may be non-semisimple. Under some extra assumption, we show that this generalized Koszul theory preserves many classical results such as the Koszul duality. Moreover, it has some close relation to the classical theory. Applications of this generalized theory to finite EI categories, directed categories, and extension algebras of standard modules of standardly stratified algebras are described. We also study the stratification property of standardly stratified algebras, and classify algebras standardly (resp., properly) stratified for all linear orders.
}
\words{331}    
\acknowledgements{

The author would like to express great appreciation to his advisor, Professor Peter Webb. Under his sincere direction and with his kind help, the author developed the work described in this thesis. He also wants to thank Professor Mazorchuk, Professor Green for the discussions and comments on the generalized Koszul theory and stratification theory, which are the main framework of this thesis.
}
\dedication{To my son Nitou (nickname of Yifeng Li), for the so many funny things and troubles he brought. }


\beforepreface

\afterpreface


\chapter{Introduction}
\label{intro_chapter}

The classical Koszul theory plays an important role in the representation theory of graded algebras. However, there are a lot of structures (algebras, categories, etc) having natural gradings with non-semisimple degree 0 parts, to which the classical theory cannot be applied. Particular examples of such structures include tensor algebras generated by non-semisimple algebras $A_0$ and $(A_0, A_0)$-bimodules $A_1$, and extension algebras of standard modules of standardly stratified algebras (see \cite{Li4}). Therefore, we are motivated to develop a generalized Koszul theory which can be used to study the above structures, and preserves many classical results such as the Koszul duality. Moreover, we also hope to get a close relation between this generalized theory and the classical theory.

In \cite{Green3,Madsen2,Madsen3,Woodcock} several generalized Koszul theories have been described, where the degree 0 part $A_0$ of a graded algebra $A$ is not required to be semisimple. In \cite{Woodcock}, $A$ is supposed to be both a left projective $A_0$-module and a right projective $A_0$-module. In Madsen's paper \cite{Madsen3}, $A_0$ is supposed to have finite global dimension. These requirements are too strong for us. The theory developed by Green, Reiten and Solberg in \cite{Green3} works in a very general framework, but some efforts are required to fill the gap between their theory and our applications. Moreover, in all these papers, relations between the generalized theories and the classical theory are lacking.

Thus we want to develop a generalized Koszul theory which can inherit many useful results of the classical theory, and can be applied to graded structures such as finite EI categories, directed categories, extension algebras, etc. Explicitly, Let $A = \bigoplus _{i \geqslant 0} A_i$ be a \textit{positively graded, locally finite} $k$-algebra generated in degrees 0 and 1, i.e., $\dim _k A_i < \infty$ and $A_1 \cdot A_i = A_{i+1}$ for all $i \geqslant 0$, where $A_0$ is an arbitrary finite-dimensional algebra. We define generalized \textit{Koszul modules} and \textit{Koszul algebras} in a way similar to the classical case. That is, a graded $A$-module $M$ is \textit{Koszul} if $M$ has a linear projective resolution, and $A$ is a \textit{Koszul algebra} if $A_0$ viewed as a graded $A$-module is Koszul. We also define \textit{quasi-Koszul modules} and \textit{quasi-Koszul algebras}: $M$ is \textit{quasi-Koszul} if the $\Ext _A^{\ast} (A_0, A_0)$-module $\Ext_A^{\ast} (M, A_0)$ is generated in degree 0, and $A$ is a \textit{quasi-Koszul algebra} if $A_0$ is a quasi-Koszul $A$-module. It turns out that this generalization works nicely for our goal. Indeed, if $A_0$ satisfies the following splitting condition (S), many classical results described in \cite{BGS,Green1,Green2,Martinez} generalize to our context.\\

\textbf{(S): Every exact sequence $0 \rightarrow P \rightarrow Q \rightarrow R \rightarrow 0$ of left (resp., right) $A_0$-modules splits if $P$ and $Q$ are left (resp., right) projective $A_0$-modules.}\\

In particular, we obtain the Koszul duality.

\begin{theorem}
Let $A = \bigoplus _{i \geqslant 1} A_i$ be a locally finite graded algebra. If $A$ is a Koszul algebra, then $E = \Ext ^{\ast}_A (-, A_0)$ gives a duality between the category of Koszul $A$-modules and the category of Koszul $\Gamma$-modules. That is, if $M$ is a Koszul $A$-module, then $E(M)$ is a Koszul $\Gamma$-module, and $E_{\Gamma}EM = \Ext ^{\ast} _{\Gamma} (EM, \Gamma_0) \cong M$ as graded $A$-modules.
\end{theorem}

We then study the homological properties of generalized Koszul algebras. Under the assumption that $A_0$ is a self-injective algebra (thus the above splitting property is satisfied), we generalize many classical results described in \cite{BGS,Green1,Green2,Martinez}.

Let $\mathfrak{r}$ be the radical of $A_0$ and $\mathfrak{R} = A \mathfrak{r} A$ be the two-sided ideal generated by $\mathfrak{r}$. Define a quotient algebra $\bar{A} = A / A \mathfrak{r} A = \bigoplus _{i \geqslant 0} A_i / (A \mathfrak{r} A)_i$. For a graded $A$-module $M = \bigoplus _{i \geqslant 0} M_i$, we let $\bar{M} = M / \mathfrak{R} M =  \bigoplus _{i \geqslant 0} M_i / (\mathfrak{R} M)_i$. We prove that $\bar{M}$ is a well defined $\bar{A}$-module, and show that $M$ is generated in degree 0 if and only if the corresponding graded $\bar{A}$-module $\bar{M}$ is generated in degree 0, establishing a correspondence between our generalized Koszul theory and the classical theory as follows:

\begin{theorem}
Let $A = \bigoplus _{i \geqslant 1} A_i$ be a locally finite graded algebra and $M$ be a graded $A$-module. Suppose that both $A$ and $M$ are projective $A_0$-modules. Then $M$ is generalized Koszul if and only if  the corresponding grade $\bar{A}$-module $\bar{M}$ is classical Koszul. In particular, $A$ is a generalized Koszul algebra if and only if $\bar{A}$ is a classical Koszul algebra.
\end{theorem}

We then focus on the applications of this generalized Koszul theory. First we define \textit{directed categories}. A directed category $\mathcal{C}$ is a $k$-linear category equipped with a partial order $\leqslant$ on $\Ob \mathcal{C}$ such that for each pair of objects $x, y \in \Ob \mathcal{C}$, the space of morphisms $\mathcal{C} (x,y)$ is non-zero only if $x \leqslant y$. Directed categories include the $k$-linearizations of skeletal \textit{finite EI categories} as special examples, which are small categories with finitely many morphisms such that every endomorphism is an isomorphism. This partial order determines a canonical preorder $\preccurlyeq$ on the isomorphism classes of simple representations. Following the technique in \cite{Webb3}, we develop a stratification theory for directed categories, describe the structures of standard modules and characterize every directed category $\mathcal{C}$ standardly stratified with respect to the canonical preorder.

By the correspondence between graded $k$-linear categories and graded algebras described in \cite{Mazorchuk2}, we can view a graded directed category as a graded algebra and vice-versa. Therefore, all of our results on graded algebras can be applied to graded directed categories. In particular, we describe a relation between the generalized Koszul theory and the stratification theory over directed categories. For every directed category $\mathcal{C}$ we construct a directed subcategory $\mathcal{D}$ such that the endomorphism algebra of each object in $\mathcal{D}$ is one-dimensional. With this construction, we acquire another correspondence between the classical Koszul theory and our generalized Koszul theory for directed categories.

\begin{theorem}
Let $\mathcal{A}$ be a graded directed category and suppose that $\mathcal{A}_0$ has the splitting property (S). Construct $\mathcal{D}$ as before. Then:
\begin{enumerate}
\item $\mathcal{A}$ is a Koszul category in our sense if and only if $\mathcal{A}$ is standardly stratified and $\mathcal{D}$ is a Koszul category in the classical sense.
\item If $\mathcal{A}$ is a Koszul category, then a graded $\mathcal{A}$-module $M$ generated in degree 0 is Koszul if and only if $M \downarrow _{\mathcal{D}} ^{\mathcal{A}}$ is a Koszul $\mathcal{D}$-module and $M$ is a projective $\mathcal{A}_0$-module.
\end{enumerate}
\end{theorem}

We then consider the application of our generalized Koszul theory to finite EI categories, which include finite groups and finite posets as examples. They have nice combinatorial properties which can be used to define length gradings on the sets of morphisms. We discuss the possibility to put such a grading on an arbitrary finite EI category. In particular, we introduce \textit{finite free EI categories} and study their representations in details, and give a sufficient condition for their category algebras to be quasi-Koszul.

\begin{theorem}
Let $\mathcal{E}$ be a finite free EI category. If every object $x \in \Ob \mathcal{E}$ is either left regular or right regular, then $k \mathcal{E}$ is quasi-Koszul. Moreover, $k\mathcal{E}$ is Koszul if and only if $\mathcal{E}$ is standardly stratified.
\end{theorem}

Let $A$ be a basic finite-dimensional algebra, and let $(\Lambda, \leqslant)$ be a finite preordered set parameterizing all simple $A$-modules $S_{\lambda}$ (up to isomorphism). This preordered set also parameterizes all indecomposable projective $A$-modules $P_{\lambda}$ (up to isomorphism). According to \cite{Cline}, the algebra $A$ is standardly-stratified with respect to $(\Lambda, \leqslant)$ if there exist modules $\Delta_{\lambda}$, $\lambda \in \Lambda$ (called \textit{standard modules}, such that the following conditions hold:
\begin{enumerate}
\item the composition factor multiplicity $[\Delta_{\lambda} : S_{\mu}] = 0$ whenever $\mu \nleqslant \lambda$;
and
\item for every $\lambda \in \Lambda$ there is a short exact sequence $0 \rightarrow K_{\lambda} \rightarrow P_{\lambda} \rightarrow \Delta_{\lambda} \rightarrow 0$ such that $K_{\lambda}$ has a filtration with factors $\Delta_{\mu}$ where $\mu > \lambda$.
\end{enumerate}
Let $\Delta$ be the direct sum of all standard modules and $\mathcal{F} (\Delta)$ be the full subcategory of $A$-mod such that each object in $\mathcal{F} (\Delta)$ has a filtration by standard modules. Since standard modules of $A$ are relative simple in $\mathcal{F} (\Delta)$, we are motivated to investigate the extension algebra $\Gamma = \Ext _A^{\ast} (\Delta, \Delta)$ of standard modules. These extension algebras were studied in \cite{Abe, Drozd, Klamt, Mazorchuk1, Miemietz}. Specifically, we are interested in the stratification property of $\Gamma$ with respect to $(\Lambda, \leqslant)$ and $(\Lambda, \leqslant ^{\textnormal{op}})$, and its Koszul property since $\Gamma$ has a natural grading. A particular question is to determine when $\Gamma$ is a generalized Koszul algebra, i.e., $\Gamma_0$ has a linear projective resolution.

Choose a fixed set of orthogonal primitive idempotents $\{ e_{\lambda} \}_{\lambda \in \Lambda}$ for $A$ such that $\sum _{\lambda \in \Lambda} e_{\lambda}= 1$. We can define a $k$-linear category $\mathcal{A}$ as follows: $\Ob \mathcal{A} = \{ e_{\lambda} \} _{\lambda \in \Lambda}$; for $\lambda, \mu \in \Lambda$, the morphism space $\mathcal{A} (e_{\lambda}, e_{\mu}) = e_{\mu} A e_{\lambda}$. A \textit{$k$-linear representation} of $\mathcal{A}$ is a $k$-linear functor from $\mathcal{A}$ to the category of finite-dimensional $k$-vector spaces. It is clear that $\mathcal{A}$-rep, the category of finite-dimensional $k$-linear representations of $\mathcal{A}$, is equivalent to $A$-mod. Call $\mathcal{A}$ the \textit{associated category} of $A$. We show that the associated category $\mathcal{E}$ of $\Gamma = \Ext _A^{\ast} (\Delta, \Delta)$ is a \textit{directed category} with respect to $\leqslant$, and characterize the stratification properties of $\mathcal{E}$ with respect to $\leqslant$ and $\leqslant ^{\textnormal{op}}$. As an analogue to linear modules of graded algebras, we define \textit{linearly filtered modules}. With this terminology, a sufficient condition is obtained for $\Gamma$ to be a generalized Koszul algebra.

Dlab and Ringel showed in \cite{Dlab2} that a finite-dimensional algebra $A$ is quasi-hereditary for all linear orderings of the simple modules (up to isomorphism) if and only if $A$ is a hereditary algebra. In \cite{Frisk1} stratification property of $A$ for different ordering of the simple modules was studied. These results motivate us to classify algebras standardly stratified or properly stratified for all linear orderings of the simple modules, which include hereditary algebras as special cases.

We prove that if $A$ is standardly stratified for all linear orders, then its associated category $\mathcal{A}$ is a directed category with respect to some partial order $\leqslant$ on $\Lambda$. Therefore, $J = \bigoplus _{\lambda \neq \mu \in \Lambda} e_{\mu} A e_{\lambda}$ can be viewed as a two-sided ideal of $A$ with respect to this chosen set of orthogonal primitive idempotents, and $A = A_0 \oplus J$ as vector spaces, where $A_0 = \bigoplus _{x \in \Ob \mathcal{A}} \mathcal{A} (x, x)$ constitutes of all endomorphisms in $\mathcal{A}$. With this observation, we describe several characterizations of algebras stratified for all linear orders, as well as a classification of these algebras. Explicitly, 

\begin{theorem}
Let $A$ be a basic finite-dimensional $k$-algebra whose associated category $\mathcal{A}$ is directed. Then the following are equivalent:
\begin{enumerate}
\item $A$ is standardly stratified (resp., properly stratified) for all linear orders;
\item the associated graded algebra $\check{A}$ is standardly stratified (resp., properly stratified) for all linear orders;
\item $\check{A}$ is the tensor algebra generated by $A_0 = \bigoplus _{\lambda \in \Lambda} \mathcal{A} (e_{\lambda}, e_{\lambda}) = e_{\lambda} A e_{\lambda}$ and a left (resp., left and right) projective $A_0$-module $\check{A}_1$.
\end{enumerate}
\end{theorem}

Let $\preccurlyeq$ be a particular linear order for which $\mathcal{A}$ is standardly stratified. It is well known that $\mathcal{F} (_{\preccurlyeq} \Delta)$ is an additive category closed under extensions, direct summands, and kernels of epimorphisms (see \cite{Cline, Dlab1, Webb3}). But in general it is not closed under cokernels of monomorphisms. However, if $A$ is standardly stratified with respect to all linear orders, there exists a (not necessarily unique) particular linear order $\preccurlyeq$ for which the corresponding category $\mathcal{F} (_{\preccurlyeq} \Delta)$ is closed under cokernels of monomorphisms. We give a criterion and classify all \textit{quasi-hereditary algebras} satisfying this property: they are precisely quotient algebras of finite-dimensional hereditary algebras.

In practice it is hard to determine whether there exists a linear order $\preccurlyeq$ for which $A$ is standardly stratified and the corresponding category $\mathcal{F} (_{\preccurlyeq} \Delta)$ is closed under cokernels of monomorphisms. Certainly, checking all linear orders is not an ideal way to do this. We then describe an explicit algorithm to construct a set $\mathcal{L}$ of linear orders with respect to all of which $A$ is standardly stratified. Moreover, if there exists a linear order $\preccurlyeq$ such that $A$ is standardly stratified and the corresponding category $\mathcal{F} (_{\preccurlyeq} \Delta)$ is closed under cokernels of monomorphisms, then $\preccurlyeq \in \mathcal{L}$.

The layout of this thesis is as follows. The generalized Koszul theory and its relation to the classical theory is developed in the Chapter 2. In the next two chapters we describe its application to directed categories and finite EI categories respectively. The extension algebras of standardly stratified algebras and their stratification properties and Koszul properties are described in Chapter 5. In the last chapter we classify algebras stratified for all linear orders, and study the problem of whether $\mathcal{F} (\Delta)$ is closed under cokernels of monomorphisms.

Here are the notation and conventions we use in this thesis. All algebras are $k$-algebras and $k$ is an algebraically closed field. If $A$ is a graded algebra, then it is positively graded, locally finite and generated in degrees 0 and 1. Denote the category of all graded locally finite $A$-modules by $A$-gmod. Let $M$ and $N$ be two $A$-modules. By $\Hom_A (M,N)$ and $\hom_A (M,N)$ we denote the spaces of all module homomorphisms and graded module homomorphisms (that is, the homomorphisms $\varphi \in \text{Hom}_A (M,N)$ such that $\varphi (M_i) \subseteq N_i$ for all $i \in \mathbb{Z}$) respectively. The $s$-th \textit{shift} $M[s]$ is defined in the following way: $M[s]_i = M_{i-s}$ for all $i \in \mathbb{Z}$. If $M$ is generated in degree $s$, then $\bigoplus _{i \geqslant s+1} M_i$ is a graded submodule of $M$, and $M_s \cong M/ \bigoplus _{i \geqslant s+1} M_i$ as vector spaces. We then view $M_s$ as an $A$-module by identifying it with this quotient module.

In the case that $A$ is non-graded, we always assume that $A$ is finite-dimensional and basic. By $A$-mod we denote the category of finitely generated modules. For $M \in A$-mod, by $\dim_k M$, $\pd_A M$ and $\rad M$ we mean the dimension of $M$ (as a vector space), the projective dimension of $M$ and the radical of $M$ respectively. The global dimension of $A$ is denoted by $\gldim A$.

All modules in this thesis are left modules if we do not make other assumptions. Composition of morphisms is from right to left. We view the zero module 0 as a projective (or free) module since this will simplify the expressions and proofs of many statements.


\chapter{A generalized Koszul theory}
\label{A generalized Koszul theory}

Now we begin to develop a generalized Koszul theory. Throughout this chapter $A$ is a positively graded and locally finite associative $k$-algebra with identity 1 generated in degrees 0 and 1, i.e., $A = \bigoplus _{i=0}^{\infty} A_i$ such that $A_i \cdot A_j = A_{i+j}$ for all $i, j \geqslant 0$; each $A_i$ is finite-dimensional. Define $J = \bigoplus _{i=1}^{\infty} A_i$, which is a two-sided ideal of $A$. An $A$-module $M$ is called \textit{graded} if $M = \bigoplus_{i \in \mathbb{Z}} M_i$ such that $A_i \cdot M_j \subseteq M_{i+j}$. We say $M$ is \textit{generated in degree s} if $M = A \cdot M_s$. It is clear that $M$ is generated in degree $s$ if and only if $JM \cong \bigoplus _{i \geqslant s+1} M_i$, which is equivalent to $J^lM \cong \bigoplus _{i \geqslant s+l} M_i$ for all $l \geqslant 1$.

In the first section we define generalized Koszul modules and study its properties. Generalized Koszul algebras with self-injective degree 0 parts are studied in Section 2. In the third section we prove the generalized Koszul duality. A relation between the generalized Koszul theory and the classical theory is described in the last section.

\section{Generalized Koszul modules}

Most results in this section are generalized from \cite{Green1,Green2,Martinez} and have been described in \cite{Li2, Li3}. We suggest the reader to refer to these papers.

We collect some preliminary results in the following lemma.
\begin{lemma}
Let $A$ be as above and $M$ be a locally finite graded $A$-module. Then:
\begin{enumerate}
\item $J = \bigoplus _{i \geqslant 1} A_i$ is contained in the graded radical of $A$;
\item $M$ has a graded projective cover;
\item the graded syzygy $\Omega M$ is also locally finite.
\end{enumerate}
\end{lemma}

\begin{proof}
By definition, the graded radical $\grad A$ is the intersection of all maximal proper graded submodules of $A$. Let $L \subsetneqq A$ be a maximal proper graded submodule. Then $L_0$ is a proper subspace of $A_0$. We claim that $A_i = L_i$ for all $i \geqslant 1$. Otherwise, we can define $\tilde{L} \subseteq A$ in the following way: $\tilde{L_0} = L_0$ and $\tilde{L}_i = A_i$ for $i \geqslant 1$. Then $L \subsetneqq \tilde{L} \subsetneqq A$, so $L$ is not a maximal proper graded submodule of $A$. This contradiction tells us that $L_i = A_i$ for all $i \geqslant 1$. Therefore, $J \subseteq \grad A$, and the first statement is proved.

We use the following fact to prove the second statement: every primitive idempotent in the algebra $A_0$ can be lifted to a primitive idempotent of $A$. Consequently, a projective $A_0$-module concentrated in some degree $d$ can be lifted to a graded projective $A$-module generated in degree $d$.

Define $\bar{M} = M/JM$, which is also a locally finite graded $A$-module. Write $\bar{M} = \bigoplus _{i \geqslant 0} \bar{M}_i$. Then each $\bar{M}_i$ is a finite-dimensional graded $A$-module since $J \bar{M} =0$ and $A_0 \bar{M}_i = \bar{M} _i$ for all $i \geqslant 0$. Therefore, $\bar{M}$ can be decomposed as a direct sum of indecomposable graded $A$-modules each of which is concentrated in a certain degree. Moreover, for each $i \in \mathbb{Z}$, there are only finitely many summands concentrated in degree $i$.

Take $L$ to be such an indecomposable summand and without loss of generality suppose that it is concentrated in degree 0. As an $A_0$-module, $L$ has a finitely generated projective cover $P_0$. By the lifting property, $P_0$ can be lifted to a finitely generated graded projective module $P$ generated in degree 0, which is a graded projective cover of $L$. Take the direct sum of these projective covers $P$ when $L$ ranges over all indecomposable summands of $\bar{M}$. In this way we obtain a graded projective cover $\tilde{P}$ of $\bar{M}$. Clearly, $\tilde{P}$ is also a graded projective cover of $M$. The second statement is proved.

Now we turn to the third statement. By the above proof, the graded projective cover $\tilde{P}$ of $M$ can be written as a direct sum $\bigoplus _{i \geqslant 0} P^i$ of graded projective modules, where $P^i$ is generated in degree $i$. For each fixed degree $i \geqslant 0$, there are only finitely many indecomposable summands $L$ of $\bar{M}$ concentrated in degree $i$, and the graded projective cover of each $L$ is finitely generated. Consequently, $P^i$ is finitely generated, and hence locally finite.

For a fixed $n \geqslant 0$, we have $\tilde{P}_n = \bigoplus _{i \geqslant 0} P^i_n = \bigoplus _{0 \leqslant i \leqslant n} P^i_n$. Since each $P^i$ is locally finite, $\dim_k P^i_n < \infty$. Therefore, $\dim_k \tilde{P}_n < \infty$, and $\tilde{P}$ is locally finite as well. As a submodule of $\tilde{P}$, the graded syzygy $\Omega M$ is also locally finite.
\end{proof}

These results will be used frequently.

\begin{lemma}
Let $0 \rightarrow L \rightarrow M \rightarrow N \rightarrow 0$ be an exact sequence of graded $A$-modules. Then:
\begin{enumerate}
  \item If $M$ is generated in degree $s$, so is $N$.
  \item If $L$ and $N$ are generated in degree $s$, so is $M$.
  \item If $M$ is generated in degree $s$, then $L$ is generated in degree $s$ if and only if $JM \cap L = JL$.
\end{enumerate}
\end{lemma}

\begin{proof}
(1): This is obvious.

(2): Let $P$ and $Q$ be graded projective covers of $L$ and $N$ respectively. Then $P$ and $Q$, and hence $P \oplus Q$ are generated in degree $s$. In particular, each graded projective cover of $M$, which is isomorphic to a direct summand of $P \oplus Q$, is generated in degree $s$. Thus $M$ is also generated in degree $s$.

(3): We always have $JL \subseteq JM \cap L$. Let $x \in L \cap JM$ be a homogeneous element of degree $i$. Since $M$ is generated in degree $s$, we have $i \geqslant s+1$. If $L$ is generated in degree $s$, then $x \in J^{i-s}L \subseteq JL$. Thus $L \cap JM \subseteq JL$, so $JL = L \cap JM$.

Conversely, the identity $JL = L \cap JM$ gives us the following commutative diagram where all rows and columns are exact:
\begin{align*}
\xymatrix{
& 0 \ar[d] & 0 \ar[d] & 0 \ar[d] &  \\
0 \ar[r] & JL \ar[r] \ar[d] & JM \ar[r] \ar[d] & JN \ar[r] \ar[d] & 0\\
0 \ar[r] & L \ar[r] \ar[d] & M \ar[r] \ar[d] & N \ar[r] \ar[d] & 0 \\
0 \ar[r] & L/JL \ar[r] \ar[d] & M/JM \ar[r] \ar[d] & N/JN \ar[r] \ar[d] & 0 \\
& 0 & 0 & 0 &
}
\end{align*}

Consider the bottom sequence. Since $(M / JM) \cong M_s $ is concentrated in degree $s$, $L/JL$ is also concentrated in degree $s$, i.e., $L/JL \cong L_s$. Let $I = A \cdot L_s$. Then $L \subseteq I + JL \subseteq L$, so $I + JL = L$. Note that $J$ is contained in the graded Jacobson radical of $A$. Therefore, by the graded Nakayama lemma, $I = A \cdot L_s = L$, so $L$ is generated in degree $s$.
\end{proof}

\begin{corollary}
Suppose that each graded $A$-module in the short exact sequence $0 \rightarrow L \rightarrow M \rightarrow N \rightarrow 0$ is generated in degree $0$. Then $J^iM \cap L = J^iL$ for all $i \geqslant 0$.
\end{corollary}

\begin{proof}
Since all modules $L$, $M$ and $N$ are generated in degree 0, all $J^s L$, $J^s M$ and $J^s N$ are generated in degree $s$ for $s \geqslant 0$. The exactness of the above sequence implies $JL = L \cap JM$, which in turns gives the exactness of $0 \rightarrow JL \rightarrow JM \rightarrow JN \rightarrow 0$. By the above lemma, $J^2 M \cap JL = J^2 L$ and this implies the exactness of $0 \rightarrow J^2L \rightarrow J^2M \rightarrow J^2N \rightarrow 0$. The conclusion follows from induction.
\end{proof}

Now we introduce generalized \textit{Koszul modules} (or called \textit{linear modules}).

\begin{definition}
A graded $A$-module $M$ generated in degree 0 is called a Koszul module (or a linear module) if it has a (minimal) projective resolution
\begin{displaymath}
\xymatrix{ \ldots \ar[r] & P^n \ar[r] & P^{n-1} \ar[r] & \ldots \ar[r] & P^1 \ar[r] & P^0 \ar[r] & M \ar[r] & 0}
\end{displaymath}
such that $P^i$ is generated in degree $i$ for all $i \geqslant 0$.
\end{definition}

A direct consequence of this definition and the previous lemma is:

\begin{corollary}
Let $M$ be a Koszul module. Then $\Omega^i(M) /J\Omega^i (M) \cong \Omega^i(M)_i$ is a projective $A_0$-module for each $i \geqslant 0$, or equivalently, $\Omega^i (M) \subseteq JP^{i-1}$, where $P^{i-1}$ is a graded projective cover of $\Omega^{i-1} (M)$ and $\Omega$ is the Heller operator.
\end{corollary}

\begin{proof}  Since $M$ is Koszul, $\Omega^i (M)$ is generated in degree $i$, and $\Omega^i(M) /J\Omega^i (M) \cong \Omega^i(M)_i$. Moreover, all $\Omega^i(M)[-i]$ are Koszul $A$-modules for $i \geqslant 0$. By induction, it is sufficient to prove $\Omega M \subseteq JP^0$. But this is obvious since $\Omega M$ is generated in degree 1. From the following commutative diagram we deduce that $\Omega M \subseteq JP^0$ if and only if the bottom sequence is exact, or equivalently $M/JM \cong P^0/JP^0 \cong P^0_0$ is a projective $A_0$-module.
\begin{align*}
\xymatrix{ 0 \ar[r] & \Omega M \ar@{=}[d] \ar[r] & JP^0 \ar[d] \ar[r] & JM \ar[d] \ar[r] & 0 \\
0 \ar[r] & \Omega M \ar[r] \ar[d] & P^0 \ar[r] \ar[d] & M \ar[r] \ar[d] & 0\\
0 \ar[r] & 0 \ar[r] & P^0/JP^0 \ar[r] & M/JM \ar[r] & 0 }
\end{align*}
\end{proof}

There are several characterizations of Koszul modules.

\begin{proposition}
Let $M$ be a graded $A$-module generated in degree 0. Then the following are equivalent:
\begin{enumerate}
  \item $M$ is Koszul.
  \item The syzygy $\Omega^i(M)$ is generated in degree $i$ for every $i \geqslant 0$.
  \item For all $i>0$, $\Omega^i(M) \subseteq JP^{i-1}$ and $\Omega^i(M) \cap J^2P^{i-1} = J \Omega^i(M)$, where $P^{i-1}$ is a graded projective cover of $\Omega^{i-1} (M)$.
  \item $\Omega^i(M) \subseteq JP^{i-1}$ and $\Omega^i(M) \cap J^{s+1} P^{i-1} = J^s \Omega^i(M)$ for all $i>0, s \geqslant 0$.
\end{enumerate}
\end{proposition}

\begin{proof}
The equivalence of (1) and (2) is clear. It is also obvious that (3) is the special case of (4) for $s=1$. Now we show (1) implies (4). Indeed, if $M$ is a Koszul module, then both $JP^0$ and $\Omega M$ are generated in degree 1 and $\Omega M \subseteq JP^0$. Therefore we have the following exact sequence
\begin{equation*}
\xymatrix@1 {0 \ar[r] & \Omega M \ar[r] & JP^0 \ar[r] & JM \ar[r] & 0}
\end{equation*}
in which all modules are generated in degree 1. By Corollary 2.1.5 $J^{s+1} P^0 \cap \Omega M = J^s \Omega M$ for all $s >0$. Note that all syzygies of $M$ are also Koszul with suitable grade shifts. Replacing $M$ by $\Omega^i(M)[-i]$ and using induction we get (4).

Finally we show (3) implies (2) to finish the proof. Since $\Omega M \subseteq JP^0$ we still have the above exact sequence. Notice that both $JM$ and $JP^0$ are generated in degree 1 and $J^2P^0 \cap \Omega M = J\Omega M$, by Lemma 2.1.2, $\Omega M$ is generated in degree 1 as well. Now the induction procedure gives us the required conclusion.
\end{proof}

The condition that $\Omega^i(M) \subseteq JP^{i-1}$ (or equivalently, $\Omega^i(M)/J\Omega^i(M) \cong \Omega^i (M)_i$ is a projective $A_0$-module) in (3) of the previous proposition is necessary, as shown by the following example:

\begin{example}
Let $G$ be a finite cyclic group of prime order $p$ and $k$ be an algebraically closed field of characteristic $p$. Let the group algebra $kG$ be concentrated on degree 0, so $J=0$. Consider the trivial $kG$-module $k$. Obviously, $k$ is not a Koszul module. But since $J=0$, the condition $J \Omega^i(k) = J^2P^{i-1} \cap \Omega^i(k)$ holds trivially.
\end{example}

From now on we suppose that $A_0$ has the following splitting property:\\

\textbf{(S): Every exact sequence $0 \rightarrow P \rightarrow Q \rightarrow R \rightarrow 0$ of left (resp., right) $A_0$-modules splits if $P$ and $Q$ are left (resp., right) projective $A_0$-modules.}\\

Clearly, semisimple algebras, self-injective algebras, and direct sums of local algebras satisfy this property.

\begin{proposition}
Let $0 \rightarrow L \rightarrow M \rightarrow N \rightarrow 0$ be a short exact sequence of graded $A$-modules such that $L$ is a Koszul $A$-module. Then $M$ is Koszul if and only if $N$ is Koszul.
\end{proposition}

\begin{proof}
We verify the conclusion by using statement (2) in the last proposition. That is, given that $\Omega^i(L)$ is generated in degree $i$ for each $i \geq 0$, we want to show that $\Omega^i(M)$ is generated in degree $i$ if and only if so is $\Omega^i (N)$.

Consider the following diagram in which all rows and columns are exact:
\begin{align*}
\xymatrix{
& 0 \ar[d] & 0 \ar[d] & 0 \ar[d] &  \\
0 \ar[r] & \Omega L \ar[r] \ar[d] & M' \ar[r] \ar[d] & \Omega N \ar[r] \ar[d] & 0\\
0 \ar[r] & P \ar[r] \ar[d] & P \oplus Q \ar[r] \ar[d] & Q \ar[r] \ar[d] & 0 \\
0 \ar[r] & L \ar[r] \ar[d] & M \ar[r] \ar[d] & N \ar[r] \ar[d] & 0 \\
& 0 & 0 & 0. &
}
\end{align*}
Here $P$ and $Q$ are graded projective covers of $L$ and $N$ respectively. We claim $M' \cong \Omega M$. Indeed, the given exact sequence induces an exact sequence of $A_0$-modules:
\begin{align*}
\xymatrix { 0 \ar[r] & L_0 \ar[r] & M_0 \ar[r] & N_0 \ar[r] & 0.}
\end{align*}
Observe that $L_0$ is a projective $A_0$-module. If $N$ is Koszul, then $N_0$ is a projective $A_0$-module since $N_0 \cong Q^0_0$, and the above sequence splits. If $M$ is generalized Koszul, then $M_0$ is a projective $A_0$-module, and this sequence splits as well by the splitting property (S). In either case we have $M_0 \cong L_0 \oplus N_0$. Thus $P \oplus Q$ is a graded projective cover of $M$, and hence $M' \cong \Omega M$ is generated in degree 1 if and only if $\Omega N$ is generated in degree 1 by Lemma 2.1.2. Replace $L$, $M$ and $N$ by $(\Omega L)[-1]$, $(\Omega M)[-1]$ and $(\Omega N)[-1]$ (all of them are generalized Koszul) respectively in the short exact sequence. Repeating the above procedure we prove the conclusion by recursion.
\end{proof}

The condition that $L$ is Koszul in this proposition is necessary. Indeed, quotient modules of a Koszul module might not be Koszul.

\begin{proposition}
Let $A$ be a Koszul algebra and $M$ be a generalized Koszul module. Then $J^i M[-i]$ is also Koszul for each $i \geqslant 1$.
\end{proposition}

\begin{proof}
Consider the following commutative diagram:
\begin{align*}
\xymatrix{
& 0 \ar[r] & \Omega M \ar[r] \ar[d] & \Omega(M_0) \ar[r] \ar[d] & JM \ar[r] & 0 \\
& 0 \ar[r] & P^0 \ar[r]^{id} \ar[d] & P^0 \ar[r] \ar[d] & 0 & \\
0 \ar[r] & JM \ar[r] & M \ar[r] & M_0 \ar[r] & 0 &}
\end{align*}
Since $M_0$ is a projective $A_0$-module and $A_0$ is Koszul, $\Omega(M_0)[-1]$ is also Koszul. Similarly, $\Omega M[-1]$ is Koszul since so is $M$. Therefore, $JM[-1]$ is Koszul by the previous proposition. Now replacing $M$ by $JM[-1]$ and using recursion, we conclude that $J^iM[-i]$ is a Koszul $A$-module for every $i \geqslant 1$.
\end{proof}

The condition that $A$ is a Koszul algebra cannot be dropped, as shown by the following example.

\begin{example}
Let $A$ be the algebra with relations $\alpha \delta = \rho \alpha = \delta^2 = \rho^2 =0$. Put all endomorphisms in degree 0 and all non-endomorphisms in degree 1.
\begin{equation*}
\xymatrix {x \ar@(lu, ld)[]|{\delta} \ar[r]^{\alpha} & y \ar@(ru, rd)[]|{\rho}}
\end{equation*}
The indecomposable projective modules are described as:
\begin{equation*}
P_x = \begin{matrix} & x_0 & \\ x_0 & & y_1 \end{matrix} \qquad P_y = \begin{matrix} y_0 \\ y_0 \end{matrix}.
\end{equation*}
Clearly, $P_x$ is a Koszul module. But $JP_x [-1] \cong S_y$, the simple top of $P_y$, is not Koszul.
\end{example}

From this proposition we deduce that if $A$ is a Koszul algebra, then it is a projective $A_0$-module. Indeed, let $M = A$ in this proposition we get $J^i[-i]$ is Koszul for every $i \geqslant 0$. Therefore, $A_i \cong J^i [-i]_i$ is a projective $A_0$-module.

\begin{lemma}
Let $M$ be a graded $A$-module generated in degree $s$. If $M_s$ is a projective $A_0$-module, then $\Ext _A^i (M, A_0) \cong \Ext _A^{i-1} (\Omega M, A_0)$ for all $i \geqslant 1$.
\end{lemma}

\begin{proof}
It is true for $i >1$. When $i=1$, consider the following exact sequence:
\begin{equation*}
0 \rightarrow \text{Hom}y _A (M, A_0) \rightarrow \Hom _A (P, A_0) \rightarrow \Hom _A (\Omega M, A_0) \rightarrow \Ext^1 _A (M, A_0) \rightarrow 0.
\end{equation*}
As a graded projective cover of $M$, $P$ is also generated in degree $s$. Since $M_s$ is a projective $A_0$-module, $P_s \cong M_s$. So
\begin{equation*}
\Hom _A (M, A_0) \cong \Hom _{A_0} (M_s, A_0) \cong \Hom _{A_0} (P_s, A_0) \cong \Hom _A (P, A_0).
\end{equation*}
Thus $\Hom _A (\Omega M, A_0) \cong \Ext^1 _A (M, A_0)$.
\end{proof}

The above lemma holds for all finite-dimensional algebras $A_0$, no matter they satisfy the splitting property (S) or not.

The following lemma is also useful.

\begin{lemma}
Let $M$ be a positively graded $A$-module and suppose that $A$ is a projective $A_0$-module. Then the following are equivalent:
\begin{enumerate}
\item all $\Omega^i(M)$ are projective $A_0$-modules, $i \geqslant 0$;
\item all $\Omega^i(M)_i$ are projective $A_0$-modules, $i \geqslant 0$;
\item $M$ is a projective $A_0$-module.
\end{enumerate}
\end{lemma}

\begin{proof}
It is clear that (1) implies (2).

(3) implies (1): Consider the exact sequence $0 \rightarrow \Omega M \rightarrow P \rightarrow M \rightarrow 0$. View it as a short exact sequence of $A_0$-modules. Since $M$ and $A$ are projective $A_0$-modules, so are all $\Omega M$. The conclusion follows from induction.

(2) implies (3): Conversely, suppose that $\Omega^i(M)_i$ is a projective $A_0$-modules for every $i \geqslant 0$. We use contradiction to show that $M$ is a projective $A_0$-module. If this not the case, we can find the minimal number $n \geqslant 0$ such that $M_n$ is not a projective $A_0$-module. Consider the short exact sequence $0 \rightarrow (\Omega M)_n \rightarrow P_n \rightarrow M_n \rightarrow 0$. We know that $(\Omega M)_n$ is not a projective $A_0$-module since otherwise the splitting property (S) of $A_0$ forces this sequence splits and hence $M_n$ is a projective $A_0$-module, which is impossible. Replacing $M$ by $\Omega M$ and using induction, we deduce that $\Omega^n (M)_n$ is not a projective $A_0$-module. This contradicts our assumption. Therefore, $M_i$ are projective $A_0$-modules for all $i \geqslant 0$.
\end{proof}

An immediate corollary of this lemma is:

\begin{corollary}
If $A$ is a projective $A_0$-module, then every Koszul module $M$ is a projective $A_0$-module.
\end{corollary}

\begin{proof}
Since $M$ is a Koszul $A$-module, then for each $i \geqslant 0$, $\Omega^i (M)_i$ is a projective $A_0$-module by Corollary 2.1.5. This is true even $A_0$ does not satisfy the splitting property (S). The conclusion then follows from the equivalence of (2) and (3) in the above lemma.
\end{proof}

In particular, if $A$ is a Koszul algebra, then by Proposition 2.1.9 it is a projective $A_0$-module. Therefore, every Koszul $A$-module is a projective $A_0$-module.

We remind the reader that the condition that $A$ is a projective $A_0$-module cannot be dropped in this proposition. Indeed, consider the indecomposable projective module $P_x$ in Example 2.10. Clearly, $P_x$ is a Koszul module, but it is not a projective $A_0$-module.

\begin{proposition}
Suppose that $A$ is a projective $A_0$-module. Let $0 \rightarrow L \rightarrow M \rightarrow N \rightarrow 0$ be an exact sequence of Koszul modules. Then it induces the following short exact sequence:
\begin{equation*}
\xymatrix{ 0 \ar[r] & \Ext ^{\ast} _A (N, A_0) \ar[r] & \Ext ^{\ast} _A (M, A_0) \ar[r] & \Ext ^{\ast} _A (L, A_0) \ar[r] & 0.}
\end{equation*}
\end{proposition}

\begin{proof}
As in the proof of Proposition 2.1.8, the above exact sequence gives exact sequences $0 \rightarrow \Omega^i (L) \rightarrow \Omega^i (M) \rightarrow \Omega^i (N) \rightarrow 0$, $i \geqslant 0$. For a fixed $i$, the sequence $ 0 \rightarrow \Omega^i (L)_i \rightarrow \Omega^i(M)_i \rightarrow \Omega^i (N)_i \rightarrow 0$ splits since all terms are projective $A_0$-modules by Corollary 2.1.5. Applying the functor $\Hom _{A_0} (-, A_0)$ we get an exact sequence
\begin{equation*}
0 \rightarrow \Hom _{A_0} (\Omega^i(N)_i, A_0) \rightarrow \Hom _{A_0} (\Omega^i(M)_i, A_0) \rightarrow \Hom _{A_0} (\Omega^i(L)_i, A_0) \rightarrow 0
\end{equation*}
which is isomorphic to
\begin{equation*}
0 \rightarrow \Hom _A (\Omega^i(N), A_0) \rightarrow \Hom _A (\Omega^i(M), A_0) \rightarrow \Hom _A (\Omega^i(L), A_0) \rightarrow 0
\end{equation*}
since all modules are generated in degree $i$. By Lemma 2.1.11, it is isomorphic to
\begin{equation*}
\xymatrix{ 0 \ar[r] & \Ext^i _A (N, A_0) \ar[r] & \Ext^i _A (M, A_0) \ar[r] & \Ext^i _A (L, A_0) \ar[r] & 0.}
\end{equation*}
Putting them together we have:
\begin{equation*}
\xymatrix{ 0 \ar[r] & \Ext ^{\ast} _A (N, A_0) \ar[r] & \Ext ^{\ast} _A (M, A_0) \ar[r] & \Ext ^{\ast} _A (L, A_0) \ar[r] & 0.}
\end{equation*}
\end{proof}

Now we define \textit{quasi-Koszul modules} over the graded algebra $A$.

\begin{definition}
A positively graded $A$-module $M$ is called quasi-Koszul if
\begin{equation*}
\Ext _A^1 (A_0, A_0) \cdot \Ext _A^i (M, A_0)=  \Ext _A^{i+1} (M, A_0)
\end{equation*}
for all $i \geq 0$. The algebra $A$ is called a quasi-Koszul algebra if $A_0$ as an $A$-module is quasi-Koszul.
\end{definition}

Clearly, a graded $A$-module $M$ is quasi-Koszul if and only if as a graded $\Ext _A^{\ast} (A_0, A_0)$-module $\Ext _A ^{\ast} (M, A_0)$ is generated in degree $0$. The graded algebra $A$ is a quasi-Koszul algebra if and only if the cohomology ring $\Ext^{\ast}_A (A_0, A_0)$ is generated in degrees 0 and 1.

The quasi-Koszul property is preserved by the Heller operator. Explicitly, if $M$ is a quasi-Koszul $A$-module with $M_0$ a projective $A_0$-module, then its syzygy $\Omega M$ is also quasi-Koszul. This is because for each $i \geqslant 1$, we have:
\begin{align*}
& \Ext _A^{i} (\Omega M, A_0) \cong \Ext _A^{i+1} (M, A_0) \\
& = \Ext _A^1 (A_0, A_0) \cdot \Ext _A^i (M, A_0)\\
& = \Ext _A^1 (A_0, A_0) \cdot \Ext _A^{i-1} (\Omega M, A_0).
\end{align*}
The identity $\Ext _A^i (M, A_0) \cong \Ext _A^{i-1} (\Omega M, A_0)$ is proved in Lemma 2.1.11.

If $A_0$ is a semisimple $k$-algebra, quasi-Koszul modules generated in degree 0 coincide with Koszul modules. This is not true if $A_0$ only satisfy the splitting property (S). Actually, by the following theorem, every Koszul module is quasi-Koszul, but the converse does not hold in general. For example, let $kG$ be the group algebra of a finite group concentrated in degree 0. The reader can check that every $kG$-module generated in degree 0 is quasi-Koszul, but only the projective $kG$-modules are Koszul. If the order $|G|$ is not invertible in $k$, then all non-projective $kG$-modules generated in degree 0 are quasi-Koszul but not Koszul.

The following theorem gives us a close relation between quasi-Koszul modules and Koszul modules.

\begin{theorem}
Suppose that $A$ is a projective $A_0$-module. Then a graded $A$-module $M$ generated in degree 0 is Koszul if and only if it is quasi-Koszul and a projective $A_0$-module.
\end{theorem}

The following lemma will be used in the proof of this theorem.

\begin{lemma}
Let $M$ be a graded $A$-module generated in degree 0. Suppose that both $A$ and $M$ are projective $A_0$-modules. Then $\Omega M$ is generated in degree 1 if and only if every $A$-module homomorphism $\Omega M \rightarrow A_0$ extends to an $A$-module homomorphism $JP \rightarrow A_0$, where $P$ is a graded projective cover of $M$.
\end{lemma}

\begin{proof}
The exact sequence $0 \rightarrow \Omega M \rightarrow P \rightarrow M \rightarrow 0$ induces an exact sequence $
0 \rightarrow (\Omega M)_1 \rightarrow P_1 \rightarrow M_1 \rightarrow 0$ of $A_0$-modules, which splits since $M_1$ is a projective $A_0$-module. Applying the functor $\Hom_{A_0} (-, A_0)$ we get another splitting exact sequence
\begin{equation*}
0 \rightarrow \Hom _{A_0} (M_1, A_0) \rightarrow \Hom _{A_0} (P_1, A_0) \rightarrow \Hom _{A_0} ((\Omega M)_1, A_0) \rightarrow 0.
\end{equation*}
Note that $(\Omega M)_0 =0$. Thus $\Omega M$ is generated in degree 1 if and only if $\Omega M / J (\Omega M) \cong (\Omega M)_1$, if and only if the above sequence is isomorphic to
\begin{equation*}
0 \rightarrow \Hom _{A_0} (M_1, A_0) \rightarrow \Hom _{A_0} (P_1, A_0) \rightarrow \Hom _{A_0} (\Omega M / J \Omega M, A_0) \rightarrow 0.
\end{equation*}
Here we use the fact that $M_1$, $P_1$ and $(\Omega M)_1$ are projective $A_0$-modules. But the above sequence is isomorphic to
\begin{equation*}
0 \rightarrow \Hom _A (JM, A_0) \rightarrow \Hom _A (JP, A_0) \rightarrow \Hom _A (\Omega M, A_0) \rightarrow 0
\end{equation*}
since $JM$ and $JP$ are generated in degree 1. Therefore, $\Omega M$ is generated in degree 1 if and only if every (non-graded) $A$-module homomorphism $\Omega M \rightarrow A_0$ extends to a (non-graded) $A$-module homomorphism $JP \rightarrow A_0$.
\end{proof}

Now let us prove the theorem.

\begin{proof}
\textbf{The only if part.} Let $M$ be a Koszul $A$-module. Without loss of generality we can suppose that $M$ is indecomposable. Clearly it is a projective $A_0$-module. Thus we only need to show that $M$ is quasi-Koszul, i.e.,
\begin{equation*}
\Ext _A^{i+1} (M, A_0) = \Ext_A^1 (A_0, A_0) \cdot \Ext _A^i (M, A_0)
\end{equation*}
for all $i \geqslant 0$. By Lemma 2.1.11, we have $\Ext _A^{i+1} (M, A_0) \cong \Ext_A^1 (\Omega^i (M), A_0)$ and $\Ext _A^i (M, A_0) \cong \Hom_A (\Omega^i(M), A_0)$. Therefore, it suffices to show $\Ext _A^1 (M, A_0) = \Ext _A^1 (A_0, A_0) \cdot \Hom _A (M, A_0)$ since the conclusion follows if we replace $M$ by $(\Omega M) [-1]$ and use recursion.

To prove this identity, we first identify $\Ext _A^1 (M, A_0)$ with $\Hom _A (\Omega M, A_0)$ by Lemma 2.1.11. Take an element $x \in \Ext _A^1 (M, A_0)$ and let $g: \Omega M \rightarrow A_0$ be the corresponding homomorphism. Since $M$ is Koszul, it is a projective $A_0$-module, and $\Omega M$ is generated in degree 1. Thus by the previous lemma, $g$ extends to $JP^0$, and hence there is a homomorphism $\tilde{g}: JP^0 \rightarrow A_0$ such that $g = \tilde{g} \iota$, where $P^0$ is a graded projective cover of $M$ and $\iota: \Omega M \rightarrow JP^0$ is the inclusion.
\begin{align*}
\xymatrix { \Omega M \ar[r]^{\iota} \ar[d]^g & JP^0 \ar[dl]^{\tilde{g}}\\
A_0 & }
\end{align*}

We have the following commutative diagram:
\begin{align*}
\xymatrix{ 0 \ar[r] & \Omega M \ar[d]^{\iota} \ar[r] & P^0 \ar@{=}[d] \ar[r] & M \ar[d]^p \ar[r] & 0\\
0 \ar[r] & JP^0 \ar[r] & P^0 \ar[r] & P^0_0 \ar[r] & 0
}
\end{align*}
where the map $p$ is defined to be the projection of $M$ onto $M_0 \cong P_0^0$.

The map $\tilde{g}: JP^0 \rightarrow A_0$ gives a push-out of the bottom sequence. Consequently, we have the following commutative diagram:
\begin{align*}
\xymatrix{ 0 \ar[r] & \Omega M \ar[d]^{\iota} \ar[r] & P^0 \ar@{=}[d] \ar[r] & M \ar[d]^p \ar[r] & 0\\
0 \ar[r] & JP^0 \ar[r] \ar[d] ^{\tilde{g}} & P^0 \ar[d] \ar[r] & P^0_0 \ar[r] \ar@{=}[d] & 0\\
0 \ar[r] & A_0 \ar[r] & E \ar[r] & P^0_0 \ar[r] & 0.}
\end{align*}

Since $P_0 \in \text{add} (A_0)$, we can find some $m$ such $P_0$ can be embedded into $A_0 ^{\oplus m}$. Thus the bottom sequence $y \in \Ext _A^1 (P_0, A_0) \subseteq \bigoplus _{i=1}^m \Ext _A^1 (A_0, A_0)$ and we can write $y = y_1 + \ldots + y_m$ where $y_i \in \Ext _A^1 (A_0, A_0)$ is represented by the sequence
\begin{equation*}
\xymatrix{ 0 \ar[r] & A_0 \ar[r] & E_i \ar[r] & A_0 \ar[r] & 0}.
\end{equation*}
Composed with the inclusion $\epsilon: P_0 \rightarrow A_0 ^{\oplus m}$, the map $\epsilon \circ p = (p_1, \ldots, p_m)$ where each component $p_i$ is defined in an obvious way. Consider the pull-backs:
\begin{equation*}
\xymatrix {0 \ar[r] & A_0 \ar[r] \ar@{=}[d] & F_i \ar[r] \ar[d] & M \ar[r] \ar[d]^{p_i} & 0 \\
0 \ar[r] & A_0 \ar[r] & E_i \ar[r] & A_0 \ar[r] & 0.}
\end{equation*}
Let $x_i$ be the top sequence. Then $x = \sum_{i=1}^m x_i = \sum_{i=1}^m y_i p_i \in \Ext _A^1 (A_0, A_0) \cdot \Hom _A (M, A_0)$ and hence $\Ext _A^1 (M, A_0) \subseteq \Ext _A^1 (A_0, A_0) \cdot \Hom _A (M, A_0)$. The other inclusion is obvious.\\

\textbf{The if part.} By Proposition 2.1.6, it suffices to show that $\Omega^i (M)$ is generated in degree $i$ for each $i>0$. But we observe that if $M$ is quasi-Koszul and a projective $A_0$-modules, then each $\Omega^i(M)$ has these properties as well (see lemma 2.1.12). Thus we only need to show that $\Omega M$ is generated in degree 1 since the conclusion follows if we replace $M$ by $\Omega M$ and use recursion. By the previous lemma, it suffices to show that each (non-graded) $A$-module homomorphism $g: \Omega M \rightarrow A_0$ extends to $JP^0$.

The map $g$ gives a push-out $x \in \Ext _A^1 (M, A_0)$ as follows:
\begin{align*}
\xymatrix{ 0 \ar[r] & \Omega M \ar[d]^g \ar[r] & P^0 \ar[d] \ar[r] & M \ar@{=}[d] \ar[r] & 0\\
0 \ar[r] & A_0 \ar[r] & E \ar[r] & M \ar[r] & 0
}
\end{align*}
Since $M$ is quasi-Koszul, $x$ is contained in $\Ext _A^1 (A_0, A_0) \cdot \Hom _A (M, A_0)$. Thus $x = \sum_{i} y_i h_i$ with $y_i \in \Ext_A^1 (A_0, A_0)$ and $h_i \in \Hom_A (M, A_0)$, and each $y_i h_i$ gives the following commutative diagram, where the bottom sequence corresponds to $y_i$:
\begin{align}
\xymatrix{ 0 \ar[r] & A_0 \ar@{=}[d] \ar[r] & E_i \ar[d] \ar[r] & M \ar[d]^{h_i} \ar[r] & 0\\
0 \ar[r] & A_0 \ar[r] & F_i \ar[r] & A_0 \ar[r] & 0
}
\end{align}
By the natural isomorphism $\Ext _A^1 (M, A_0) \cong \Hom_A (\Omega M, A_0)$ (see Lemma 2.1.11), each $y_ih_i$ corresponds an $A$-homomorphism $g_i: \Omega M \rightarrow A_0$ such that the following diagram commutes:
\begin{align}
\xymatrix{ 0 \ar[r] & \Omega M \ar[d]^{g_i} \ar[r] & P^0 \ar[d] \ar[r] & M \ar@{=}[d] \ar[r] & 0\\
0 \ar[r] & A_0 \ar[r] & E_i \ar[r] & M \ar[r] & 0
}
\end{align}

Diagrams (2.1.2) and (2.1.3) give us:
\begin{align*}
\xymatrix{ 0 \ar[r] & \Omega M \ar@{=}[d] \ar[r]^{\iota} & JP^0 \ar[d]^{\tilde{j}} \ar[r] & JM \ar[d]^j \ar[r] & 0\\
0 \ar[r] & \Omega M \ar[r] \ar[d]^{g_i} & P^0 \ar[d]^{\tilde{h_i}} \ar[r] & M \ar[r] \ar[d]^{h_i} & 0\\
0 \ar[r] & A_0 \ar[r]^{\rho} & F_i \ar[r] & A_0 \ar[r] & 0
}
\end{align*}
Since $JM$ is sent to 0 by $h_ij$, there is a homomorphism $\varphi_i$ from $JP_0$ to the first term $A_0$ of the bottom sequence such that $\rho \varphi_i = \tilde{h_i} \tilde{j}$. Then $g_i$ factors through $\varphi_i$, i.e., $g_i = \varphi_i \iota$. Since $g = \sum_i g_i$, we know that $g$ extends to $JP^0$. This finishes the proof.
\end{proof}

An easy corollary of the above theorem is:

\begin{corollary}
Suppose that $A$ is a Koszul algebra. Then a graded $A$-module $M$ is Koszul if and only if it is quasi-Koszul and a projective $A_0$-module
\end{corollary}

\begin{proof}
This is clear since if $A$ is a Koszul algebra, then it is a projective $A_0$-module.
\end{proof}

\section{Generalized Koszul algebras}
In this section we generalize to our context some useful results on classical Koszul algebras appearing in \cite{BGS}. As before, $A$ is a positively graded, locally finite associative $k$-algebra. We suppose that $A_0$ is a \textbf{self-injective} algebra, so the splitting property (S) is satisfied. For two graded $A$-modules $M$ and $N$, we use $\Hom _A (M, N)$ and $\hom _A (M, N)$ to denote the space of all module homomorphisms and the space of graded module homomorphisms respectively. The derived functors Ext and ext correspond to Hom and hom respectively.

Recall that $A$ a \textit{quasi-Koszul algebra} if $A_0$ is quasi-Koszul as an $A$-module. In particular, if $A_0$ is a Koszul $A$-module, then $A$ is a quasi-Koszul algebra.

\begin{theorem}
The graded algebra $A$ is quasi-Koszul if and only if the opposite algebra $A^{\textnormal{op}}$ is quasi-Koszul.
\end{theorem}

\begin{proof}
Since the quasi-Koszul property is invariant under the Morita equivalence, without loss of generality we can suppose that $A$ is a basic algebra. Therefore, $A_0$ is also a basic algebra. Let $M$ and $N$ be two graded $A$-modules. We claim $\ext^i_{A}(M, N) \cong \ext _{A^{\textnormal{op}}}^i (DN, DM)$ for all $i \geqslant 0$, where $D$ is the graded duality functor $\hom_k (-,k)$. Indeed, Let
\begin{align*}
\xymatrix{ \ldots \ar[r] & P^2 \ar[r] & P^1 \ar[r] & P^0 \ar[r] & M \ar[r] & 0 }
\end{align*}
be a projective resolution of $M$. Applying the graded functor $\hom_A (-,N)$ we get the following chain complex $C^{\ast}$:
\begin{align*}
\xymatrix{ 0 \ar[r] & \hom_A (P^0, N) \ar[r] & \hom_A (P^1, N) \ar[r] & \ldots. }
\end{align*}
Using the natural isomorphism $\hom_A (P^i, N) \cong \hom _{A^{\textnormal{op}}} (DN, DP^i)$, we get another chain complex $E^{\ast}$ isomorphic to the above one:
\begin{align*}
\xymatrix{ 0 \ar[r] & \hom_{A^{\textnormal{op}}} (DN, DP^0) \ar[r] & \hom_{A^{\textnormal{op}}} (DN, DP^1) \ar[r] & \ldots. }
\end{align*}
Notice that all $DP^i$ are graded injective $A^{\textnormal{op}}$-modules. Thus
\begin{equation*}
\ext ^i_{A}(M, N) \cong H^i(C^{\ast}) \cong H^i(E^{\ast}) \cong \ext _{A^{\textnormal{op}}}^i (DN, DM)
\end{equation*}
which is exactly our claim.

Now let $M = N = A_0$. Then $\ext ^i_{A} (A_0, A_0) \cong \ext _{A^{\textnormal{op}}}^i (DA_0, DA_0)$. Since $A_0$ is self-injective and basic, it is a Frobenius algebra. Therefore, $DA_0$ is isomorphic to $A_0^{\textnormal{op}}$ as a left $A_0^{\textnormal{op}}$-module (and hence as a left $A^{\textnormal{op}}$-module). By the next proposition, $A_0$ is a quasi-Koszul $A$-module if and only if $A_0^{\textnormal{op}}$ is a quasi-Koszul $A^{\textnormal{op}}$-module.
\end{proof}

However, if $A_0$ is a Koszul $A$-module, $A^{\textnormal{op}}_0$ need not be a Koszul $A^{\textnormal{op}}$-module, as shown by the following example.

\begin{example}
Let $A$ be the algebra with relations $\alpha \delta = \delta^2 =0$. Put all endomorphisms in degree 0 and all non-endomorphisms in degree 1.
\begin{equation*}
\xymatrix {x \ar@(lu, ld)[]|{\delta} \ar[r]^{\alpha} & y}
\end{equation*}
The indecomposable projective modules are described as:
\begin{equation*}
P_x = \begin{matrix} & x_0 & \\ x_0 & & y_1 \end{matrix} \qquad P_y = y_0.
\end{equation*}
The reader can check that $A$ is a Koszul algebra, but $A^{\textnormal{op}}$ is not. This is because $A^{\textnormal{op}}$ is not a projective $A_0^{\textnormal{op}} \cong A_0$-module, so it cannot be a Koszul algebra.
\end{example}

\begin{proposition}
The graded algebra $A$ is Koszul if and only if $A$ is a projective $A_0$-module, and whenever ext$^i _A (A_0, A_0 [n]) \neq 0$ we have $n=i$.
\end{proposition}

\begin{proof}
If $A$ is a Koszul algebra, it is clearly a projective $A_0$-module. Moreover, there is a linear projective resolution
\begin{equation*}
\xymatrix { \ldots \ar[r] & P^2 \ar[r] & P^1 \ar[r] & P^0 \ar[r] & A_0 \ar[r] & 0}
\end{equation*}
with $P^i$ being generated in degree $i$. Applying $\hom_A (-, A_0 [n])$ we find that all terms in this complex except $\hom _A (P^{n}, A_0 [n])$ are 0. Consequently, ext$^i _A (A_0, A [n]) \neq 0$ unless $i=n$.

Conversely, suppose that $A$ is a projective $A_0$-module and ext$^i _A (A_0, A_0 [n]) = 0$ unless $n=i$, we want to show that $\Omega^{i} (A_0)$ is generated in degree $i$ by induction. Obviously, $\Omega^0 (A_0) = A_0$ is generated in degree 0. Suppose that $\Omega^j(A_0)$ is generated in degree $j$ for $0 \leqslant j \leqslant i$. Now consider $\Omega^{i+1} (A_0)$. By applying the graded version of Lemma 2.1.11 recursively, we have
\begin{equation*}
\text{hom} _A (\Omega^{i+1} (A_0), A_0 [n]) = \text{ext} _A^{i+1} (A_0, A_0 [n]).
\end{equation*}
The right-hand side is 0 unless $n = i + 1$, so $\Omega^{i+1} (A_0)$ is generated in degree $i+1$. By induction we are done.
\end{proof}

The reader can check that the conclusion of this proposition is also true for Koszul modules. i.e., $M$ is a Koszul $A$-module if and only if $\text{ext} _A ^i (M, A_0[n]) \neq 0$ implies $n = i$.

Define $A_0[A_1]$ to be the tensor algebra generated by $A_0$ and the $(A_0, A_0)$-bimodule $A_1$. Explicitly,
\begin{equation*}
A_0[A_1] = A_0 \oplus A_1 \oplus (A_1 \otimes A_1) \oplus (A_1 \otimes A_1 \otimes A_1) \oplus \ldots,
\end{equation*}
where all tensors are over $A_0$ and we use $\otimes$ rather than $\otimes _{A_0}$ to simplify the notation. This tensor algebra has a natural grading. Clearly, $A$ is a quotient algebra of $A_0[A_1]$. Let $R$ be the kernel of the quotient map $q: A_0[A_1] \rightarrow A$. We say that $A$ is a \textit{quadratic algebra} if the ideal $R$ has a set of generators contained in $A_1 \otimes A_1$.

\begin{theorem}
If $A$ is a Koszul algebra, then it is a quadratic algebra.
\end{theorem}

\begin{proof}
This proof is a modification of the proofs of Theorem 2.3.2 and Corollary 2.3.3 in \cite{BGS}. First, consider the exact sequence
\begin{equation*}
\xymatrix{ 0 \ar[r] & W \ar[r] & A \otimes A_1 \ar[r] & A \ar[r] & A_0 \ar[r] & 0}
\end{equation*}
where $W$ is the kernel of the multiplication. Clearly, $\Omega(A_0) \cong J = \bigoplus _{i \geqslant 1} A_i$. Since the image of $(A \otimes A_1)_1 = A_0 \otimes A_1$ under the multiplication is exactly $A_1 = \Omega(A_0)_1$, $A \otimes A_1$ is a projective cover of $\Omega(A_0)$ and $\Omega^2 (A_0) = W \subseteq J \otimes A_1$. Therefore, $W$ is generated in degree 2, and hence $W/JW \cong W_2$ is concentrated in degree 2. Observe that $A$ is a quotient algebra of $A_0[A_1]$ with kernel $R$. Let $R_n$ be the kernel of the quotient map $A_1 ^{\otimes n} \rightarrow A_n$.

If $A$ is not quadratic, we can find some $x \in R_n$ with $n > 2$ such that $x$ is not contained in the two-sided ideal generated by $\sum _{i=2}^{n-1} R_i$. Consider the following composite of maps:
\begin{equation*}
\xymatrix {A_1 ^{\otimes n} = A_1 ^{\otimes n-1} \otimes A_1 \ar[r] ^-p & A_{n-1} \otimes A_1 \ar[r]^-m & A_n}.
\end{equation*}
Clearly $p(x) \in W$ since $m(p(x)) =0$. We show $p(x) \notin JW$ by contradiction.

Indeed, if $p(x) \in JW$, then $p(x) \in A_1 W$ since $JW \cong \bigoplus _{i \geqslant 3} W_i = A_1 W$ (notice that $W$ is generated in degree 2). Therefore, we can express $p(x)$ as a linear combination of vectors of the form $\lambda \cdot w$ with $\lambda \in A_1$ and $w \in W$. But $W \subseteq J \otimes A_1$, so each $w$ can be expressed as $\sum_i w_i' \otimes \lambda_i'$ with $w_i' \in A_{n-2}$, $\lambda_i' \in A_1$ such that $\sum_{i=1}^s w'_i \cdot \lambda'_i =0$ by the definition of $W$.

Since there is a surjective product map $\varphi: A_1 ^{\otimes n-2} \twoheadrightarrow A_{n-2}$, we can choose a pre-image $v_i^1 \otimes \ldots \otimes v_i^{n-2} \in \varphi^{-1} (w_i')$ for each $i$ and define
\begin{equation*}
\tilde{w} = \sum_{i=1}^s v_i^1 \otimes \ldots \otimes v_i^{n-2} \otimes \lambda'_i
\end{equation*}
which is contained in $R_{n-1}$ clearly. Observe that $p(\lambda \otimes \tilde{w}) = \lambda \cdot w$. Since $p(x)$ is a linear combination of vectors of the form $\lambda \cdot w$, by the above process we can get some $y$ which is a linear combination of vectors of the form $\lambda \otimes \tilde{w}$ such that $p(y) = p(x)$. Clearly, $p(x -y) = 0$ and $y \in A_1 \otimes R_{n-1}$.

Consider the following short exact sequence
\begin{equation*}
\xymatrix{ 0 \ar[r] & R_{n-1} \ar[r] & A_1 ^{\otimes n-1} \ar[r] & A_{n-1} \ar[r] & 0.}
\end{equation*}
Since $A_0$ is a Koszul $A$-module, $A$ is a projective $A_0$-module. Thus $A_1$ is a projective $A_0$-module as well. Therefore, the following sequence is also exact:
\begin{equation*}
\xymatrix{ 0 \ar[r] & R_{n-1} \otimes A_1 \ar[r] & A_1 ^{\otimes n} \ar[r]^-p & A_{n-1} \otimes A_1 \ar[r] & 0.}
\end{equation*}
Thus $x - y \in R_{n-1} \otimes A_1$ since $p(x-y) =0$. It follows $x \in A_1 \otimes R_{n-1} + R_{n-1} \otimes A_1$, contradicting our choice of $x$.

We proved $x \notin JW$. Then $p(x) \in W/JW \cong W_2$ is of degree 2. But this is impossible since $p$ as a graded homomorphism sends $x \in R_n$ with $n >2$ to an element of degree $n$.
\end{proof}

We can define the \textit{Koszul complex} for $A$ in a way similar to the classical situation.

Let $A \cong A_0[A_1] /(R)$ be quadratic with $R \subseteq A_1 \otimes A_1$ being a set of relations. Define $P_n^n = \bigcap _{i=0} ^{n-2} A_1 ^{\otimes i} \otimes R \otimes A_1 ^{\otimes n-i-2} \subseteq A_1 ^{\otimes n}$. In particular, $P_0^0 = A_0$, $P_1^1 = A_1$ and $P_2^2 = R$. Let $P^n  = A \otimes P_n^n$ such that $A_0 \otimes P_n^n \cong P_n^n$ is in degree $n$. Define $d^n: P^n \rightarrow P^{n-1}$ to be the restriction of $A \otimes A_1 ^{\otimes n} \rightarrow A \otimes A_1 ^{\otimes n-1}$ by $a \otimes v_1 \otimes \ldots \otimes v_n \mapsto av_1 \otimes v_2 \otimes \ldots \otimes v_n$. The reader can check $d^{n-1} d^{n} =0$ for $n \geqslant 1$. Therefore we get the following Koszul complex $K^{\ast}$:
\begin{equation*}
\xymatrix {\ldots \ar[r] & P^3 \ar[r]^{d^3} & A \otimes R \ar[r]^{d^2} & A \otimes A_1 \ar[r]^{d^1} & A \ar[r] & 0.}
\end{equation*}

\begin{theorem}
Let $A \cong A_0[A_1] /(R)$ be a quadratic algebra. Then $A$ is a Koszul algebra if and only if the Koszul complex is a projective resolution of $A_0$.
\end{theorem}

\begin{proof}
One direction is trivial. Now suppose that $A_0$ is a Koszul $A$-module. The Koszul complex $K^{\ast}$ of $A$ has the following properties:

\noindent (1). Let $Z^n$ be the kernel of $d^n: P^n \rightarrow P^{n-1}$. The restricted map $d_n^n:$
\begin{equation*}
P^n_n = A_0 \otimes \big{(}\bigcap _{i=0} ^{n-2} A_1 ^{\otimes i} \otimes R \otimes A_1 ^{\otimes n-i-2} \big{)} \rightarrow P^{n-1} _n = A_1 \otimes (\bigcap _{i=0} ^{n-3} A_1 ^{\otimes i} \otimes R \otimes A_1 ^{\otimes n-i-3})
\end{equation*}
is injective. Therefore $Z^n_i = 0$ for every $i \leqslant n$.

\noindent (2). $Z_{n+1}^n$, the kernel of the map $d_n^{n+1}:$
\begin{equation*}
P^n _{n+1} = A_1 \otimes \big{(} \bigcap_{i=0} ^{n-2} A_1^{\otimes i} \otimes R \otimes A_1 ^{\otimes n-i -2} \big{)} \rightarrow P^{n-1} _{n+1} = A_2 \otimes \big{(} \bigcap_{i=0} ^{n-3} A_1^{\otimes i} \otimes R \otimes A_1 ^{\otimes n-i -3} \big{)}
\end{equation*}
is
\begin{equation*}
A_1 \otimes \big{(} \bigcap_{i=0} ^{n-2} A_1^{\otimes i} \otimes R \otimes A_1 ^{\otimes n-i -2} \big{)} \cap (R \otimes A_1 ^{\otimes n-1}) = \bigcap_{i=0} ^{n-1} A_1^{\otimes i} \otimes R \otimes A_1 ^{\otimes n-i -1}
\end{equation*}
which is exactly $P_{n+1} ^{n+1}$ (or $d^{n+1}_{n+1} (P_{n+1} ^{n+1})$ since $d^{n+1}_{n+1}$ is injective by the last property).

We claim that each $P^n = A \otimes P^n_n$ is a projective $A$-module. Clearly, it is enough to show that each $P_n^n = Z_n ^{n-1}$ is a projective $A_0$-module. We prove the following stronger conclusion. That is, $Z^n_i$ are projective $A_0$-modules for $i \in \mathbb{Z}$ and $n \geqslant 0$. We use induction on $n$.

Since $A_0$ is a Koszul $A$-module, $A_i$ are projective $A_0$-modules for all $i \geqslant 0$. The conclusion is true for $Z^0 \cong J$ since $J_0 = 0$ and $J_m = A_m$ for $m \geqslant 1$. Suppose that it is true for $l \leqslant n$. That is, all $Z^l _i$ are projective $A_0$-modules for $l \leqslant n$ and $i \in \mathbb{Z}$. Consider $l = n+1$. By the second property described above, $P_{n+1} ^{n+1} = Z_{n+1}^n$, which is a projective $A_0$-module by the induction hypothesis. Therefore, $P^{n+1} = A \otimes P^{n+1}_{n+1}$ is a projective $A$-module, so $P^{n+1} _i$ are all projective $A_0$-modules for $i \in \mathbb{Z}$. But the following short exact sequence of $A_0$-modules splits
\begin{equation*}
\xymatrix{ 0 \ar[r] & Z^{n+1} _i \ar[r] & P^{n+1} _i \ar[r] & Z^n_i \ar[r] & 0}
\end{equation*}
since $Z^n_i$ is a projective $A_0$-module by the induction hypothesis. Now as a direct summand of $P^{n+1}_i$ which is a projective $A_0$-module, $Z^{n+1} _i$ is a projective $A_0$-module as well. Our claim is proved by induction.

We claim that this complex is acyclic. First, the sequence
\begin{equation*}
\xymatrix{ P^1 = A \otimes A_1 \ar[r] & P^0 = A \otimes A_0 \ar[r] & A_0 \ar[r] & 0}
\end{equation*}
is right exact. By induction on $n>1$
\begin{equation*}
\text{ext} _A ^{n+1} (A_0, A_0 [m]) = \text{coker} \big{(} \text{hom} _A (P^n, A_0 [m]) \rightarrow \text{hom} _A (Z^n, A_0 [m]) \big{)}.
\end{equation*}
By Property (1), $Z^n_m = 0$ if $m < n+1$. If $m>n+1$, $\hom _A (P^n, A_0 [m]) =0$ since $P^n$ is generated in degree $n$, so $\text{ext} _A ^{n+1} (A_0, A_0 [m]) = \text{hom} _A (Z^n, A_0 [m])$ by the above identity. But the left-hand side of this identity is non-zero only if $m = n+1$ since $A_0$ is Koszul. Therefore, $\hom _A (Z^n, A_0 [m]) = 0$ for $m > n+1$. Consequently, $Z^n$ is generated in degree $n+1$. By property (2), $Z^n _{n+1} = d^{n+1}_{n+1} (P^{n+1} _{n+1})$, so $Z^n = d^{n+1} (P^{n+1})$ since both modules are generated in degree $n+1$. Therefore, the Koszul complex is acyclic, and hence is a projective resolution of $A_0$.
\end{proof}

\section{Generalized Koszul duality}
In this section we prove the Koszul duality. As before, $A$ is a positively graded, locally finite algebra , and $A_0$ has the splitting property (S). Define $\Gamma = \Ext ^{\ast}_A (A_0, A_0)$ which has a natural grading. Notice that $\Gamma_0 \cong A_0^{\textnormal{op}}$ has the splitting property (S) as well. Let $M$ be a graded $A$-module. Then $\Ext ^{\ast}_A (M, A_0)$ is a graded $\Gamma$-module. Moreover, if $M$ is Koszul, then $\Ext _A^{\ast} (M, A_0)$ is generated in degree 0, so it is a finitely generated $\Gamma$-module. Thus $E = \Ext ^{\ast}_{A} (-, A_0)$ gives rise to a functor from the category of Koszul $A$-modules to $\Gamma$-gmod.

\begin{theorem}
If $A$ is a Koszul algebra, then $E = \Ext ^{\ast}_A (-, A_0)$ gives a duality between the category of Koszul $A$-modules and the category of Koszul $\Gamma$-modules. That is, if $M$ is a Koszul $A$-module, then $E(M)$ is a Koszul $\Gamma$-module, and $E_{\Gamma}EM = \Ext ^{\ast} _{\Gamma} (EM, \Gamma_0) \cong M$ as graded $A$-modules.
\end{theorem}

\begin{proof}
Since $M$ and $A_0$ both are Koszul, by Proposition 2.1.9 $M_0$ and $JM[-1]$ are Koszul, where $J= \bigoplus _{i \geqslant 1} A_i$. Furthermore, we have the following short exact sequence of Koszul modules:
\begin{align*}
\xymatrix{ 0 \ar[r] & \Omega M [-1] \ar[r] & \Omega (M_0)[-1] \ar[r] & JM[-1] \ar[r] & 0.}
\end{align*}
As in the proof of Proposition 2.1.8, this sequence induces exact sequences recursively:
\begin{align*}
\xymatrix{ 0 \ar[r] & \Omega^i(M)[-i] \ar[r] & \Omega^i(M_0)[-i] \ar[r] & \Omega^{i-1}(JM[-1]) [1-i] \ar[r] & 0,}
\end{align*}
and gives exact sequences of $A_0$-modules:
\begin{align*}
\xymatrix{ 0 \ar[r] & \Omega^i(M)_i \ar[r] & \Omega^i(M_0)_i \ar[r] & \Omega^{i-1}(JM [-1])_{i-1} \ar[r] & 0.}
\end{align*}
Note that all terms appearing in the above sequence are projective $A_0$-modules. Applying the functor $\Hom_{A_0} (-, A_0)$ and using the following isomorphism for a graded $A$-module $N$ generated in degree $i$
\begin{equation*}
\Hom _A (N, A_0) \cong \Hom_A (N_i, A_0) \cong \Hom _{A_0} (N_i, A_0,)
\end{equation*}
we get:
\begin{align*}
0 \rightarrow \Hom _A (\Omega^{i-1}(JM[-1]), A_0) \rightarrow \Hom _A (\Omega^i(M_0), A_0) \rightarrow \Hom _A (\Omega^iM, A_0) \rightarrow 0.
\end{align*}
By Lemma 2.1.11, this sequence is isomorphic to
\begin{align*}
0 \rightarrow \Ext^{i-1} _A (JM[-1], A_0) \rightarrow \Ext^i _A (M_0, A_0) \rightarrow \Ext^i _A (M, A_0) \rightarrow 0.
\end{align*}
Now let the index $i$ vary and put these sequences together. We have:
\begin{align*}
\xymatrix{ 0 \ar[r] & E(JM[-1])[1] \ar[r] & E(M_0) \ar[r] & EM \ar[r] & 0.}
\end{align*}

Let us focus on this sequence. We claim $\Omega(EM) \cong E(JM[-1])[1]$. Indeed, since $M_0$ is a projective $A_0$-module and the functor $E$ is additive, $E(M_0)$ is a projective $\Gamma$-module. Since $JM[-1]$ is Koszul, $JM[-1]$ is quasi-Koszul and hence $E(JM[-1])$ as a $\Gamma$-module is generated in degree 0. Thus $E(JM[-1])[1]$ is generated in degree 1, and $E(M_0)$ is minimal. This proves the claim. Consequently, $\Omega(EM)$ is generated in degree 1 as a $\Gamma$-module. Moreover, replacing $M$ by $JM[-1]$ (which is also Koszul) and using the claimed identity, we have
\begin{equation*}
\Omega^2(EM) = \Omega(E(JM[-1])[1]) = \Omega( E(JM[-1]) [1] = E(J^2M[-2])[2],
\end{equation*}
which is generated in degree 2. By recursion, we know that $\Omega^i (EM) \cong E(J^iM [-i])[i]$ is generated in degree $i$ for all $i \geqslant 0$. Thus $EM$ is a Koszul $\Gamma$-module. In particular for $M = A$,
\begin{equation*}
EA = \Ext _A ^{\ast} (A, A_0) = \Hom _A (A, A_0) = \Gamma_0
\end{equation*}
is a Koszul $\Gamma$-module.

Since $\Omega ^i (EM)$ is generated in degree $i$ and
\begin{align*}
\Omega^i (EM)_i & \cong E(J^iM [-i])[i]_i \cong E(J^iM[-i])_0 \\
& = \Hom _A (J^iM[-i], A_0) \cong \Hom _A (M_i, A_0),
\end{align*}
we have
\begin{align*}
\Hom_{\Gamma} (\Omega^i(EM), \Gamma_0) & \cong \Hom _{\Gamma _0} (\Omega^i(EM)_i, \Gamma_0) \nonumber \\
& \cong \Hom _{\Gamma _0} (\Hom_A (M_i, A_0), \Gamma_0) \nonumber \\
& \cong \Hom _{\Gamma_0} (\Hom_{A_0} (M_i, A_0), \Gamma_0) \nonumber \\
& \cong M_i.
\end{align*}
The last isomorphism holds because $M_i$ is a projective $A_0$-module and $\Gamma_0 \cong A_0^{\textnormal{op}}$.

We have proved that $EM$ is a Koszul $\Gamma$-module. Therefore, $(\Omega ^i (EM))_i$ is a projective $\Gamma_0$-module for every $i \geqslant 0$. Applying Lemma 2.1.11 recursively, $\Ext _{\Gamma}^i (EM, \Gamma_0) \cong \Hom_{\Gamma} (\Omega^i(EM), \Gamma_0) \cong M_i$ for every $i \geqslant 0$. Adding them together, we conclude that $E_{\Gamma}E(M) \cong \bigoplus_{i=0}^{\infty} M_i \cong M$.

Now we have $E_{\Gamma} ( E(A)) = E_{\Gamma} (\Gamma_0) \cong A$. Moreover, $\Gamma$ is a graded algebra such that $\Gamma_0 \cong A_0^{\textnormal{op}}$ is self-injective as an algebra and Koszul as a $\Gamma$-module. Using this duality, we can exchange $A$ and $\Gamma$ in the above reasoning and get $EE_{\Gamma}(N) \cong N$ for an arbitrary Koszul $\Gamma$-module $N$. Thus $E$ is a dense functor.

Let $L$ be another Koszul $A$-module. Since $L, M, EL, EM$ are all generated in degree 0, we have
\begin{align*}
\hom _{\Gamma} (EL, EM) & \cong \Hom_{\Gamma_0} ((EL)_0, (EM)_0)\\
& \cong \Hom_{\Gamma_0} (\Hom_A (L, A_0), \Hom_A (M, A_0))\\
& \cong \Hom_{A_0 ^{\textnormal{op}}} (\Hom_{A_0} (L_0, A_0), \Hom_{A_0} (M_0, A_0))\\
& \cong \Hom_{A_0} (L_0, M_0) \cong \hom_A (L, M).
\end{align*}
Consequently, $E$ is a duality between the category of Koszul $A$-modules and the category of Koszul $\Gamma$-modules.
\end{proof}

\begin{remark}
We can also use $\hom _A(-, A_0)$ to define the functor $E$ on the category of Koszul $A$-modules, namely $E := \bigoplus _{i \geqslant 0} \text{ext} _A^i (-, A_0[i])$. Indeed, for a Koszul $A$-module $M$, we have:
\begin{align*}
\Ext _A^{\ast} (M, A_0) & = \bigoplus _{i \geqslant 0} \Ext _A^i (M, A_0) \\
& = \bigoplus _{i \geqslant 0} \bigoplus _{j \in \mathbb{Z}} \text{ext} _A^i (M, A_0 [j]) \\
& = \bigoplus _{i \geqslant 0} \text{ext} _A^i (M, A_0[i])
\end{align*}
since $\text{ext} _A^i (M, A_0 [j]) = 0$ for $i \neq j$.
\end{remark}

\section{A relation to the classical theory}

In this section we describe a correspondence between the generalized Koszul theory and the classical theory. As before, let $A = \bigoplus _{i \geqslant 0} A_i$ be a locally finite, positively graded algebra generated in degrees 0 and 1. At this moment we do \textbf{not} need the splitting condition (S). Let $\mathfrak{r}$ be the radical of $A_0$, and $\mathfrak{R} = A \mathfrak{r} A$ be the two-sided ideal generated by $\mathfrak{r}$. Note that $\mathfrak{r} + J$ is also a two sided-ideal of $A$ where $J = \bigoplus _{i \geqslant 1} A_i$, and it coincides with the radical of $A$ if $A$ is finite-dimensional. We then define the quotient graded algebra $\bar{A} = A / \mathfrak{R} = \bigoplus _{i \geqslant 0} A_i / \mathfrak{R}_i$.

\begin{lemma}
Notation as above, $\mathfrak {R} _s = \sum _{i = 0}^s A_i \mathfrak{r} A_{s-i} = (J + \mathfrak{r}) ^{s+1} _s$, and $\bar{A}$ is a well defined graded algebra.
\end{lemma}

\begin{proof}
Since $\mathfrak{R} = A \mathfrak{r} A$, the first identity is clear. Now we prove the second one. This is clearly true for $s = 0$ since $\sum _{i = 0}^0 A_i \mathfrak{r} A_{s-i} = A_0 \mathfrak{r} A_0 = \mathfrak{r} =  (\mathfrak{r} + J) _0$. If $s \geqslant 1$, then $A_i \mathfrak{r} A_{s-i} \subseteq A_s \subseteq J^i \cdot \mathfrak{r} \cdot J ^{s-i} \subseteq (\mathfrak{r} + J) ^{s+1}$. Therefore, $\sum _{i = 0}^s A_i \mathfrak{r} A_{s-i} \subseteq (\mathfrak{r} + J) ^{s+1}$. But every element is homogeneous and has degree $s$, so $\sum _{i = 0}^s A_i \mathfrak{r} A_{s-i} \subseteq (\mathfrak{r} + J) ^{s+1} _s$.

On the other hand,
\begin{equation*}
 (\mathfrak{r} + J) ^{s+1} = \sum _{X_i = \mathfrak{r}, J, \quad i=0} ^{s+1} X_1 \cdot \ldots \cdot X_{s+1}.
\end{equation*}
Correspondingly,
\begin{equation*}
  (\mathfrak{r} + J) ^{s+1} _s = \sum _{X_i = \mathfrak{r}, J, \quad i = 0} ^{s+1} (X_1 \cdot \ldots \cdot X_{s+1} )_s.
\end{equation*}
Clearly, if all $X_i = J$ for $0 \leqslant i \leqslant s+1$, then $(X_1 \cdot \ldots \cdot X_{s+1} )_s = 0$. So we can assume that there is at leat one $X_i = \mathfrak{r}$.

Take $0 \neq v_s \in (X_1 \cdot \ldots \cdot X_{s+1} )_s$. If $X_1 = \mathfrak{r}$, then $v_s \in \mathfrak{r} A_s$; if $X_{s+1} = \mathfrak{r}$, then $v_s \in A_s \mathfrak{r}$. Otherwise, we have some $0 < t < s+1$ such that $X_i = \mathfrak{r}$, and hence $v_s \in (J \cdot \mathfrak{r} \cdot J )_s = \sum _{i = 1}^s A_i \mathfrak{r} A_{s-i}$. In all cases we have $v_s \in \sum _{i = 0}^s A_i \mathfrak{r} A_{s-i}$. Therefore, $(\mathfrak {r} + J) ^{s+1} _s \subseteq \sum _{i = 0}^s A_i \mathfrak{r} A_{s-i}$. This proves the first statement.

The product of $\bar{A}$ is defined by the following rule. Take $a_s \in A_s$ and $a_t \in A_t$, $s, t \geqslant 0$, we define $\bar{a}_s \cdot \bar{a}_t = \overline{a_s a_t}$ to be the image of $a_sa_t$ in $\bar{A}_{s+t} = A_{s+t} / \mathfrak{R} _{s+t}$. Since by the first statement,
\begin{equation*}
\bar{A}_s = A_s / \sum _{i = 0}^s A_i \mathfrak{r} A_{s-i}, \quad \bar{A}_t = A_t / \sum _{i = 0}^t A_i \mathfrak{r} A_{t-i}, \quad \bar{A}_{s+t} = A_{s+t} / \sum _{i = 0} ^{s+t} A_i \mathfrak{r} A_{s+t-i},
\end{equation*}
it is enough to show that
\begin{equation*}
\sum _{i = 0}^s A_i \mathfrak{r} A_{s-i} \cdot A_t \subseteq \sum _{i = 0} ^{s+t} A_i \mathfrak{r} A_{s+t-i}, \quad A_s \cdot \sum _{i = 0}^t A_i \mathfrak{r} A_{t-i} \subseteq \sum _{i = 0} ^{s+t} A_i \mathfrak{r} A_{s+t-i}.
\end{equation*}
But these two inclusions hold obviously. Therefore, the product defined in this way gives rise to a well define product of $\bar{A}$ by bilinearity.
\end{proof}

Note that $\bar{A}_0 = A_0 / \mathfrak{r}$ is a semisimple algebra.

Given an arbitrary graded $A$-module $M = \bigoplus _{i \geqslant 0} M_i$, we can define a graded $\bar{A}$-module $\bar{M} =  M / \mathfrak{R} M = \bigoplus _{i \geqslant 0} M_i / (\mathfrak{R} M)_i$.

\begin{lemma}
Let $\bar{M}$ be as above. Then
\begin{enumerate}
\item $(\mathfrak {R} M)_n = \sum_{i = 0} ^n A_i \mathfrak{r} M_{n-i}$;
\item $\bar{M}$ is a well defined $\bar{A}$-module;
\item if $M$ is generated in degree 0, then
\begin{equation*}
(\mathfrak {R} M)_n = \sum_{i = 0} ^n A_i \mathfrak{r} M_{n-i} = \sum_{i = 0} ^n A_i \mathfrak{r} A_{n-i} M_0 =( (\mathfrak{r} + J) ^{n+1} M)_n.
\end{equation*}
\end{enumerate}
\end{lemma}

\begin{proof}
Since $A_i \mathfrak{r} \in \mathfrak{R}$, we have  $A_i \mathfrak{r} M_{n-i} \subseteq (\mathfrak{R} M)_i$. Letting $i$ vary we get $(\mathfrak {R} M)_n \supseteq \sum_{i = 0} ^n A_i \mathfrak{r} M_{n-i}$. On the other hand,
\begin{equation*}
(\mathfrak{R} M)_n = \sum _{s=0} ^n \mathfrak{R} _s M_{n-s} = \sum _{s=0}^n \sum_{i=0}^s A_i \mathfrak{r} A_{s-i} M_{n-s} \subseteq \sum _{s=0}^n \sum_{i=0}^s A_i \mathfrak{r} M_{n-i} = \sum _{i=0}^n A_i \mathfrak{r} M_{n-i}.
\end{equation*}
This proves the first statement.

Now we prove the second statement. Take $a_s \in A_s$ and $v_t \in M_t$. We define $\bar{a}_s \cdot \bar{v}_t$ to be the image $\overline{a_s v_t}$ of $a_s v_t$ in $\bar{M}_{s+t}$. Since
\begin{equation*}
\bar{A}_s = A_s / \sum _{i = 0}^s A_i \mathfrak{r} A_{s-i}, \quad \bar{M}_{t} = M_t / \sum _{i = 0}^t A_i \mathfrak{r} M_{t-i}, \quad \bar{M}_{s+t} = M_{s+t} / \sum _{i = 0} ^{s+t} A_i \mathfrak{r} M_{s+t-i},
\end{equation*}
it suffices to show the following two inclusions:
\begin{equation*}
\sum _{i = 0}^s A_i \mathfrak{r} A_{s-i} \cdot M_t \subseteq \sum _{i=0}^{s+t} A_i \mathfrak{r} M_{s+t-i}, \quad A_s \cdot \sum _{i = 0}^t A_i \mathfrak{r} M_{t-i} \subseteq \sum _{i=0}^{s+t} A_i \mathfrak{r} M_{s+t-i}.
\end{equation*}
But these two inclusions are clearly true. Thus $\bar{M}$ is a well defined graded $\bar{A}$-module.

The first identity in the third statement has been established in (1). The second identity is clearly true since $M$ is generated in degree 0. So we only need to show the last identity. It is clear that
\begin{equation*}
((\mathfrak{r} + J) ^{n+1} M)_n = \sum_{i = 0} ^n ( (\mathfrak{r} + J) ^{n+1})_i M_{n-i} = \sum_{i = 0} ^n ( (\mathfrak{r} + J) ^{n+1})_i A_{n-i} M_0.
\end{equation*}
By taking $i = n$ and using the previous lemma, we have
\begin{equation*}
\sum_{i = 0} ^n A_i \mathfrak{r} A_{n-i} M_0 = ((\mathfrak{r} + J) ^{n+1}) _n M_0 \subseteq ( (\mathfrak{r} + J) ^{n+1} M)_n.
\end{equation*}
But on the other hand,
\begin{align*}
\sum_{i = 0} ^n ((\mathfrak{r} + J) ^{n+1})_i A_{n-i} M_0 & \subseteq \sum_{i = 0} ^n ((\mathfrak{r} + J) ^{n+1})_i \cdot ((\mathfrak{r} + J) ^{n-i}) _{n-i}  M_0\\
& \subseteq \sum_{i = 0} ^n ((\mathfrak{r} + J) ^{2n+1-i}) _n  M_0\\
& \subseteq ((\mathfrak{r} + J) ^{n+1}) _n  M_0\\
& = \sum_{i = 0} ^n A_s \mathfrak{r} A_{n-s} M_0
\end{align*}
since for $0 \leqslant i \leqslant n$ we always have $(\mathfrak{r} + J) ^{2n+1-i} \subseteq (\mathfrak{r} + J) ^{n+1}$. This finishes the proof.
\end{proof}

We use an example to show our construction.

\begin{example}
Let $A$ be the path algebra of the following quiver with relations: $\delta^2 = \theta^2 = 0$, $\theta \alpha = \alpha \delta = \beta$. Put $A_0 = \langle 1_x, 1_y, \delta, \theta \rangle$ and $A_1 = \langle \alpha, \beta \rangle$.
\begin{equation*}
\xymatrix{ x \ar@(ld,lu)[]|{\delta} \ar@<0.5ex>[r] ^{\alpha} \ar@<-0.5ex>[r] _{\beta} & y \ar@(rd,ru)[]|{\theta}}
\end{equation*}
The structures of graded indecomposable projective $A$-modules are:
\begin{equation*}
P_x = \begin{matrix} & x_0 & \\ x_0 & & y_1 \\ & y_1 & \end{matrix} \qquad P_y = \begin{matrix} y_0 \\ y_0 \end{matrix}.
\end{equation*}
We find $\mathfrak{r} = \langle \delta, \theta \rangle$, $\mathfrak{R} = \langle \delta, \theta, \beta \rangle$. Then the quotient algebra $\bar{A}$ is the path algebra of the following quiver with a natural grading:
\begin{equation*}
\xymatrix{ x \ar@(ld,lu)[]|{1_x} \ar[r] ^{\alpha} & y \ar@(rd,ru)[]|{1_y}}.
\end{equation*}

Let $M = \rad P_x = \langle \delta, \alpha, \beta \rangle$ which is a graded $A$-module. This module has the following structure and is not generated in degree 0:
\begin{equation*}
M = \begin{matrix} x_0 & & y_1 \\ & y_1 & \end{matrix}.
\end{equation*}
Then $\bar{M}_0 = M_0 / \mathfrak{r} M_0 = \langle \bar{\delta} \rangle \cong \bar{S}_x$, $\bar{M}_1 = M_1 / (\mathfrak{r} M_1 + A_1 \mathfrak{r} M_0) = \langle \bar{\alpha} \rangle \cong \bar{S}_y [1]$. Therefore, $\bar{M} \cong \bar{S}_x \oplus \bar{S}_y [1]$ is a direct sum of two simple $\bar{A}$-modules, and is not generated in degree 0 either.

We also note that since $M$ is not generated in degree 0, the identities in the third statement of the previous lemma are no long true. Indeed, we have:
\begin{align*}
& ((\mathfrak{r} + J)^2 M)_1 = \sum _{i = 0}^1 A_i \mathfrak{r} A_{1-i} M_0 = \mathfrak{r} A_1 M_0 + A_1 \mathfrak{r} M_0 = 0;\\
& (\mathfrak{R} M)_1 = \sum _{i = 0}^1 A_i \mathfrak{r} M_{1-i} = \mathfrak{r} M_1 + A_1 \mathfrak{r} M_0 = \langle \beta \rangle.
\end{align*}
\end{example}

The following proposition is crucial to prove the correspondence.

\begin{proposition}
A graded $A$-module $M$ is generated in degree 0 if and only if the corresponding graded $\bar{A}$-module $\bar{M}$ is generated in degree 0.
\end{proposition}

\begin{proof}
If $M$ is generated in degree 0, then $A_i M_0 = M_i$ for all $i \geqslant 0$. By our definition, it is clear that $\bar{A}_i \bar{M}_0 = \bar{M}_i$. That is, $\bar{M}$ is generated in degree 0.

Conversely, suppose that $\bar{M}$ is generated in degree 0. We want to show $A_i M_0 = M_i$ for $i \geqslant 0$. We use induction to prove this identity. Clearly, it holds for $i = 0$. So we suppose that it is true for all $0 \leqslant i < n$ and consider $M_n$.

Take $v_n \in M_n$ and consider its image $\bar{v}_n$ in $\bar{M}_n = M_n / \sum _{i = 0}^n A_i \mathfrak{r} M_{n-i}$. Since $\bar{M}$ is generated in degree 0, we can find some $a_n \in A_n$ and $v_0 \in M_0$ such that $\bar{v}_n = \bar{a}_n \bar{v}_0$. Thus $\overline {v_n - a_n v_0} = \bar{v}_n - \bar{a}_n \bar{v}_0 = 0$. This means
\begin{equation*}
v_n - a_n v_0 \in \sum_{i = 0}^n A_i \mathfrak{r} M_{n-i} = \mathfrak{r} M_n + \sum _{i=1}^n A_i \mathfrak{r} M_{n-i} = \mathfrak{r} M_n + \sum _{i=1}^n A_i \mathfrak{r} A_{n-i} M_0,
\end{equation*}
where the last identity follows from the induction hypothesis. But it is clear $A_n M_0 \supseteq \sum _{i=1}^n A_i \mathfrak{r} A_{n-i} M_0$, so $v_n - a_n v_0 \in \mathfrak{r} M_n + A_n M_0$. Consequently, $v_n \in \mathfrak{r} M_n + A_n M_0$. Since $v_n \in M_n$ is arbitrary, we have $M_n \subseteq \mathfrak{r} M_n + A_n M_0$. Applying Nakayama's lemma to these $A_0$-modules, we conclude that $M_n = A_n M_0$ as well. The conclusion then follows from induction.
\end{proof}

\begin{lemma}
Let $M$ be a graded $A$-module generated in degree 0. If $P$ is a grade projective cover of $M$, then $\bar{P}$ is a graded projective cover of $\bar{M}$.
\end{lemma}

\begin{proof}
Clearly, $\bar{P}$ is a grade projective module. Both $\bar{P}$ and $\bar{M}$ are generated in degree 0 by the previous proposition. To show that $\bar{P}_0$ is a graded projective cover of $\bar{M}_0$, it suffices to show that $\bar{P}_0$ is a projective cover of $\bar{M}_0$ as $\bar{A}_0$-modules. But this is clearly true since $\bar{P}_0 = P_0 / \mathfrak{r} P_0 \cong M_0 / \mathfrak{r} M_0 \cong \bar{M}_0$.
\end{proof}

The procedure of sending $M$ to $\bar{M}$ preserves exact sequences of graded $A$-modules which are projective $A_0$-modules.

\begin{lemma}
Let $0 \rightarrow L \rightarrow M \rightarrow N \rightarrow 0$ be a short exact sequence of graded $A$-modules such that all terms are projective $A_0$-modules. Then the corresponding sequence $0 \rightarrow \bar{L} \rightarrow \bar{M} \rightarrow \bar{N} \rightarrow 0$ is also exact.
\end{lemma}

\begin{proof}
For each $i \geqslant 0$, the given exact sequence induces a short exact sequence of $A_0$-modules $0 \rightarrow L_i \rightarrow M_i \rightarrow N_i \rightarrow 0$. Since all terms are projective $A_0$-modules, this sequence splits. Thus we get an exact sequence of $\bar{A}_0$-modules $0 \rightarrow \bar{L}_i \rightarrow \bar{M}_i \rightarrow \bar{N}_i \rightarrow 0$. Let the index $i$ vary and take direct sum. Then we get an exact sequence of graded $\bar{A}$-modules $0 \rightarrow \bar{L} \rightarrow \bar{M} \rightarrow \bar{N} \rightarrow 0$ as claimed.
\end{proof}

The condition that all terms are projective $A_0$-modules cannot be dropped, as shown by the following example.

\begin{example}
Let $A = A_0 = k[t]/(t^2)$ and $S$ be the simple module and consider a short exact sequence of graded $A$-modules $0 \rightarrow S \rightarrow A \rightarrow S \rightarrow 0$. We have $\bar{A} \cong k$. But the corresponding sequence $0 \rightarrow \bar{S} \rightarrow \bar{A} \rightarrow \bar{S} \rightarrow 0$ is not exact. Actually, the first map $\bar{S} \rightarrow \bar{A}$ is 0 since the image of $S$ is contained in $\mathfrak{r} A_0$.
\end{example}

Now we can prove the main result of this section.

\begin{theorem}
Let $A = \bigoplus _{i \geqslant 1} A_i$ be a locally finite graded algebra and $M$ be a graded $A$-module. Suppose that both $A$ and $M$ are projective $A_0$-modules. Then $M$ is generalized Koszul if and only if  the corresponding grade $\bar{A}$-module $\bar{M}$ is classical Koszul. In particular, $A$ is a generalized Koszul algebra if and only if $\bar{A}$ is a classical Koszul algebra.
\end{theorem}

\begin{proof}
Let
\begin{equation}
\xymatrix{ \ldots \ar[r] & P^2 \ar[r] & P^1 \ar[r] & P^0 \ar[r] & M \ar[r] & 0}
\end{equation}
be a minimal projective resolution of $M$. Note that all terms in this resolution and all syzygies are projective $A_0$-modules. By Lemmas 2.4.5 and 2.4.6, $\bar{M}$ has the following minimal projective resolution
\begin{equation}
\xymatrix{ \ldots \ar[r] & \overline {P^2} \ar[r] & \overline {P^1} \ar[r] & \overline {P^0} \ar[r] & \overline {M} \ar[r] & 0}.
\end{equation}
Moreover, this resolution is linear if and only if the resolution (2.4.1) is linear by Proposition 2.4.4. That is, $M$ is generalized Koszul if and only if $\bar{M}$ is classical Koszul. This proves the first statement. Applying it to the graded $A$-module $A_0$ we deduce the second statement immediately.
\end{proof}

If $A$ has the splitting property (S), we have a corresponding version for the previous theorem.

\begin{corollary}
Let $A = \bigoplus _{i \geqslant 1} A_i$ be a locally finite graded algebra satisfying the splitting property (S).
\begin{enumerate}
\item $A$ is a generalized Koszul algebra if and only it is a projective $A_0$-module and $\bar{A}$ is a classical Koszul algebra.
\item Suppose that $A$ is a projective $A_0$-module. A graded $A$-module $M$ is generalized Koszul if and only if it is a projective $A_0$-module and the corresponding graded $\bar{A}$-module $\bar{M}$ is classical Koszul.
\end{enumerate}
\end{corollary}

\begin{proof}
If $A$ is a generalized Koszul algebra, then it is a projective $A_0$-module. Moreover, $\bar{A}$ is a classical Koszul algebra by the previous theorem. The converse statement also follows from the previous theorem. This proves the first statement.

If $A$ is a projective $A_0$-module and $M$ is generalized Koszul, then $M$ is a projective $A_0$-module. Moreover, $\bar{M}$ is a classical Koszul module by the previous theorem. The converse statement follows from the previous theorem as well.
\end{proof}

We cannot drop the condition that $A$ is a projective $A_0$-module in the above theorem, as shown by the following example.

\begin{example}
Let $A$ be the path algebra of the following quiver with relations: $\delta^2 = \theta^2 = 0$, $\theta \alpha = \alpha \delta = 0$. Put $A_0 = \langle 1_x, 1_y, \delta, \theta \rangle$ and $A_1 = \langle \alpha \rangle$.
\begin{equation*}
\xymatrix{ x \ar@(ld,lu)[]|{\delta} \ar[r] ^{\alpha} & y \ar@(rd,ru)[]|{\theta}}
\end{equation*}
The structures of graded indecomposable projective $A$-modules are:
\begin{equation*}
P_x = \begin{matrix} & x_0 & \\ x_0 & & y_1 \end{matrix} \qquad P_y = \begin{matrix} y_0 \\ y_0 \end{matrix}.
\end{equation*}
We find $\mathfrak{r} = \langle \delta, \theta \rangle = \mathfrak{R}$. Then the quotient algebra $\bar{A}$ is the path algebra of the following quiver:
\begin{equation*}
\xymatrix{ x \ar@(ld,lu)[]|{1_x} \ar[r] ^{\alpha} & y \ar@(rd,ru)[]|{1_y}}.
\end{equation*}

Let $\Delta_x = P_x / S_y = \langle \delta, 1_x \rangle$ which is a graded $A$-module concentrated in degree 0. The first syzygy $\Omega (\Delta_x) \cong S_y[1]$ is generated in degree 1, but the second syzygy $\Omega^2 (\Delta_x) \cong S_y[1]$ is not generated in degree 2. Therefore, $\Delta_x$ is not generalized Koszul. However, $\bar{\Delta}_x \cong \bar{S}_x$ is obviously a classical Koszul $\bar{A}$-module. Moreover, we can check that $A$ is not a generalized Koszul algebra, but $\bar{A}$ is a classical Koszul algebra.
\end{example} 
\chapter{Applications to directed categories}
\label{Applications to directed categories}

In this chapter we apply the general Koszul theory developed in the previous chapter to a type of structure called \textit{directed categories}. We first exploit their stratification properties, and show that there is a close relation between the stratification theory and the generalized Koszul theory. We also prove another correspondence between the generalized Koszul theory and the classical theory for directed categories similar to the one described in Section 2.4.

All directed categories $\mathcal{A}$ we consider in this chapter are \textit{locally finite $k$-linear categories} with finitely many objects. That is, for $x, y \in \Ob \mathcal{A}$, the set of morphisms $\mathcal{A} (x,y)$ is a finite-dimensional $k$-vector space, and the composition of morphisms is $k$-linear. To simplify the technical part, we suppose furthermore that $\mathcal{A}$ is \textit{skeletal}, i.e., $x \cong y$ implies $x = y$ for $x,y \in \Ob \mathcal{A}$.

\section{Stratification properties of directed categories}

We start with some preliminary background from stratification theory, for which the reader can find more details in \cite{Cline, Dlab1, Dlab2, Webb1, Webb3, Xi}. Let $A$ be a basic finite-dimensional $k$-algebra with a chosen set of orthogonal primitive idempotents $\{ e_{\lambda} \} _{\lambda \in \Lambda}$ indexed by a set $\Lambda$ such that $\sum _{\lambda \in \Lambda} e_{\lambda} = 1$. Let $P_{\lambda} = A e_{\lambda}$ and $S_{\lambda} = P_{\lambda} / \rad P_{\lambda}$. According to \cite{Cline}, $A$ is \textit{standardly stratified} with respect to a fixed preorder $\preccurlyeq$ if there exist modules $\Delta_{\lambda}$, $\lambda \in \Lambda$, such that the following conditions hold:
\begin{enumerate}
\item the composition factor multiplicity $[\Delta_{\lambda} : S_{\mu}] = 0$ whenever $\mu \npreceq \lambda$;
and
\item for every $\lambda \in \Lambda$ there is a short exact sequence $0 \rightarrow K_{\lambda} \rightarrow P_{\lambda} \rightarrow \Delta_{\lambda} \rightarrow 0$ such that $K_{\lambda}$ has a filtration with factors $\Delta_{\mu}$ where $\mu \succ \lambda$.
\end{enumerate}
The \textit{standard module} $\Delta_{\lambda}$ is the largest quotient of $P_{\lambda}$ with only composition factors $S_{\mu}$ such that $\mu \preccurlyeq \lambda$. Similarly, the \textit{proper standard module} $\bar{ \Delta}_{\lambda}$ can be defined as the largest quotient of $P_{\lambda}$ such that all composition factors $S_{\mu}$ satisfy $\mu \succ \lambda$ except for one copy of $S_{\lambda}$. \textit{Costandard module} $\nabla _{\lambda}$ and \textit{proper costandard module} $\bar {\nabla} _{\lambda}$ are defined dually.

In the case that $A$ is standardly stratified, standard modules $\Delta_{\lambda}$ have the following description:
\begin{equation*}
\Delta_{\lambda} = P_{\lambda} / \sum _{\mu \succ \lambda} \text{tr} _{P_{\mu}} (P_{\lambda}),
\end{equation*}
where tr$_{P_{\mu}} (P_{\lambda})$ is the trace of $P_{\mu}$ in $P _{\lambda}$ (\cite{Dlab1, Webb3}). Let $\Delta$ be the direct sum of all standard modules and $\mathcal{F} (\Delta)$ be the full subcategory of $A$-mod such that each object in $\mathcal{F} (\Delta)$ has a filtration by standard modules. Similarly we define $\mathcal{F} (\overline {\Delta})$, $\mathcal{F} (\nabla)$ and $\mathcal{F} (\overline {\nabla})$. It is well known that $\Delta$ has finite projective dimension, and $\mathcal{F} (\Delta)$ is closed under direct summands, extensions, kernels of epimorphisms, but it is not closed under cokernels of monomorphisms in general. For more details, see \cite{Cline, Dlab1, Dlab2, Webb3}.

We construct a relation $\sim$ on the preordered set $\Lambda$ as follows: for $\lambda, \mu \in \Lambda$, define $\lambda \sim \mu$ if $\lambda \preccurlyeq \mu$ and $\mu \preccurlyeq \lambda$. The reader can check that $\sim$ is an equivalence relation. Let $\{ \Lambda_i \} _{i=1}^n$ be the set of all equivalence classes and define $E_i = \sum _{\lambda \in \Lambda_i} e_{\lambda}$. Then we define a partial order $\leqslant$ on the set $\{ E_i \} _{i =1}^n$ in the following way: $E_i \leqslant E_j$ if and only if there are $\lambda \in \Lambda_i$ and $\mu \in \Lambda_j$ such that $\lambda \preccurlyeq \mu$. The reader can check that this partial order is well defined. In the case that this induced partial order $\leqslant$ is a linear order, $A$ is standardly stratified if and only if $_AA \in \mathcal{F} (\Delta)$ according to Theorem 3.4 in \cite{Webb3}, and the algebra is \textit{properly stratified} if $_AA \in \mathcal{F} (\Delta) \cap \mathcal{F} (\overline {\Delta})$. If the endomorphism algebra of each standard module is one-dimensional, then $A$ is called \textit{quasi-hereditary}.

Now we introduce directed categories and consider their stratification properties.

\begin{definition}
A locally finite $k$-linear category $\mathcal{A}$ is a directed category if there is a partial order $\leqslant$ on $\Ob \mathcal{A}$ such that $x \leqslant y$ whenever $\mathcal{A} (x, y) \neq 0$.
\end{definition}

Correspondingly, we define \textit{directed algebras}.

\begin{definition}
A finite-dimensional algebra $A$ is called a directed algebra with respect to a partial order $\leqslant$ on a chosen set of orthogonal idempotents $\{ e_i \} _{i=1}^n$ with $\sum _{i=1}^n  e_i = 1$ if $e_j \leqslant e_i$ whenever $\Hom _A (A e_i, A e_j) \cong e_i A e_j \neq 0$.
\end{definition}

Notice that in the above definition we do not require the idempotents $e_i$ to be primitive. Clearly, every algebra $A$ is always directed with respect to the trivial set $\{1\}$.

We point out that although a directed category $\mathcal{A}$ is defined with respect to a partial order $\leqslant$, we can extend this partial order to a linear order with respect to which $\mathcal{A}$ is still directed. Indeed, let $O_1$ be the set of maximal objects in $\mathcal{A}$ with respect to $\leqslant$, $O_2$ be the set of maximal objects in $\Ob \mathcal{A} \setminus O_1$ with respect to $\leqslant$, and so on. In this way we get $\Ob \mathcal{A} = \bigsqcup _{i=1}^n O_i$. Then define arbitrary linear orders $\leqslant_i$ on $O_i$, $1 \leqslant i \leqslant n$. Take $x, y \in \Ob \mathcal{A}$ and suppose that $x \in O_i$ and $y \in O_j$, $1 \leqslant i, j \leqslant n$. Define $x \tilde {<} y$ if $i < j$ or $i = j$ but $x <_i y$. Then $\tilde{\leqslant}$ is a linear order extending $\leqslant$, and $\mathcal{A}$ is still directed with respect to $\tilde{\leqslant}$. Therefore, without loss of generality we assume that $\leqslant$ is linear.

Gabriel's construction \cite{Bautista} gives a bijective correspondence between directed categories and directed algebras. Let $A$ be a directed algebra with respect to a linear order $\leqslant$ on a chosen set of orthogonal idempotents $\{e_i\} _{i=1}^n$ such that $\sum _{i=1}^n  e_i = 1$. Then we can construct a directed category $\mathcal{A}$ in the following way: $\Ob \mathcal{A} = \{e_i\} _{i=1}^n$ with the same linear order, and $\mathcal{A} (e_i, e_j) = e_j A e_i \cong \Hom _A (A e_j, A e_i)$. The reader can check that $\mathcal{A}$ is indeed a directed category. We call $\mathcal{A}$ the \textit{associated category} of $A$.

Conversely, given a directed category $\mathcal{A}$ with a linear order $\leqslant$ on $\Ob \mathcal{A}$, we obtain an algebra $A$ which is directed with respect to $\leqslant$ on the chosen set of orthogonal idempotents $\{ 1_x \} _{x \in \Ob \mathcal{A}}$, namely, $1_x \leqslant 1_y$ if and only if $x \leqslant y$. As a $k$-vector space, $A = \bigoplus _{x, y \in \Ob \mathcal{A}} \mathcal{A} (x, y)$. For two morphisms $\alpha: x \rightarrow y$ and $\beta: z \rightarrow w$, the product $\beta \cdot \alpha = 0$ if $y \neq z$, otherwise it is the composite morphism $\beta \alpha$. Since every vector in $A$ is a linear combination of morphisms in $\mathcal{A}$, the multiplication of morphisms can be extended linearly to a well defined product in $A$. The reader can check that the algebra $A$ we get in this way is indeed a directed algebra, which is called the \textit{associated algebra} of $\mathcal{A}$.

It is well known that $A$-mod is equivalent to the category of finite-dimensional \textit{$k$-linear representations} of $\mathcal{A}$, where a $k$-linear representation of $\mathcal{A}$ is defined as a $k$-linear functor from $\mathcal{A}$ to the category of finite-dimensional vector spaces. If one of $A$ and $\mathcal{A}$ is graded, then the other one can be graded as well. Moreover, $A$-gmod is equivalent to the category of finite-dimensional graded $k$-linear representations of $\mathcal{A}$. For more details, see \cite{Mazorchuk1}. Because of these facts, we may identify a directed category $\mathcal{A}$ with its associated algebra $A$ and abuse notation and terminologies. For example, we may say idempotents in $\mathcal{A}$, ideals of $\mathcal{A}$ and so on. We hope this will not cause confusions to the reader and point out that all results in previous chapter can be applied to graded directed categories.

Let $\mathcal{A}$ be a directed category. The morphism space of $\mathcal{A}$ can be decomposed as the direct sum of $\mathcal{A} 1_x$ with $x$ ranging over all objects, where by $\mathcal{A} 1_x$ we denote the vector space formed by all morphisms with source $x$. Therefore, each $\mathcal{A} 1_x$ is a projective $\mathcal{A}$-module, and every indecomposable projective $\mathcal{A}$-module is isomorphic to a direct summand of a certain $\mathcal{A} 1_x$. The isomorphism classes of simple $\mathcal{A} (x, x)$-modules with $x$ varying within $\Ob \mathcal{A}$ give rise to isomorphism classes of simple $\mathcal{A}$-modules. Explicitly, let $V$ be a simple $\mathcal{A} (x, x)$-module for some object $x$, we can construct a simple $\mathcal{A}$-module $S$: $S(x) = V$ and $S(y) = 0$ for $y \neq x$. These results are well known for finite EI categories, see \cite{Webb2}.

Our next task is to translate some results on finite EI categories in Section 2 of \cite{Webb3} to directed categories. First, we want to show that every directed category is stratified with respect to the given linear order.

\begin{proposition}
Let $\mathcal{D}$ and $\mathcal{E}$ be full subcategories of a directed category $\mathcal{A}$ such that $\Ob \mathcal{D} \cup \Ob \mathcal{E} = \Ob \mathcal{A}$, $\Ob \mathcal{D} \cap \Ob \mathcal{E} = \emptyset$, and $\mathcal{A} (x,y) = 0$ for $x \in \Ob \mathcal{D}$ and $y \in \Ob \mathcal{E}$. Let $e = \sum _{x \in \Ob \mathcal{D}} 1_x$ and $I = \mathcal{A} e \mathcal{A}$. Then $I$ is a stratified ideal of $\mathcal{A}$.
\end{proposition}

\begin{proof}
The proof is similar to that of Proposition 2.2 in \cite{Webb2}. Clearly $I$ is idempotent. Notice that $\mathcal{A} e$ is the space constituted of all morphisms with sources contained in $\Ob \mathcal{D}$ and $e \mathcal{A} e$ is the space constituted of all morphisms with both sources and targets contained in $\Ob \mathcal{D}$. Since $\mathcal{A} (x, y) = 0$ for $x \in \Ob \mathcal{D}$ and $y \in \Ob \mathcal{E}$, these two spaces coincide, i.e., $\mathcal{A} e = e \mathcal{A} e$. In particular, $\mathcal{A} e$ is projective $e \mathcal{A} e$-module, here $e \mathcal{A} e$ is an algebra for which the associated directed category is precisely $\mathcal{D}$. Therefore, Tor$ _n ^{e \mathcal{A} e} (\mathcal{A}e, e\mathcal{A} ) = 0$ for $n \geqslant 1$. Furthermore,
\begin{equation*}
\mathcal{A} e \otimes _{e \mathcal{A} e} e \mathcal{A} = e \mathcal{A} e \otimes _{e \mathcal{A} e} e \mathcal{A} \cong e \mathcal{A}.
\end{equation*}
We claim $e \mathcal{A} = \mathcal{A} e \mathcal{A}$. Clearly, $e \mathcal{A} \subseteq \mathcal{A} e \mathcal{A}$. On the other hand, since we just proved $\mathcal{A} e = e \mathcal{A} e$, we have $\mathcal{A} e \mathcal{A} = e \mathcal{A} e \mathcal{A} \subseteq e \mathcal{A}$. Therefore, $e \mathcal{A} = \mathcal{A} e \mathcal{A}$ as we claimed. In conclusion, $I$ is indeed a stratified ideal of $\mathcal{A}$.
\end{proof}

\begin{corollary}
Every directed category $\mathcal{A}$ is stratified with respect to the given linear order on $\Ob \mathcal{A}$.
\end{corollary}

\begin{proof}
The linear order $\leqslant$ on $\Ob \mathcal{A}$ gives a filtration on $\Ob \mathcal{A}$ in the following way: let $S_1$ be a set containing the unique maximal object in $\Ob \mathcal{A}$, $S_2$ is formed by adding the unique maximal object in $\Ob \mathcal{A} \setminus S_1$ into $S_1$, $S_3$ is formed by adding the maximal object in $\Ob \mathcal{A} \setminus S_2$ into $S_2$, and so on. Consider the full subcategories $\mathcal{D}_i$ formed by $S_i$ and let $e_i = \sum _{x \in S_i} 1_x$. Then the ideals $\mathcal{A} e_i \mathcal{A}$ give a stratification of $\mathcal{A}$ by the previous proposition.
\end{proof}

Now we want to describe standard modules and give a characterization of standardly stratified directed categories $\mathcal{A}$ with respect to the particular preorders induced by the given linear orders. Before doing that, we need to define this preorder on a complete set of primitive idempotents of $\mathcal{A}$ (or precisely, primitive idempotents of the associated algebra $A$). For every object $x$, $\mathcal{A} (x, x) = 1_x \mathcal{A} 1_x$ is a finite-dimensional $k$-algebra, so we can choose a complete set of orthogonal primitive idempotents $E_x = \{e_{\lambda}\} _{\lambda \in \Lambda_x}$ with $\sum_{\lambda \in \Lambda_x} e_{\lambda} = 1_x$. In this way we get a complete set of orthogonal primitive idempotents $\bigsqcup _{x \in \Ob \mathcal{A}} E_x$. The linear order $\leqslant$ on $\Ob \mathcal{A}$ can be used to define a preordered set $(\Lambda, \preceq)$ indexing all these primitive idempotents, namely for $e_{\lambda} \in E_x$ and $e_{\mu} \in E_y$, $e_{\lambda} \preceq e_{\mu}$ if and only if $x \leqslant y$. We can check that $\preceq$ is indeed a preorder. We denote $e_{\lambda} \prec e_{\mu}$ if $e_{\lambda} \preceq e_{\mu}$ but $e_{\mu} \npreceq e_{\lambda}$ for $\lambda, \mu \in \Lambda$. Let $P_{\lambda} = \mathcal{A} e_{\lambda}$. Therefore, the preordered set $(\Lambda, \preceq)$ also index all indecomposable projective representations of $\mathcal{A}$ up to isomorphism.

The following proposition gives a description of standard $\mathcal{A}$-modules with respect to the above preorder.

\begin{proposition}
The standard $\mathcal{A}$-module $\Delta_{\lambda}$ is only supported on $x$ with value $\Delta_{\lambda} (x) \cong 1_x \mathcal{A} e_{\lambda}$, where $x \in \Ob \mathcal{A}$ and $e_{\lambda} \in E_x$.
\end{proposition}

\begin{proof}
Recall standard modules have the following description:
\begin{equation*}
\Delta_{\lambda} = P_{\lambda} / \sum_{\mu \succ \lambda, \text{ } \mu \in \Lambda} \text{tr} _{P_{\mu}} (P _{\lambda}),
\end{equation*}
where tr$_{P_{\mu}} (P_{\lambda})$ is the trace of $P_{\mu}$ in $P_{\lambda}$.

Let us first analyze the structure of $P_{\lambda} = \mathcal{A} e_{\lambda}$. Since $e_{\lambda} \in E_x$, $P_{\lambda}$ is a direct summand of $\mathcal{A} 1_x$. The value of $\mathcal{A} 1_x$ on an arbitrary object $y$ is $1_y \mathcal{A} 1_x$, the space of all morphisms from $x$ to $y$. Therefore, the value of $P_{\lambda}$ on $y$ is $1_y \mathcal{A} e_{\lambda}$. Since $\mathcal{A}$ is directed with respect to $\leqslant$, if $x > y$, then there is no nontrivial morphisms from $x$ to $y$. Therefore, $1_y \mathcal{A} 1_x$ and hence $1_y \mathcal{A} e_{\lambda}$ are 0. We deduce immediately that $\Delta_{\lambda}$ is only supported on objects $y$ satisfying $x \leqslant y$.

Let $y$ be an object such that $y > x$. Then every $e_{\mu} \in E_y$ satisfies $e_{\mu} \succ e_{\lambda}$. Since $\sum _{e_{\mu} \in E_y} e_{\mu} = 1_y$, by taking the sum we find that tr$ _{\mathcal{A} 1_y} (P_{\lambda})$ is contained in $\sum_{\mu \succ \lambda, \text{ } \mu \in \Lambda} \text{tr} _{P_{\mu}} (P_{\lambda})$. The value on $y$ of tr$ _{\mathcal{A} 1_y} (\mathcal{A} 1_x)$ is $1_y \mathcal{A} 1_x$. Since $P_{\lambda} = \mathcal{A} e_{\lambda}$ is a direct summand of $\mathcal{A} 1_x$, the value on $y$ of tr$ _{\mathcal{A} 1_y} (P_{\lambda})$ is $1_y \mathcal{A} e_{\lambda}$. Consequently, the value of $\sum_{\mu \succ \lambda, \text{ } \mu \in \Lambda} \text{tr} _{P_{\mu}} (P_{\lambda})$ on $y$ contains $1_y \mathcal{A} e_{\lambda}$, which equals the value of $P_{\lambda}$ on $y$. Therefore, the value of $\sum_{\mu \succ \lambda, \text{ } \mu \in \Lambda} \text{tr} _{P_{\mu}} (P_{\lambda})$ on $y$ is precisely $1_y \mathcal{A} e_{\lambda}$, so the value of $\Delta_{\lambda}$ on $y$ is 0.

We have proved that $\Delta_{\lambda}$ is only supported on $x$. Clearly, its value on $x$ is $1_x \mathcal{A} e_{\lambda}$.
\end{proof}

This proposition tells us that standard modules are exactly indecomposable direct summands of $\bigoplus _{x \in \Ob \mathcal{A}} \mathcal{A} (x, x)$ (viewed as a $\mathcal{A}$-module by identifying it with the quotient module $\bigoplus _{x, y \in \Ob \mathcal{A}} \mathcal{A} (x, y) / \bigoplus _{x \neq y} \mathcal{A} (x, y)$).

To simplify the expression, we stick to the following convention frow now on:\\

\textbf{Convention:} When we say that a directed category is standardly stratified, we always refer to the preorder $\preceq$ induced by the given linear order $\leqslant$ on the set of objects.\\

The next theorem characterizes standardly stratified directed categories.

\begin{theorem}
Let $\mathcal{A}$ be a directed category. Then $\mathcal{A}$ is standardly stratified if and only if the morphism space $\mathcal{A} (x,y)$ is a projective $\mathcal{A} (y,y)$-module for every pair of objects $x, y \in \Ob \mathcal{A}$. In particular, if $\mathcal{A}$ is standardly stratified, then $\bigoplus _{x \in \Ob \mathcal{A}}  \mathcal{A} (x, x)$ as a $\mathcal{A}$-module has finite projective dimension.
\end{theorem}

\begin{proof}
Suppose that $\mathcal{A}$ is standardly stratified and take two arbitrary objects $x$ and $y$ in $\mathcal{A}$. Since 0 is regarded as a projective module, we can assume $\mathcal{A} (x, y) \neq 0$ and want to show that it is a projective $\mathcal{A} (y, y)$-module. Consider the projective $\mathcal{A}$-module $\mathcal{A} 1_x$, which has a filtration with factors standard modules. Since each standard module is only supported on one object, the value of $\mathcal{A} 1_x$ on $y$ is exactly the sum of these standard modules with non-zero values on $y$. This sum is direct since standard modules supported on $y$ are non-comparable with respect to the preorder and therefore there is no extension between them (or because by the previous proposition each of these standard modules is projective viewed as a $\mathcal{A} (y,y)$-module). Therefore, the value of $\mathcal{A} 1_x$ on $y$ is a projective $\mathcal{A} (y,y)$-module. But the value of $\mathcal{A} 1_x$ on $y$ is precisely $\mathcal{A} (x,y)$, so the only if part is proved.

Conversely, let $P_{\lambda} = \mathcal{A} e_{\lambda}$ be an indecomposable projective $\mathcal{A}$-module. Its value on an arbitrary object $y$ is $1_y \mathcal{A} e_{\lambda} \cong 1_y \mathcal{A} 1_x$ which is either 0 or isomorphic to a direct summand of $\mathcal{A} (x, y)$. If $\mathcal{A} (x, y)$ is a projective $\mathcal{A} (y, y)$-module, then the value of $P_{\lambda}$ on $y$ is a projective $\mathcal{A} (y, y)$-module as well. This value can be expressed as a direct sum of standard modules supported on $y$ since standard modules are exactly indecomposable direct summands of $\bigoplus _{x \in \Ob \mathcal{A}} \mathcal{A} (x, x)$. Therefore we can get a filtration of $P_{\lambda}$ by standard modules.

It is well known that the projective dimension of a standard module is finite if the algebra is standardly stratified. Since $\bigoplus _{x \in \Ob \mathcal{A}}  \mathcal{A} (x, x)$ as a $\mathcal{A}$-module is a direct sum of standard modules, the last statement follows from this fact immediately.
\end{proof}

If the directed category $\mathcal{A}$ is standardly stratified, then $\bigoplus _{x \in \Ob \mathcal{A}} \mathcal{A} (x, x)$ has finite projective dimension. We will prove later that if $\mathcal{A}$ is the $k$-linearization of a finite EI category, the converse statement is also true.

\section{Koszul properties of directed categories}

From now on we suppose that $\mathcal{A}$ is a \textit{graded} category, that is, there is a grading on the morphisms in $\mathcal{A}$ such that $\mathcal{A}_i \cdot \mathcal{A}_j \subseteq \mathcal{A} _{i+j}$, where we denote the subspace constituted of all morphisms with grade $i$ by $\mathcal{A}_i$. Furthermore, $\mathcal{A}$ is supposed to satisfy the following condition: $\mathcal{A}_i \cdot \mathcal{A}_j = \mathcal{A} _{i+j}$. Every vector in $\mathcal{A}_i$ is a linear combination of morphisms with degree $i$. Clearly, $\mathcal{A}_i = \bigoplus _{x, y \in \Ob \mathcal{A}} \mathcal{A} (x, y)_i$. We always suppose $\mathcal{A}_i = 0$ for $i <0$ and $\mathcal{A}_0 = \bigoplus _{x \in \Ob \mathcal{A}} \mathcal{A} (x, x)$. This is equivalent to saying that $\mathcal{A}_0$ is the direct sum of all standard $\mathcal{A}$-modules by Proposition 3.1.5.

Given a graded directed category $\mathcal{A}$, we can apply the functor $E = \Ext _{\mathcal{A}} ^{\ast} (-, \mathcal{A}_0)$ to construct the \textit{Yoneda category} $E (\mathcal{A}_0)$: $\Ob E(\mathcal{A} _0) = \Ob \mathcal{A}$ and $E (\mathcal{A} _0) (x, y) _n = \Ext ^n _{\mathcal{A}} (\mathcal{A}(x, x), \mathcal{A} (y,y))$. This is precisely the categorical version of the Yoneda algebra. By the correspondence between graded algebras and graded categories, we can define \textit{Koszul categories, quasi-Koszul categories, quadratic categories, Koszul modules, quasi-Koszul modules} for graded categories as well. We do not repeat these definitions here but emphasize that all results described in the previous sections can be applied to graded categories.

A corollary of Theorem 3.1.6 relates stratification theory to Koszul theory in the context of directed categories.

\begin{theorem}
Let $\mathcal{A}$ be a graded directed category and suppose that $\mathcal{A}_0$ coincides with $\bigoplus _{x \in \Ob \mathcal{A}} \mathcal{A} (x, x)$ and satisfies the splitting property (S). Then:
\begin{enumerate}
\item $\mathcal{A}$ is standardly stratified if and only if $\mathcal{A}$ is a projective $\mathcal{A}_0$-module.
\item $\mathcal{A}$ is a Koszul category if and only if $\mathcal{A}$ is standardly stratified and quasi-Koszul.
\item If $\mathcal{A}$ is standardly stratified, then a graded $\mathcal{A}$-module $M$ generated in degree 0 is Koszul if and only if it is a quasi-Koszul $\mathcal{A}$-module and a projective $\mathcal{A}_0$-module.
\end{enumerate}
\end{theorem}

\begin{proof}
Take an arbitrary pair of objects $x, y \in \Ob \mathcal{A}$. If $\mathcal{A}$ is standardly stratified, then $\mathcal{A} (x,y)$ is either 0 (a zero projective module) or a projective $\mathcal{A} (y,y)$-module by Theorem 3.1.6. Notice that each $\mathcal{A} (x, y)_i$ is a $\mathcal{A} (y,y)$-module since $\mathcal{A} (y,y) \subseteq \mathcal{A}_0$, and we have the decomposition $\mathcal{A} (x,y) = \bigoplus _{i \geqslant 0} \mathcal{A} (x,y)_i$. Therefore, $\mathcal{A}(x,y)_i$ is a projective $\mathcal{A} (y,y)$-module, and hence a projective $\mathcal{A}_0$-module since only the block $\mathcal{A} (y,y)$ of $\mathcal{A}_0$ acts on $\mathcal{A}(x,y)_i$ nontrivially. In conclusion, $\mathcal{A}_i = \bigoplus _{ x, y \in \Ob \mathcal{A}} \mathcal{A} (x,y)_i$ is a projective $\mathcal{A}_0$-module. Conversely, if $\mathcal{A}_i$ are projective $\mathcal{A}_0$-modules for all $i \geqslant 0$, then each $\mathcal{A} (x, y)_i$, and hence $\mathcal{A} (x,y)$ are projective $\mathcal{A} (y,y)$-modules, so $\mathcal{A}$ is standardly stratified again by the previous theorem. The first statement is proved.

We know that $\mathcal{A}$ is Koszul if and only if it is quasi-Koszul and $\mathcal{A}$ is a projective $\mathcal{A}_0$-module. Then (2) follows immediately from (1).

The last part is an immediate result of Corollary 2.1.18.
\end{proof}

To each graded category $\mathcal{A}$ we can associate an \textit{associated quiver} $Q$ in the following way: the vertices of $Q$ are exactly the objects in $\mathcal{A}$; if $\mathcal{A}(x, y)_1 \neq 0$, then we put an arrow from $x$ to $y$ with $x, y$ ranging over all objects in $\mathcal{A}$. Clearly, the associated quiver of $\mathcal{A}$ is completely determined by $\mathcal{A}_0$ and $\mathcal{A}_1$. There is no loop in $Q$ since $\mathcal{A} (x, x)_1 =0$ for each $x \in \Ob \mathcal{A}$.

\begin{proposition}
Let $\mathcal{A}$ be a graded category with $\mathcal{A}_0 = \bigoplus _{x \in \Ob \mathcal{A}} \mathcal{A} (x, x)$ and $Q$ be its associated quiver. Then $\mathcal{A}$ is a directed category if and only if $Q$ is an acyclic quiver.
\end{proposition}

\begin{proof}
Assume that $\mathcal{A}$ is directed. By the definition, there is a partial order $\leqslant$ on $\Ob \mathcal{A}$ such that $\mathcal{A} (x,y) \neq 0$ only if $x \leqslant y$ for $x, y \in \Ob \mathcal{A}$. In particular, $\mathcal{A}(x, y)_1 \neq 0$ only if $x < y$. Therefore, an arrow $x \rightarrow y$ exists in $Q$ only if $x < y$. If there is an oriented cycle
\begin{equation*}
x_1 \rightarrow x_2 \rightarrow \ldots \rightarrow x_n \rightarrow x_1
\end{equation*}
in $Q$, then $x_1 < x_2 < \ldots < x_n < x_1$, which is impossible. Hence $Q$ must be acyclic.\

Conversely, if $Q$ is acyclic, we then define $x \leqslant y$ if and only if there is a directed path (including trivial path with the same source and target) from $x$ to $y$ in $Q$ for $x, y \in \Ob \mathcal{A}$. This gives rise to a well defined partial order on $\Ob \mathcal{A}$. We claim that $\mathcal{A}$ is a directed category with respect to this partial order, i.e., $\mathcal{A} (x, y) \neq 0$ implies $x \leqslant y$. Since it holds trivially for $x = y$, we assume that $x \neq y$. Take a morphism $0 \neq \alpha \in \mathcal{A} (x, y)$ with a degree $n$ (this is possible since $\mathcal{A} (x,y)$ is a non-zero graded space). Since $\mathcal{A}_n = \mathcal{A}_1 \cdot \ldots \cdot \mathcal{A}_1$, we can express $\alpha$ as a linear combination of composite morphisms
\begin{equation*}
\xymatrix {x=x_0 \ar[r] ^-{\alpha_1} & x_1 \ar[r] ^-{\alpha_2} & \ldots \ar[r] ^-{\alpha_n} & x_n = y}
\end{equation*}
with each $\alpha_i \in \mathcal{A}_1$ and all $x_i$ distinct (since endomorphisms in $\mathcal{A}$ are contained in $\mathcal{A}_0$). Therefore, there is a nontrivial directed path
\begin{equation*}
x=x_0 \rightarrow  x_1 \rightarrow x_2 \rightarrow \ldots \rightarrow x_n=y
\end{equation*}
in $Q$, and we have $x < x_1 < x_2 < \ldots < y$, which proves our claim.
\end{proof}

Let $\mathcal{A}$ be a graded directed category. We define the \textit{free directed category} $\hat {\mathcal{A}}$ of $\mathcal{A}$ by using the associated quiver $Q$. Explicitly, $\hat {\mathcal{A}}$ has the same objects and endomorphisms as $\mathcal{A}$. For each pair of objects $x \neq y$ we construct $\hat{ \mathcal {C}} (x, y)$ as follows. let $\Gamma_{x,y}$ be the set of all paths from $x$ to $y$ in $Q$. In the case that $\Gamma_{x,y} = \emptyset$ we let $\hat{ \mathcal {C}} (x, y) = 0$. Otherwise, take an arbitrary path $\gamma \in \Gamma_{x,y}$ pictured as below
\begin{equation*}
x \rightarrow x_1 \rightarrow x_2 \rightarrow \ldots \rightarrow x_{n-1} \rightarrow y,
\end{equation*}
and define $(x, y)_{\gamma}$ to be
\begin{equation*}
\mathcal{A}(x_{n-1}, y)_1 \otimes _{\mathcal{A} (x_{n-1}, x_{n-1})} \mathcal{A}(x_{n-2}, x_{n-1})_1 \otimes _{\mathcal{A} (x_{n-2}, x_{n-2})} \ldots \otimes _{\mathcal{A} (x_1, x_1)} \mathcal{A}(x, x_1)_1.
\end{equation*}
Finally, we define
\begin{equation*}
\hat {\mathcal{A}} (x, y) = \bigoplus _{\gamma \in \Gamma_{x, y}} (x, y)_{\gamma}.
\end{equation*}
It is clear that $\hat{\mathcal{A}}$ is also a graded category with $\hat{\mathcal{A}}_0 = \mathcal{A}_0$ and $\hat{\mathcal{A}}_1 = \mathcal{A}_1$. Therefore, $\hat{\mathcal{A}}$ has the same associated quiver as that of $\mathcal{A}$ and is also a directed category by the above lemma. Actually, the associated graded algebra $\hat{A}$ is precisely the tensor algebra generated by $\mathcal{A}_0$ and $\mathcal{A}_1$, and $\mathcal{A}$ is a quotient category of $\hat {\mathcal{A}}$. We will prove later that a graded category $\mathcal{A}$ is standardly stratified for all linear orders if and only if it is a free directed category and $\mathcal{A}_1$ is a projective $\mathcal{A}_0$-module. The reader can also check that if two grade categories $\mathcal{A}$ and $\mathcal{D}$ have the same degree 0 and degree 1 components, then one is a directed category if and only if so is the other.

\begin{theorem}
Let $\mathcal{A}$ be a directed Koszul category and suppose that $\mathcal{A}_0$ satisfies the splitting condition (S), then the Yoneda category $\mathcal{E} = E(\mathcal{A}_0)$ is also directed and Koszul.
\end{theorem}

\begin{proof}
Applying the Koszul duality (Theorem 2.3.1) we know that $\mathcal{E}$ is a Koszul category. What we need to show is that $\mathcal{E}$ is a directed category as well. Since $\mathcal{A}$ is standardly stratified, $\pd _{\mathcal{A}} \mathcal{A} _0 < \infty$. Therefore, all morphisms in $\mathcal{E}$ form a finite-dimensional space. In particular, $\mathcal{E}$ is a locally finite $k$-linear category.

Let $\leqslant$ be the linear order on $\Ob \mathcal{A}$ with respect to which $\mathcal{A}$ is directed. This linear order gives a linear order on $\Ob \mathcal{E} = \Ob \mathcal{A}$. We claim that $\mathcal{E}$ is directed with respect to $\leqslant$, i.e., if $x \nleq y$ are two distinct objects in $\mathcal{E}$, then $\mathcal{E} (x, y) = 0$.

Since $\mathcal{E}$ is the Yoneda category of $\mathcal{A}$, $\mathcal{E} (x, y) = 1_y \mathcal{E} 1_x \cong \Ext _{\mathcal{A}} ^{\ast} (\mathcal{A}_0 1_x, \mathcal{A}_0 1_y)$. But $\mathcal{A}$ is a Koszul category, so $\mathcal{A}_i$ are projective $\mathcal{A}_0$-modules for all $i \geqslant 0$ by Theorem 3.2.1. Therefore, all $\Omega^i (\mathcal{A}_0 1_x)_i$ are projective $\mathcal{A}_0$-modules (Lemma 2.1.12), and we have (Lemma 2.1.11)
\begin{equation*}
\mathcal{E} (x, y)_i \cong \Ext _{\mathcal{A}} ^i (\mathcal{A}_0 1_x, \mathcal{A}_0 1_y) \cong \Hom _{\mathcal{A}} (\Omega^i (\mathcal{A}_0 1_x), \mathcal{A}_0 1_y).
\end{equation*}
Observe that $\mathcal{A}_0 1_y = \mathcal{A} (y, y)$ is only supported on $y$ and $y \ngeq x$. If we can prove the statement that each $\Omega^i (\mathcal{A}_0 1_x)$ is only supported on objects $z$ with $z \geqslant x$, then our claim is proved.

Clearly, $\Omega^0 (\mathcal{A}_0 1_x) = \mathcal{A}_0 1_x = \mathcal{A} (x,x)$ is only supported on $x$, so the statement is true for $i =0$. Now suppose that $\Omega^n (\mathcal{A}_0 1_x)$ is only supported on objects $z \geqslant x$ and consider $\Omega^{n+1} (\mathcal{A}_0 1_x)$. Let $S$ be the set of objects $z$ such that the value $1_z \Omega^n (\mathcal{A}_0 1_x)$ of $\Omega^n (\mathcal{A}_0 1_x)$ on $z$ is non-zero. Then we can find a short exact sequence:
\begin{equation*}
\xymatrix {0 \ar[r] & N \ar[r] & \bigoplus _{z \in S} (\mathcal{A} 1_z) ^{m_z} \ar[r]^p & \Omega^n (\mathcal{A}_0 1_x) \ar[r] & 0}
\end{equation*}
such that the map $p$ gives a surjection $p_z: (1_z \mathcal{A} 1_z) ^{m_z} \rightarrow 1_z \Omega^n (\mathcal{A}_0 1_x)$ for $z \in S$. Thus $p$ is a surjection and $\Omega^{n+1} (\mathcal{A}_0 1_x)$ is a direct summand of $N$. Notice that all $\mathcal{A} 1_z$ are supported only on objects $w \geqslant z$, and $z \geqslant x$ by the induction hypothesis. Therefore, the submodule $\Omega^{n+1} (\mathcal{A}_0 1_x) \subseteq N \subseteq \bigoplus _{z \in S} (\mathcal{A} 1_z) ^{m_z} $ is only supported on objects $w \geqslant x$. Our statement is proved by induction. This finishes the proof.\
\end{proof}

Actually, we will prove later that the associated categories of extension algebras of standard modules of standardly stratified algebras are always directed. This immediately generalizes the above theorem since by the given condition $\mathcal{A}_0$ coincides with the direct sum of standard modules. This condition is crucial. The following example shows that without it the Yoneda category $E (\mathcal{A} _0)$ might not be directed even if $\mathcal{A}$ is a Koszul directed category.

\begin{example}
Let $\mathcal{A}$ be the following category. Put an order $x < y$ on the objects and the following grading on morphisms: $\mathcal{A}_0 = \langle 1_x, 1_y, \beta \rangle$, $\mathcal{A}_1 = \langle \alpha \rangle$.
\begin{equation*}
\xymatrix{ x \ar@/^/ [rr]^{\alpha} \ar@ /_/ [rr]_{\beta} & & y.}
\end{equation*}
This category is directed obviously. It is standardly stratified (actually hereditary) with $\Delta_x \cong k_x$ and $\Delta_y \cong k_y$. By the exact sequence
\begin{equation*}
\xymatrix{ 0 \ar[r] & k_y [1] \ar[r] & \mathcal{A} \ar[r] & \mathcal{A}_0 \ar[r] & 0,}
\end{equation*}
$\mathcal{A} _0$ is a Koszul module. But $\Delta_x \oplus \Delta_y \ncong \mathcal{A} _0$. Furthermore, $\Delta_x \oplus \Delta_y$ is not Koszul since from the short exact sequence
\begin{equation*}
\xymatrix{ 0 \ar[r] & k_y [1] \oplus k_y \ar[r] & \mathcal{A} \ar[r] & \Delta_x \oplus \Delta_y \ar[r] & 0}
\end{equation*}
we find that $\Omega(\Delta_x \oplus \Delta_y)$ is not generated in degree 1.\

By computation we get the Yoneda category $\mathcal{D} = E (\mathcal{A} _0)$ pictured as below, with relation $\alpha \cdot \beta = 0$.
\begin{equation*}
\xymatrix{ x \ar@/^/ [rr]^{\alpha} & & y \ar@ /^/ [ll]^{\beta}.}
\end{equation*}
This is not a directed category with respect to the order $x <y$. However, we check: $P_y = \Delta'_y = \mathcal{D}_0 1_y$ and $\Delta'_x = k_x \cong P_x/P_y$. Therefore, $\Delta'_x \oplus \Delta'_y \cong \mathcal{D}_0$, and $\mathcal{D}$ is standardly stratified. The exact sequence
\begin{equation*}
\xymatrix{ 0 \ar[r] & P_y [1] \ar[r] & \mathcal{D} \ar[r] & \mathcal{D}_0 \ar[r] & 0}
\end{equation*}
tells us that $\mathcal{D}_0$ is a Koszul $\mathcal{D}$-module. Therefore, $\mathcal{D}$ is standardly stratified and Koszul, but not directed.
\end{example}

In the previous chapter we have described a correspondence between the generalized Koszul theory and the classical theory. For directed categories, there is another close relation between these two theories, which we describe here. Let $\mathcal{A}$ be a graded directed category. We then define a subcategory $\mathcal{D}$ of $\mathcal{A}$ by replacing all endomorphism rings in $\mathcal{A}$ by $k \cdot 1$, the span of the identity endomorphism. Explicitly, $\Ob \mathcal{D} = \Ob \mathcal{A}$; for $x, y \in \Ob \mathcal{D}$, $\mathcal{D} (x, y) = k \langle 1_x \rangle $ if $x = y$ and $ \mathcal{D} (x, y) = \mathcal{A} (x,y)$ otherwise. Clearly, $\mathcal{D}$ is also a graded directed category with $\mathcal{D}_i = \mathcal{A}_i$ for every $i \geqslant 1$. Observe that the degree 0 component $\mathcal{D}_0$ is semisimple.

\begin{theorem}
Let $\mathcal{A}$ be a graded directed Koszul category and define the subcategory $\mathcal{D}$ as above. If $M$ is a Koszul $\mathcal{A}$-module, then $M \downarrow _{\mathcal{D}} ^{\mathcal{A}}$ is a Koszul $\mathcal{D}$-module. In particular, $\mathcal{D}$ is a Koszul category in the classical sense.
\end{theorem}

\begin{proof}
We prove the conclusion by induction on the size of $\Ob \mathcal{A}$. If the size of $\Ob \mathcal{A}$ is 1, the conclusion holds trivially. Now suppose that the conclusion is true for categories with at most $n$ objects and let $\mathcal{A}$ be a graded directed category with $n+1$ objects. Take $x$ to be a minimal object in $\mathcal{A}$ and define $\mathcal{A}_x$ ($\mathcal{D}_x$, resp.) to be the full subcategory of $\mathcal{A}$ ($\mathcal{D}$, resp.) formed by removing $x$ from it. Clearly $\mathcal{A}_x$ and $\mathcal{D}_x$ have $n$ objects.\

The following fact, which is well known in the context of finite EI categories (see \cite{Xu1, Xu2}), is essential in the proof.\\

\textbf{Fact:} Every graded $\mathcal{A}_x$-module $N$ can be viewed as a $\mathcal{A}$-module with $N(x) =0$ by induction. Conversely, every graded $\mathcal{A}$-module $M$ with $M(x) = 0$ can also be viewed as a $\mathcal{A}_x$-module by restriction. Furthermore, if $M(x) = 0$, then $\Omega^i (M)(x) = 0$ for all $i \geqslant 0$. The above induction and restriction preserves projective modules: a projective $\mathcal{A}_x$-module is still projective when viewed as a $\mathcal{A}$-module; conversely, a projective $\mathcal{A}$-module $P$ with $P(x) =0$ is still projective viewed as a $\mathcal{A}_x$-module. Therefore, a graded $\mathcal{A}$-module $M$ with $M(x)=0$ is Koszul if and only if it is Koszul as a $\mathcal{A}_x$-module. All these results hold for the pair $(\mathcal{D}, \mathcal{D}_x)$ similarly.\\

By this fact, we only need to handle Koszul $\mathcal{A}$-modules $M$ with $M(x) \neq 0$. Indeed, if $M(x) = 0$, then $M$ is also Koszul regarded as a $\mathcal{A}_x$-module. By the induction hypothesis, $M \downarrow _{\mathcal{D}_x} ^{\mathcal{A}_x}$ is a Koszul $\mathcal{D}_x$-module. By the above fact, $M \downarrow _{\mathcal{D}} ^{\mathcal{A}}$ is a Koszul $\mathcal{D}$-module. Thus the conclusion is true for Koszul $\mathcal{A}$-modules $M$ with $M(x) =0$.\

Firstly we consider the special case $M = \mathcal{A}_0 1_x = \mathcal{A} (x,x)$ which is concentrated on $x$ when viewed as a $\mathcal{A}$-module. It is clear that
\begin{equation*}
\Omega (\mathcal{A}_0 1_x) \downarrow _{\mathcal{D}} ^{\mathcal{A}} = \Omega (\mathcal{D}_0 1_x)
\end{equation*}
as vector spaces since for each pair $u \neq v \in \Ob \mathcal{A}$, $\mathcal{A} (u, v) = \mathcal{D} (u, v)$, and
\begin{equation*}
\mathcal{A}_0 1_x \downarrow _{\mathcal{D}} ^{\mathcal{A}} \cong (\mathcal{D}_0 1_x)^m \cong k_x^m,
\end{equation*}
where $m = \dim _k \mathcal{A} (x, x)$. Since $\mathcal{A}_01_x$ is a Koszul $\mathcal{A}$-module, $\Omega (\mathcal{A}_0 1_x) [-1]$ is a Koszul $\mathcal{A}$-module supported on $\Ob \mathcal{A}_x$. By the induction hypothesis, $\Omega(\mathcal{D}_0 1_x) [-1] = \Omega (\mathcal{A}_0 1_x) \downarrow _{\mathcal{D}} ^{\mathcal{A}} [-1]$ is a Koszul $\mathcal{D}_x$-module, and hence a Koszul $\mathcal{D}$-module. Therefore, $\mathcal{D}_0 1_x$, and hence $\mathcal{A}_0 1_x \downarrow _{\mathcal{D}} ^{\mathcal{A}} \cong (\mathcal{D}_0 1_x)^m$ are Koszul $\mathcal{D}$-modules. In the case that $y \neq x$, $\mathcal{D}_0 1_y$ is a direct summand of $\mathcal{A}_0 1_y \downarrow _{\mathcal{D}} ^{\mathcal{A}}$. It is Koszul viewed as a $\mathcal{D}_x$-module by the induction hypothesis, and hence is a Koszul $\mathcal{D}$-module. Consequently, $\mathcal{D}_0$ is a Koszul $\mathcal{D}$-module, so $\mathcal{D}$ is a Koszul category in the classical sense.\

Now let $M$ be an arbitrary Koszul $\mathcal{A}$-module with $M(x) \neq 0$. Consider the exact sequence
\begin{equation}
\xymatrix {0 \ar[r] & \Omega M \downarrow _{\mathcal{D}} ^{\mathcal{A}} \ar[r] & P \downarrow _{\mathcal{D}} ^{\mathcal{A}} \ar[r]^p & M \downarrow _{\mathcal{D}} ^{\mathcal{A}} \ar[r] & 0}
\end{equation}
induced by
\begin{equation*}
\xymatrix {0 \ar[r] & \Omega M \ar[r] & P \ar[r]^p & M \ar[r] & 0.}
\end{equation*}
The structures of $\mathcal{A}$ and $\mathcal{D}$ give the following exact sequence:
\begin{equation*}
\xymatrix {0 \ar[r] & _{\mathcal{D}} \mathcal{D} \ar[r] & _{\mathcal{D}} \mathcal{A} \ar[r] & \bigoplus _{x \in \Ob \mathcal{A}} k_x^{m_x} \ar[r] & 0}
\end{equation*}
where $k_x \cong \mathcal{D}_01_x$ and $m_x = \dim _k \mathcal{A} (x, x) - 1$. Since $P \downarrow _{\mathcal{D}} ^{\mathcal{A}} \in \text{add} (_{\mathcal{D}} \mathcal{A})$, the above sequence gives us a corresponding sequence for $P \downarrow _{\mathcal{D}} ^{\mathcal{A}}$:
\begin{equation}
\xymatrix {0 \ar[r] & P' \ar[r]^{\iota} & P \downarrow _{\mathcal{D}} ^{\mathcal{A}} \ar[r] & T \ar[r] & 0,}
\end{equation}
here $P'$ is a projective $\mathcal{D}$-module and $T \in \text{add} (\mathcal{D}_0)$. Both of them are generated in degree 0. \

Putting sequences (3.2.1) and (3.2.2) together we get
\begin{equation*}
\xymatrix{ 0 \ar[r] & \Omega (M') \ar[r] \ar[d]^{\varphi} & P' \ar[r]^{p \circ \iota} \ar[d]^{\iota} & M' \ar[d] \ar[r] & 0 \\
0 \ar[r] & \Omega (M) \downarrow _{\mathcal{D}} ^{\mathcal{A}} \ar[r] & P \downarrow _{\mathcal{D}} ^{\mathcal{A}} \ar[r]^p \ar[d] & M \downarrow _{\mathcal{D}} ^{\mathcal{A}} \ar[r] \ar[d] & 0 \\
& & T \ar[r] &  M \downarrow _{\mathcal{D}} ^{\mathcal{A}} /M'
}
\end{equation*}
where $M' = (p \circ \iota) (P')$. Notice that $p$ and $\iota$ both are injective restricted to degree 0 components. Therefore $p \circ \iota$ is also injective restricted to the degree 0 component of $P'$ (actually it is an isomorphism restricted to the degree 0 component). Consequently, $P'$ is a projective cover of $M'$ and the kernel of $p \circ \iota$ is indeed $\Omega (M')$.\

We claim that $\varphi$ is an isomorphism and hence $\Omega(M') \cong (\Omega M) \downarrow _{\mathcal{D}} ^{\mathcal{A}}$. It suffices to show $T \cong M \downarrow _{\mathcal{D}} ^{\mathcal{A}} /M'$ by the snake lemma. First, since $T$ is concentrated in degree 0 in sequence (3.1.2), $(P \downarrow _{\mathcal{D}} ^{\mathcal{A}} )_i = P'_i$ and
\begin{equation*}
M'_i = (p \circ \iota) (P'_i) = p ((P \downarrow _{\mathcal{D}} ^{\mathcal{A}} )_i) = (M \downarrow _{\mathcal{D}} ^{\mathcal{A}}) _i, i >0.
\end{equation*}
Therefore, $ M \downarrow _{\mathcal{D}} ^{\mathcal{A}} /M'$ is concentrated in degree 0. Furthermore, since $M$ is Koszul, $(P \downarrow _\mathcal{D} ^{\mathcal{A}} )_0 = P_0 = M_0 = (M \downarrow _{\mathcal{D}} ^{\mathcal{A}})_0$, and $P'_0 = M'_0$ because $p\circ \iota$ restricted to $P_0'$ is an isomorphism as well. We deduce that
\begin{equation*}
T \cong (P \downarrow _\mathcal{D} ^{\mathcal{A}} )_0 / P'_0 \cong (M \downarrow _\mathcal{D} ^{\mathcal{A}} )_0 / M'_0 = M \downarrow _{\mathcal{D}} ^{\mathcal{A}} /M',
\end{equation*}
exactly as we claimed.\

Now consider the rightmost column of the above diagram. Clearly, the bottom term $M \downarrow _{\mathcal{D}} ^{\mathcal{A}} /M' \cong T \in \text{add} (\mathcal{D}_0)$ is Koszul since we just proved that $\mathcal{D}_0$ is Koszul. The $\mathcal{A}$-module $(\Omega M) [-1]$ is Koszul since $M$ is supposed to be Koszul. Moreover, because $x$ is minimal and $M$ is generated in degree 0, $M(x) \subseteq M_0$ and hence $(\Omega M) [-1] (x) = (\Omega M)(x) = 0$. Therefore, $(\Omega M) [-1]$ is a Koszul $\mathcal{A}$-module supported on $\Ob \mathcal{A}_x$, so it is also a Koszul $\mathcal{A}_x$-module. By the induction hypothesis, $(\Omega M') [-1] \cong (\Omega M) [-1] \downarrow _{\mathcal{D}} ^{\mathcal{A}}$ is Koszul viewed as a $\mathcal{D}_x$-module, and hence Koszul as a $\mathcal{D}$-module. Thus the top term $M'$ is a Koszul $\mathcal{D}$-module since as a homomorphic image of $P'$ (which is generated in degree 0) it is generated in degree 0 as well. By Proposition 2.1.8, $M \downarrow _{\mathcal{D}} ^{\mathcal{A}}$ is also a Koszul $\mathcal{D}$-module since $\mathcal{D}_0$ is semisimple by our construction. The conclusion follows from induction.\
\end{proof}

The converse of the above theorem is also true.

\begin{theorem}
Let $\mathcal{A}$ be a graded directed category and construct the subcategory $\mathcal{D}$ as before. Suppose that $\mathcal{D}$ is Koszul in the classical sense. Let $M$ be a graded $\mathcal{A}$-module generated in degree 0 such that $\Omega^i(M)_i$ are projective $\mathcal{A}_0$-modules for all $i \geqslant 0$. Then $M$ is a Koszul $\mathcal{A}$-module whenever $M \downarrow _{\mathcal{D}} ^{\mathcal{A}}$ is a Koszul $\mathcal{D}$-module.
\end{theorem}

\begin{proof}
We use the similar technique to prove the conclusion. Notice that we always assume that $\mathcal{A}_0 = \bigoplus _{x \in \Ob \mathcal{A}} \mathcal{A} (x,x)$. If $\mathcal{A}$ has only one object, then Koszul modules are exactly projective modules generated in degree 0 and the conclusion holds. Suppose that it is true for categories with at most $n$ objects. Let $\mathcal{A}$ be a category of $n+1$ objects and take a minimal object $x$. Define $\mathcal{A}_x$ and $\mathcal{D}_x$ as before. As in the proof of last theorem, a graded $\mathcal{A}$-module $M$ with $M(x) =0$ is Koszul if and only if it is Koszul viewed as a $\mathcal{A}_x$-module by restriction. and the same result holds for the pair $(\mathcal{D}, \mathcal{D}_x)$. In particular, $\mathcal{D}_x$ is a Koszul category. Therefore, we only need to show that an arbitrary graded $\mathcal{A}$-module $M$ which is generated in degree 0 and satisfies the following conditions is Koszul: $\Omega^i(M)_i$ is a projective $\mathcal{A}_0$-module for each $i \geqslant 0$; $M(x) \neq 0$; and $M \downarrow _{\mathcal{D}} ^{\mathcal{A}}$ is a Koszul $\mathcal{D}$-module.\

Let $M$ be such a $\mathcal{A}$-module and consider the commutative diagram:

\begin{equation}
\xymatrix{ & K \ar@{=}[r] \ar[d] & K \ar[d] \ar[r] & 0 \ar[d] \\
0 \ar[r] & \Omega(M \downarrow _{\mathcal{D}} ^{\mathcal{A}}) \ar[d] \ar[r] & P \ar[r] \ar[d]^{\varphi} & M \downarrow _{\mathcal{D}} ^{\mathcal{A}} \ar[r] \ar@{=}[d] & 0 \\
0 \ar[r] & (\Omega M) \downarrow _{\mathcal{D}} ^{\mathcal{A}} \ar[r] & \tilde{P} \downarrow _{\mathcal{D}} ^{\mathcal{A}} \ar[r] & M \downarrow _{\mathcal{D}} ^{\mathcal{A}} \ar[r] & 0}
\end{equation}
where $P$ and $\tilde{P}$ are projective covers of $M \downarrow _{\mathcal{D}} ^{\mathcal{A}}$ and $M$ respectively. Since $M_0$ is a projective $\mathcal{A}_0$-module, $P_0 = M_0 = (M \downarrow _{\mathcal{D}} ^{\mathcal{A}} )_0 =(\tilde{P} \downarrow _{\mathcal{D}} ^{\mathcal{A}} )_0$ as vector spaces, and the induced map $\varphi$ restricted to $P_0$ is an isomorphism. Therefore $\varphi$ is surjective since both $P$ and $\tilde{P} \downarrow _{\mathcal{D}} ^{\mathcal{A}}$ are generated in degree 0. Let $K$ be the kernel of $\varphi$.\

We have the following exact sequence similar to sequence (3.2.2):
\begin{equation*}
\xymatrix{ 0 \ar[r] & P' \ar[r] & \tilde{P} \downarrow _{\mathcal{D}} ^{\mathcal{A}} \ar[r] ^{\tilde{p}} & T \ar[r] & 0, }
\end{equation*}
where $P'$ is a projective $\mathcal{D}$-module such that $P'_i = (\tilde{P} \downarrow _{\mathcal{D}} ^{\mathcal{A}} )_i$ for every $i \geqslant 1$, and $T \in \text{add} (\mathcal{D}_0)$.\

Let $P''$ be a projective cover of $T$ (as a $\mathcal{D}$-module). Then we obtain:

\begin{equation}
\xymatrix{ & & K \ar@{=}[r] \ar[d] & K \ar[d] \\
0 \ar[r] & P' \ar@{=}[d] ^{\alpha} \ar[r] & P \ar[r]^p \ar[d]^{\varphi} & P'' \ar[d]^{p''} \ar[r] & 0 \\
0 \ar[r] & P' \ar[r] & \tilde{P} \downarrow _{\mathcal{D}} ^{\mathcal{A}} \ar[r] ^{\tilde{p}} & T \ar[r] & 0}.
\end{equation}
We give some explanations here. Since $P$ is a projective $\mathcal{D}$-module and the map $p''$ is surjective, the map $\tilde{p} \circ \varphi$ factors through $p''$ and gives a map $p: P \rightarrow P''$. Restricted to degree 0 components, $p''$ and $\varphi$ (see diagram (3.2.3)) are isomorphisms and $\tilde{p}$ is surjective. Thus $p$ restricted to the degree 0 components is also surjective. But $P''$ is generated in degree 0, so $p$ is surjective. Since $P_0 = (\tilde{P} \downarrow _{\mathcal{D}} ^{\mathcal{A}} )_0$ and $P''_0 = T_0 = T$, $\alpha$ restricted to the degree 0 components is an isomorphism, and hence an isomorphism of projective $k\mathcal{D}$-modules (notice that the middle row splits since $P''$ is a projective $\mathcal{D}$-module, so the kernel should be a projective $\mathcal{D}$-module generated in degree 0). By the snake Lemma, the kernel of $p''$ is also $K$ up to isomorphism.\

Let $J = \bigoplus _{i \geqslant 1} \mathcal{D}_i$. Since $\mathcal{D}_0$ is supposed to be a Koszul $\mathcal{D}$-module, $J [-1] \cong \Omega (\mathcal{D} _0) [-1]$ is a Koszul $\mathcal{D}$-module, too. Consider the leftmost column in diagram 3.2.3. The top term $K[-1]$ is a Koszul $\mathcal{D}$-module since $K \cong P'' / T \cong P''/P''_0 \in \text{add} (J)$. The middle term $\Omega (M \downarrow _{\mathcal{D}} ^{\mathcal{A}}) [-1]$ is Koszul as well since $M \downarrow _{\mathcal{D}} ^{\mathcal{A}}$ is supposed to be Koszul. By Proposition 2.1.8, the bottom term $(\Omega M)[-1] \downarrow _{\mathcal{D}} ^{\mathcal{A}}$ must be Koszul.

Since $M$ is generated in degree 0 and $x$ is a minimal object, $M(x) \subseteq M_0$, so $M(x) = (M \downarrow _{\mathcal{D}} ^{\mathcal{A}}) (x) \subseteq (M \downarrow _{\mathcal{D}} ^{\mathcal{A}}) _0$ as vector spaces. Similarly, $P(x) \subseteq P_0$ and $P(x) \cong (M \downarrow _{\mathcal{D}} ^{\mathcal{A}}) (x)$, so $\Omega (M \downarrow _{\mathcal{D}} ^{\mathcal{A}}) (x) = 0$. Consequently, $(\Omega M)[-1]$ is supported on $\Ob \mathcal{A}_x$ by observing the leftmost column of diagram (3.2.3). Moreover, we can show as in the proof of Theorem 3.2.3 that all of its syzygies are supported on $\Ob \mathcal{A}_x$, and $\Omega^i ((\Omega M)[-1])_i = \Omega^{i+1} (M)_{i+1}$ are projective $(\mathcal{A}_x)_0$-modules. Therefore, applying the induction hypothesis to $(\Omega M)[-1]$ supported on $\Ob \mathcal{A}_x$ and $(\Omega M) [-1] \downarrow _{\mathcal{D}} ^{\mathcal{A}}$ supported on $\Ob \mathcal{D}_x$, we conclude that $(\Omega M) [-1]$ is a Koszul $\mathcal{A} _x$-module, and hence a Koszul $\mathcal{A}$-module. Clearly, $M$ is a Koszul $\mathcal{A}$-module since it is generated in degree 0. The conclusion follows from induction.
\end{proof}

\begin{remark}
We remind the reader that in the previous two theorems we do not assume that $\mathcal{A}_0 = \bigoplus _{x \in \Ob \mathcal{A}} \mathcal{A} (x,x)$ satisfies the splitting condition (S). On the other hand, by our construction, $\mathcal{D}_0 \cong \bigoplus _{x \in \Ob \mathcal{D}} k_x$ is a semisimple algebra.
\end{remark}

Assuming that $\mathcal{A}_0$ indeed has the splitting property, we get the following nice correspondence.

\begin{theorem}
Let $\mathcal{A}$ be a graded directed category and suppose that $\mathcal{A}_0$ has the splitting property (S). Construct $\mathcal{D}$ as before. Then:
\begin{enumerate}
\item $\mathcal{A}$ is a Koszul category in our sense if and only if $\mathcal{A}$ is standardly stratified and $\mathcal{D}$ is a Koszul category in the classical sense.
\item If $\mathcal{A}$ is a Koszul category, then a graded $\mathcal{A}$-module $M$ generated in degree 0 is Koszul if and only if $M \downarrow _{\mathcal{D}} ^{\mathcal{A}}$ is a Koszul $\mathcal{D}$-module and $M$ is a projective $\mathcal{A}_0$-module.
\end{enumerate}
\end{theorem}

\begin{proof}
If $\mathcal{A}$ is Koszul in our sense, then it is standardly stratified by (2) of Theorem 3.2.1, and $\mathcal{D}$ is Koszul in the classical sense by Theorem 3.2.5. Conversely, if $\mathcal{D}$ is Koszul in the classical sense, then $\mathcal{A}_0 \downarrow _{\mathcal{D}} ^{\mathcal{A}} \in \text{add} (\mathcal{D}_0)$ is a Koszul $\mathcal{D}$-module. If $\mathcal{A}$ is furthermore standardly stratified, then it is a projective $\mathcal{A}_0$-module by Theorem 3.2.1. Therefore, all $\Omega^i( \mathcal{A}_0 )_i$ are projective $\mathcal{A}_0$-modules according to Lemma 2.1.12. Thus $\mathcal{A}_0$ is a Koszul $\mathcal{A}$-module by Theorem 3.2.6, and hence $\mathcal{A}$ is a Koszul category. This proves the first statement.

Now suppose that $\mathcal{A}$ is Koszul. Then $\mathcal{A}$ is a projective $\mathcal{A}_0$-module. If $M$ is a Koszul $\mathcal{A}$-module, $M \downarrow _{\mathcal{D}} ^{\mathcal{A}}$ is a Koszul $\mathcal{D}$-module by Theorem 3.2.5. Furthermore, $M$ is a projective $\mathcal{A}_0$-module by Corollary 2.1.18. Conversely, if $M$ is a projective $\mathcal{A}_0$-module and $M \downarrow _{\mathcal{D}} ^{\mathcal{A}}$ is a Koszul $\mathcal{D}$-module, then by Lemma 2.1.12 $\Omega^i (M)_i$ are projective $\mathcal{A}_0$-modules for all $i \geqslant 0$. By Theorem 3.2.6. $M$ is a Koszul $\mathcal{A}$-module.
\end{proof} 
\chapter{Applications to finite EI categories}
\label{Applications to finite EI categories}

In this chapter we study the application of our generalized Koszul theory to \textit{finite EI categories}. They include finite groups and posets as special examples, and the $k$-linearizations of (skeletal) finite EI categories are directed categories. The application of classical Koszul theory to incidence algebras of finite poset has been discussed in \cite{Reiner}.

In the first section we give some background on finite EI categories and some preliminary results. In the second section we study in details \textit{finite free EI categories} and their representations. Koszul properties of finite EI categories are studied in the last section.

Through this chapter all finite EI categories are \textit{skeletal}.

\section{Preliminaries}

For the reader's convenience, we include in this section some background on the representation theory of finite EI categories. Please refer to \cite{Webb2, Webb3, Xu1, Xu2} for more details.

A \textit{finite EI category} $\mathcal{E}$ is a small category with finitely many morphisms such that every endomorphism in $\mathcal{E}$ is an isomorphism. Examples include finite groups (viewed as categories with one object) and finite posets (all endomorphisms are identities). The category $\mathcal{E}$ is \textit{connected} if for any two distinct objects $x$ and $y$, there is a list of objects $x=x_0, x_1, \ldots, x_n=y$ such that either $\Hom_{\mathcal{E}} (x_i, x_{i+1})$ or $\Hom_{\mathcal{E}} (x_{i+1}, x_i)$ is not empty, $0 \leqslant i \leqslant n-1$. Every finite EI category is a disjoint union of connected components, and each component is a full subcategory.

A \textit{representation} of $\mathcal{E}$ is a functor $R$ from $\mathcal{E}$ to the category of finite-dimensional $k$-vector spaces. The functor $R$ assigns a vector space $R(x)$ to each object $x$ in $\mathcal{E}$, and a linear transformation $R(\alpha): R(x) \rightarrow R(y)$ to each morphism $\alpha: x \rightarrow y$ such that all composition relations of morphisms in $\mathcal{E}$ are preserved under $R$. A \textit{homomorphism} $\varphi: R_1 \rightarrow R_2$ of two representations is a natural transformation of functors.

A finite EI category $\mathcal{E}$ determines a finite-dimensional $k$-algebra $k\mathcal{E}$ called the \textit{category algebra}. It has basis all morphisms of $\mathcal{E}$, and the multiplication is defined by the composition of morphisms (the composition is 0 when two morphisms cannot be composed) and bilinearity. By Theorem 7.1 of \cite{Mitchell}, a representation of $\mathcal{E}$ is equivalent to a $k\mathcal{E}$-module. Thus we do not distinguish these two concepts throughout this chapter.

By Proposition 2.2 in \cite{Webb2}, if $\mathcal{E}$ and $\mathcal{D}$ are equivalent finite EI categories, $k \mathcal{E}$-mod is Morita equivalent to $k \mathcal{D}$-mod. Moreover If $\mathcal{E} = \bigsqcup _{i=1}^{m} \mathcal{E}_i$ is a disjoint union of several full categories, the category algebra $k\mathcal{E}$ has an algebra decomposition $k\mathcal{E}_1 \oplus \ldots \oplus k\mathcal{E}_m$.  Thus it is sufficient to study the representations of connected, skeletal finite EI categories. We make the following convention:\\

\noindent \textbf{Convention:} \textit{All finite EI categories in this chapter are \textbf{connected} and \textbf{skeletal}. Thus endomorphisms, isomorphisms and automorphisms in a finite EI category coincide.}\\

Under the hypothesis that $\mathcal{E}$ is skeletal, if $x$ and $y$ are two distinct objects in $\mathcal{E}$ with $\Hom_{\mathcal{E}} (x,y)$ non-empty, then $\Hom_{\mathcal{E}} (y,x)$ is empty. Indeed, if this is not true, we can take $\alpha \in \Hom_{\mathcal{E}} (y, x)$ and $\beta \in \Hom_{\mathcal{E}} (x,y)$. The composite $\beta \alpha$ is an endomorphism of $y$, hence an automorphism. Similarly, the composite $\alpha \beta$ is an automorphism of $x$. Thus both $\alpha$ and $\beta$ are isomorphisms, so $x$ is isomorphic to $y$. But this is impossible since $\mathcal{E}$ is skeletal and $x \neq y$.

\section{Finite free EI categories}

The content of this section comes from \cite{Li1}. Please refer to that paper for more details.

Before defining \textit{finite free EI categories}, we need to introduce \textit{finite EI quivers}, which are finite quivers with extra structure.

\begin{definition}
A finite EI quiver $\hat{Q}$ is a datum $(Q_0, Q_1, s, t, f, g)$, where: $(Q_0, Q_1, s, t)$ is a finite acyclic quiver with vertex set $Q_0$, arrow set $Q_1$, source map $s$ and target map $t$. The map $f$ assigns a finite group $f(v)$ to each vertex $v \in Q_0$; the map $g$ assigns an $(f(t(\alpha)), f(s(\alpha)))$-biset to each arrow $\alpha \in Q_1$.
\end{definition}

If $f$ assigns the trivial group to each vertex in $Q_0$ in the above definition, we obtain a quiver in the usual sense. In this sense, finite acyclic quivers are special cases of finite EI quivers.

Each finite EI quiver $\hat{Q} = (Q_0, Q_1, s, t, f, g)$ determines a finite EI category $\mathcal{C_{\hat{Q}}}$ in the following way: the objects in $\mathcal{C_{\hat{Q}}}$ are precisely the vertices in $Q_0$. For a particular object $v$ in $\mathcal{C_{\hat{Q}}}$, we define $\Aut_{\mathcal{C_{\hat{Q}}}} (v) = f (v)$, which is a finite group by our definition. It remains to define $\Hom_{\mathcal{C_{\hat{Q}}}} (v,w)$ if $v \neq w$ are distinct vertices in $Q_0$, and the composition of morphisms.

Let $\xymatrix{v \ar@{~>}[r]^{\gamma} & w}$ be a directed path from $v$ to $w$. Then $\gamma$ can be written uniquely as a composition of arrows, where $v_i \in Q_0$ and $\alpha_i \in Q_1$ for $i=1, \ldots, n$.
\begin{equation*}
\xymatrix{v=v_0 \ar[r]^{\alpha_1} & v_1 \ar[r]^{\alpha_2} & \ldots \ar[r]^{\alpha_n} & v_n=w}
\end{equation*}
Notice that $g(\alpha_i)$ is an $(f(v_i), f(v_{i-1}))$-biset, so we define:
\begin{equation*}
H_{\gamma}=g(\alpha_n) \times_{f(v_{n-1})} g(\alpha_{n-1}) \times_{f(v_{n-2})} \ldots \times_{f(v_1)} g(\alpha_1),
\end{equation*}
the biset product defined in \cite{Webb0}. Finally, $\Hom_{\mathcal{C_{\hat{Q}}}} (v,w)$ can be defined as $\bigsqcup _{\gamma} H_{\gamma}$, the disjoint union of all $H_{\gamma}$, over all possible paths $\gamma$ from $x$ to $y$. In the case $v = w$ we define $\Hom_{\mathcal{C_{\hat{Q}}}} (v,v) = f(v)$.

Let $\alpha$ and $\beta$ be two morphisms in $\mathcal{C_{\hat{Q}}}$. They lie in two sets $H_{\gamma_1}$ and $H_{\gamma_2}$, where $\gamma_1$ and $\gamma_2$ are two paths determined by $\alpha$ and $\beta$ respectively, possibly of length 0. Their composite $\beta \circ \alpha$ can be defined by the following rule:\ it is 0 if the composite $\gamma_2 \gamma_1$ is not defined in $\hat{Q}$. Otherwise, the initial vertex $v$ of $\gamma_2$ is exactly the terminal vertex of $\gamma_1$. Since there is a natural surjective map $p: H_{\gamma_2} \times H_{\gamma_1} \rightarrow H_{\gamma_2} \times_{f(v)} H_{\gamma_1}$, we define $\beta \circ \alpha = p(\beta, \alpha)$, the image of $(\beta, \alpha)$ in $H_{\gamma_2} \times_{f(v)} H_{\gamma_1}$. This definition satisfies the associative rule,
and in this way we get a finite EI category $\mathcal{C_{\hat{Q}}}$ from $\hat{Q}$.

\begin{definition}
A finite EI category $\mathcal{E}$ is a finite free EI category \footnote{The terminology ``free EI category'' has previously been introduced with different meaning in \cite{Luck}. See Definition 16.1 on page 325. By L\"{u}ck's definition, a finite EI category $\mathcal{E}$ is a free EI category if $\Aut_{\mathcal{E}} (y)$ acts freely on $\Hom _{\mathcal{E}} (x, y)$ for all $x, y \in \Ob (\mathcal{E})$.} if it is isomorphic to the finite EI category $\mathcal{C_{\hat{Q}}}$ generated from a finite EI quiver $\hat{Q}$ by the above construction.
\end{definition}

In practice it is inconvenient to check whether a finite EI category $\mathcal{E}$ is free or not by using the definition. Fortunately, there is an equivalent characterization built upon unfactorizable morphisms: the Unique Factorization Property (UFP).

\begin{definition}
A morphism $\alpha: x \rightarrow z$ in a finite EI category $\mathcal{E}$ is unfactorizable if $\alpha$ is not an isomorphism and whenever it has a factorization as a composite $\xymatrix{x \ar[r]^{\beta} & y\ar[r]^{\gamma} & z}$, then either $\beta$ or $\gamma$ is an isomorphism.
\end{definition}

The reader may want to know the relation between the terminology \textit{unfactorizable morphism} and the term \textit{irreducible morphism} which is widely accepted and used, for example in \cite{Auslander, Bautista, Xu1}. Indeed, in this chapter they coincide since we only deal with finite EI categories. But in a more general context, they are different, as we explain in the following example:

\begin{example}
Consider the following category $\mathcal{E}$ with two objects $x \ncong y$. The non-identity morphisms in $\mathcal{E}$ are generated by $\alpha: x \rightarrow y$ and $\beta: y \rightarrow x$ with the only nontrivial relation $\beta \alpha = 1_x$. Then the morphisms in $\mathcal{E}$ are $1_x, 1_y, \alpha, \beta \text{ and } \alpha\beta$. It is not a finite EI category since $\alpha \beta \in \End _{\mathcal{E}} (y)$ is not an isomorphism. Then neither $\alpha$ nor $\beta$ are irreducible morphisms since one of them is a split monomorphism and the other is a split epimorphism. However, the reader can check that they are unfactorizable morphisms.
\end{example}

\begin{equation*}
\xymatrix{ x \ar@(ul,dl)_{1_x} \ar@/^/[rr] ^{\alpha} & & y \ar@/^/[ll] ^{\beta} \ar@(ur,dr) ^{1_y}}
\end{equation*}

Note that the composite of an unfactorizable morphism with an isomorphism is still unfactorizable.

\begin{proposition}
Let $\alpha: x \rightarrow y$ be an unfactorizable morphism. Then $h\alpha g$ is also unfactorizable for every $h \in \Aut_{\mathcal{E}}(y)$ and every $g \in \Aut_{\mathcal{E}}(x)$.
\end{proposition}

\begin{proof}
Fix a decomposition $h\alpha g=\alpha_1 \alpha_2$. Then we have $\alpha = (h^{-1} \alpha_1) (\alpha_2 g^{-1})$. But $\alpha$ is unfactorizable, so by definition either one of $h^{-1}\alpha_1$ and $\alpha_2 g^{-1}$ is an isomorphism. Without loss of generality, let $h^{-1}\alpha_1$ be an isomorphism. Then $\alpha_1$ is an isomorphism since $h^{-1}$ is an automorphism.
\end{proof}

Let $\mathcal{E}$ be a finite EI category. By the previous proposition, the set of unfactorizable morphisms from an object $x$ to another object $y$ is closed under the actions of $\Aut_{\mathcal{E}}(x)$ and $\Aut_{\mathcal{E}}(y)$. Choose a fixed representative for each ($\Aut_{\mathcal{E}} (y)$,$\Aut_{\mathcal{E}} (x)$)-orbit. Repeating this process for all pairs of different objects $(x,y)$, we get a set $A =\{ \alpha_1, \ldots, \alpha_n\}$ of orbit representatives. Elements in $A$ are called \textit{representative unfactorizable morphisms}.

We should point out here that each finite EI category $\mathcal{E}$ determines a finite EI quiver $\hat{Q}$ in the following way: its vertices are objects in $\mathcal{E}$; we put an arrow $x \rightarrow y$ in $\hat{Q}$ for each representative unfactorizable morphism $\alpha: x \rightarrow y$ in $\mathcal{E}$. Thus the arrows biject with all representative unfactorizable morphisms $\alpha: x \rightarrow y$ in $\mathcal{E}$, or equivalently, all $\Aut _{\mathcal{E}} (y) \times \Aut _{\mathcal{E}} (x)$-orbits of unfactorizable morphisms in $\mathcal{E}$. The map $f$ assigns $\Aut_{\mathcal{E}} (x)$ to each object $x$; the map $g$ assigns the ($\Aut_{\mathcal{E}} (y)$, $\Aut_{\mathcal{E}} (x)$)-biset where a representative unfactorizable morphism $\alpha: x \rightarrow y$ lies to the corresponding arrow. Obviously, this finite EI quiver is unique up to isomorphism. We call this quiver the \textit{finite EI quiver} of $\mathcal{E}$.

Now suppose that $\mathcal{E}$ is a finite free EI category. It is possible that there is more than one finite EI quiver generating $\mathcal{E}$, although they must have the same vertices. However, it is not hard to see that all those finite EI quivers are subquivers of the finite EI quiver of $\mathcal{E}$.

All non-isomorphisms can be written as composites of unfactorizable morphisms.
\begin{proposition}
Let $\alpha: x \rightarrow y$ be a morphism with $x \neq y$. Then it has a decomposition \xymatrix{x=x_0 \ar[r]^{\alpha_1} & x_1 \ar[r]^{\alpha_2} & \ldots \ar[r]^{\alpha_n} & x_n=y}, where all $\alpha_i$ are unfactorizable.
\end{proposition}

\begin{proof}
If $\alpha$ is unfactorizable, we are done. Otherwise, we can find a decomposition for $\alpha$: \xymatrix{x \ar[r]^{\alpha_1} & x_1 \ar[r]^{\alpha_2} & y} where neither $\alpha_1$ nor $\alpha_2$ is an isomorphism. In particular, $x_1$ is different from $x$ and $y$. If both $\alpha_1$ and $\alpha_2$ are unfactorizable, we are done. Otherwise, assume $\alpha_1$ is not unfactorizable. Repeating the above process, we can get a decomposition \xymatrix{x \ar[r] ^{\alpha_{11}} & x_{11} \ar[r]^{\alpha_{12}} & x_1 \ar[r]^{\alpha_2} & y}. With the same reasoning, $x, x_{11}, x_1, y$ are pairwise different. Since there are only finitely many objects, this process ends after finitely many steps. Therefore we get a decomposition of $\alpha$ into unfactorizable morphisms.
\end{proof}

For an arbitrary finite EI category $\mathcal{E}$, the ways to decompose a non-isomorphism into unfactorizable morphisms need not to be unique. However, we can show that for finite free EI categories, this decomposition is unique up to a trivial relation, i.e., they satisfy the property defined below:
\begin{definition}
A finite EI category $\mathcal{E}$ satisfies the Unique Factorization Property (UFP) if whenever a non-isomorphism $\alpha$ has two decompositions into unfactorizable morphisms:
\begin{equation*}
\xymatrix{x=x_0 \ar[r]^{\alpha_1} & x_1 \ar[r]^{\alpha_2} & \ldots \ar[r]^{\alpha_m} &x_m=y \\
x=x_0 \ar[r]^{\beta_1} & y_1 \ar[r]^{\beta_2} & \ldots \ar[r]^{\beta_n} & y_n=y}
\end{equation*}
then $m=n$, $x_i = y_i$, and there are $ h_i \in \Aut_{\mathcal{E}} (x_i)$ such that the following diagram commutes, $1 \leqslant i \leqslant n-1$:
\begin{align*}
\xymatrix{ x_0 \ar[r]^{\alpha_1} \ar@{=}[d]^{id} & x_1 \ar[r]^{\alpha_2} \ar[d]^{h_1} & \ldots \ar[r]^{\alpha_{\ldots}} \ar[d]^{h_{\ldots}} & x_{n-1} \ar[r]^{\alpha_n} \ar[d]^{h_{n-1}} & x_n \ar@{=}[d]^{id} \\
x_0 \ar[r]^{\beta_1} & x_1 \ar[r]^{\beta_2} & \ldots \ar[r]^{\beta_{\ldots}} & x_{n-1} \ar[r]^{\beta_n} & x_n}
\end{align*}
\end{definition}

The UFP gives a characterization of finite free EI categories.
\begin{proposition}
A finite EI category $\mathcal{E}$ is free if and only if it satisfies the UFP.
\end{proposition}

\begin{proof}
Suppose $\mathcal{E}$ is a finite free EI category generated by a finite EI quiver $\hat{Q}= (Q_0, Q_1, s, t, f, g)$. Let $\alpha: v \rightarrow w$ be an arbitrary non-isomorphism. By the previous proposition $\alpha$ can be written as a composite of unfactorizable morphisms. Let $\alpha_m \circ \ldots \circ \alpha_1$ and $\beta_n \circ \ldots \circ \beta_1$ be two such decompositions of $\alpha$. It is easy to see from definitions that an unfactorizable morphism in $\mathcal{E}$ lies in $g(\tau)$ for some unique arrow $\tau \in Q_1$. Thus $\alpha_m \circ \ldots \circ \alpha_1$ and $\beta_n \circ \ldots \circ \beta_1$ determine two paths $\gamma_1$ and $\gamma_2$ in $\hat{Q}$ from $v$ to $w$. But $\alpha$ is contained in $\Hom_{\mathcal{E}} (v,w) = \bigsqcup _{\gamma} H_{\gamma}$, the disjoint union taken over all possible paths from $v$ to $w$, so $\gamma_1$ must be the same as $\gamma_2$. Consequently, $m =n$, and $\alpha_i$ and $\beta_i$ have the same target and source for $1 \leqslant i \leqslant n$. By the definition of biset product, the fact
\begin{equation*}
\alpha_n \circ (\alpha_{n-1} \ldots \circ \alpha_1) = \beta_n \circ (\beta_{n-1} \circ \ldots \circ \beta_1)
\end{equation*}
in the biset product implies that there is an automorphisms $g_{n-1} \in \Aut _{\mathcal{E}} (x_{n-1})$ such that
\begin{equation*}
\alpha_n = \beta_n g_{n-1}, \quad (\alpha_{n-1} \circ \ldots \circ \alpha_1) = g_{n-1}^{-1} (\beta_{n-1} \circ \ldots \circ \beta_1),
\end{equation*}
where $x_{n-1}$ is the common target of $\alpha_{n-1}$ and $\beta_{n-1}$. By an easy induction on $n$, we show that $\{ \alpha_i \}_{i=1}^n$ and $\{ \beta_j \}_{j=1}^n$ have the required relations in the previous definition. Thus $\mathcal{E}$ satisfies the UFP.

On the other hand, if $\mathcal{E}$ satisfies the UFP,  we want to show that $\mathcal{E}$ is isomorphic to the finite free EI category $\mathcal{C_{\hat{Q}}}$ generated from its finite EI quiver $\hat{Q}$. Define a functor $F: \mathcal{E} \rightarrow \mathcal{C_{\hat{Q}}}$ in the following way: First, $F(x) = x$ for every object $x$ in $\mathcal{E}$ since Ob$(\mathcal{E}) = \Ob (\mathcal{C_{\hat{Q}}})$ by our construction. Furthermore, it is also clear that $\Aut_{\mathcal{E}}(x) = \Aut \mathcal{C_{\hat{Q}}} (x)$ for every object $x$, and the biset of unfactorizable morphisms from $x$ to $y$ in $\mathcal{E}$ is the same as that in $\mathcal{C_{\hat{Q}}}$ for every pair of different objects $x$ and $y$. Therefore we can let $F$ be the identity map restricted to automorphisms and unfactorizable morphisms in $\mathcal{E}$. Proposition 4.2.6 tells us that every non-isomorphism $\alpha$ in $\mathcal{E}$ is a composite $\alpha_n \circ \ldots \circ \alpha_1$ of unfactorizable morphisms, so $F(\alpha)$ can be defined as $F(\alpha_n) \circ \ldots \circ F(\alpha_1)$. By the UFP $F$ is well-defined and is a bijection restricted to $\Hom_{\mathcal{E}} (x, y)$ for each pair of distinct objects $x$ and $y$. Consequently, $F$ is a bijection from Mor($\mathcal{E}$) to Mor$ (\mathcal{C_{\hat{Q}}})$ and so is an isomorphism. This finishes the proof.
\end{proof}

The following two lemmas give some special properties of finite free EI categories.

\begin{lemma}
Let $\mathcal{E}$ be a finite free EI category and $\alpha: x \rightarrow y$ be an unfactorizable morphism. Define $H = \Aut _{\mathcal{E}} (y)$ and $H_0 = \Stab _H (\alpha)$. If $| H_0 |$ is invertible in $k$, then the cyclic module $k\mathcal{E} \alpha$ is projective.
\end{lemma}

\begin{proof}
This is Lemma 5.2 of \cite{Li1}, where we assumed that the automorphism groups of all objects are invertible in $k$ but only used the fact that $| H_0 |$ is invertible in $k$. Here we give a sketch of the proof. Let $e = \frac{1} { | H_0 |} \sum_{h\in H_0} h$. Then $e$ is well defined since $| H_0 |$ is invertible in $k$, and is an idempotent in $k\mathcal{E}$. Now define a map $\varphi: k\mathcal{E}e \rightarrow k\mathcal{E} \alpha$ by sending $re$ to $r\alpha$ for $r \in k\mathcal{E}$. We can check that $\varphi$ is an $k\mathcal{E}$-module isomorphism. Thus $k\mathcal{E} \alpha$ is projective. See \cite{Li1} for a detailed proof.
\end{proof}

\begin{lemma}
Let $\mathcal{E}$ be a finite free EI category and $\alpha: x \rightarrow y$ and $\alpha': x' \rightarrow y'$ be two distinct unfactorizable morphisms in $\mathcal{E}$. Then $k\mathcal{E} \alpha \cap k\mathcal{E} \alpha' = 0$ or $k\mathcal{E} \alpha = k\mathcal{E} \alpha'$.
\end{lemma}

\begin{proof}
This is Lemma 5.1 of \cite{Li1}. We give a sketch of the proof. Notice that $k\mathcal{E} \alpha$ is spanned by all morphisms of the form $\beta \alpha$ with $\beta: y \rightarrow z$ a morphism starting at $y$. Similarly, $k\mathcal{E} \alpha'$ is spanned by all morphisms of the form $\beta' \alpha'$ with $\beta': y' \rightarrow z'$ a morphism starting at $y'$. If $x \neq x'$ or $y \neq y'$, then by the Unique Factorization Property of finite free EI categories we conclude that the set $\mathcal{E} \alpha \cap \mathcal{E} \alpha' = \emptyset$, and the conclusion follows. If $x =x'$ and $y= y'$, then the set $\mathcal{E} \alpha$ coincides with the set $\mathcal{E} \alpha'$ if and only if there is an automorphism $h \in \Aut _{\mathcal{E}} (y)$ such that $h\alpha = \alpha'$ again by the UFP. Otherwise, we must have $\mathcal{E} \alpha \cap \mathcal{E} \alpha' = \emptyset$. The conclusion follows from this observation.
\end{proof}

\begin{remark}
The reader can check that the conclusion of Lemma 4.2.9 is true for any non-isomorphisms $\alpha$ in $\mathcal{E}$ by using the UFP. Moreover, a direct check shows that it is also true for automorphisms. Similarly, we can also prove that Lemma 4.2.10 still holds if we assume that $\alpha$ and $\alpha'$ are two morphisms with the same target and source.
\end{remark}

It is well known that every subgroup of a free group is still free. Finite free EI categories have a similar property.

\begin{proposition}
Let $\mathcal{E}$ be a finite EI category. Then $\mathcal{E}$ is a finite free EI category if and only if all of its full subcategories are finite free EI categories.
\end{proposition}

\begin{proof}
The if part is trivial since $\mathcal{E}$ is such a subcategory of itself. Now let $\mathcal{D}$ be a full subcategory of $\mathcal{E}$. We want to show that $\mathcal{D}$ satisfies the UFP.

Take a factorizable morphism $\alpha$ in $\mathcal{D}$ and two decompositions of the following form:
\begin{equation*}
\xymatrix{x=x_0 \ar[r]^{\alpha_1} & x_1 \ar[r]^{\alpha_2} & \ldots \ar[r]^{\alpha_m} &x_m=y \\
x=x_0 \ar[r]^{\beta_1} & x'_1 \ar[r]^{\beta_2} & \ldots \ar[r]^{\beta_n} & x'_n=y}
\end{equation*}
where all $\alpha_i$ and $\beta_j$ are unfactorizable in $\mathcal{D}$, but possibly factorizable in $\mathcal{E}$. Decomposing them into unfactorizable morphisms in $\mathcal{E}$, we get two extended sequences of unfactorizable morphisms as follows:

\begin{equation*}
\xymatrix{x=x_0 \ar[r]^{\delta_1} & w_1 \ar[r]^{\delta_2} & \ldots \ar[r]^{\delta_r} &w_r=y \\
x=x_0 \ar[r]^{\theta_1} & w_1 \ar[r]^{\theta_2} & \ldots \ar[r]^{\theta_r} & w_r=y}
\end{equation*}
Each pair of morphisms $\delta_i$ and $\theta_i$ have the same source and target since $\mathcal{E}$ is a finite free EI category. Moreover, there are $h_i \in \Aut _{\mathcal{E}} (w_i)$, $1 \leqslant i \leqslant r-1$, such that $\theta_1 = h_1 \delta_1$, $\theta_2 = h_2 \delta_2 h_1^{-1}$,  $\ldots$, $\theta_{r-1} = h_{r-1} \delta_{r-1} h_{r-2}^{-1}$, $\theta_r = h_{r-1}^{-1} \delta_r$. If we can prove the fact that $m=n$ and $x_1 = x_1'$, $\ldots$, $x_m = x_n'$, then the conclusion follows. Indeed, if this is true, say $x_1 = x_1' = w_{r_1}$, $\ldots$, $x_m = x'_m = w_{r_m}$, then we have $\beta_1 = h_{r_1} \alpha_1$, $\beta_2 = h_{r_2} \alpha_2 h_{r_1}^{-1}$, $\ldots$, $\beta_n = \alpha_n h_{r_n}^{-1}$, which is exactly the UFP.

We show this fact by contradiction. Suppose that $x_1 = x'_1$, $x_2 = x_2'$, $\ldots$, $x_{i-1} = x'_{i-1}$, and let $x_i$ be the first object in the sequence different from $x_i'$. Notice both $x_m \ldots x_1$ and $x_n' \ldots x_1'$ are subsequences of the sequence $w_r \ldots w_1$. As a result, $x_i$ appears before $x_i'$ in $w_r \ldots w_1$, or after $x_i'$. Without loss of generality we assume that $x_i$ is before $x_i'$. Let $x_{i-1} = x'_{i-1} = w_a$, $x_i= w_b$ and $x_i'= w_c$. We must have $a \leq b < c$. Furthermore, we check that
\begin{equation*}
\beta_i = \theta_c \circ \theta_{c-1} \circ \ldots \circ \theta_{a+1} = (\theta_c \circ \ldots \circ \theta_{b+1}) \circ (\theta_b \circ \ldots \circ \theta_{a+1}),
\end{equation*}
with $(\theta_c \circ \ldots \circ \theta_{b+1})$ contained in $\Hom_{\mathcal{E}} (x_i, x_i')$ and $(\theta_b \circ \ldots \circ \theta_{a+1})$ contained in $\Hom_{\mathcal{E}} (x_{i-1}, x_i)$. But $\mathcal{D}$ is a full subcategory, so $\beta_i$ is the composite of two non-isomorphisms in $\mathcal{D}$. This is a contradiction since we have assumed that $\beta_i$ is unfactorizable in $\mathcal{D}$.
\end{proof}

The condition that $\mathcal{D}$ is a full subcategory of $\mathcal{E}$ is required. Consider the following two examples:

\begin{example}
Let $\mathcal{E}$ be the free category generated by the following quiver. Let $\mathcal{D}$ be the subcategory of $\mathcal{E}$ obtained by removing the morphism $\beta$ from $\mathcal{E}$. The category $\mathcal{E}$ is free, but the subcategory $\mathcal{D}$ is not free. Indeed, morphisms $\alpha$, $\beta\alpha$, $\gamma \beta$, $\gamma$ are all unfactorizable in $\mathcal{D}$. Thus the morphism $\gamma \beta \alpha$ has two decompositions $(\gamma\beta) \circ \alpha$ and $(\gamma)\circ (\beta\alpha)$, which contradicts the UFP.
\end{example}

\begin{equation*}
\xymatrix{ \bullet \ar[r]^{\alpha} & \bullet \ar[r]^{\beta} & \bullet \ar[r]^{\gamma} & \bullet }
\end{equation*}

\begin{example}
Let $\mathcal{E}$ be the following category: both $\Aut _{\mathcal{E}} (x)$ and $\Aut _{\mathcal{E}} (z)$ are trivial groups of order 1; $\Aut _{\mathcal{E}} (y) = \langle g \rangle$ has order 2; $g$ interchanges $\beta_1$ and $\beta_2$ and fixes $\alpha$. Then $\beta_1 \alpha = (\beta_2 g) \alpha = \beta_2 (g \alpha) = \beta_2 \alpha$. It is not hard to check that $\mathcal{E}$ satisfies the UFP; but the subcategory formed by removing the morphism $g$ from $\mathcal{E}$ does not satisfy the UFP.
\end{example}

\begin{equation*}
\xymatrix{ x \ar[r]^{\alpha} & y \ar@<0.5ex>[r]^{\beta_1} \ar@<-0.5ex>[r]_{\beta_2} & z }
\end{equation*}

We record here a theorem classifying finite EI categories with hereditary category algebras. For a proof, please refer to \cite{Li1}.

\begin{theorem}
The category algebra of a finite EI category $\mathcal{E}$ is hereditary if and only if $\mathcal{E}$ is a finite free EI category and the automorphism group of each object has order invertible in $k$.
\end{theorem}

Finite free EI categories have a certain universal property which is stated in the following proposition:

\begin{proposition}
Let $\mathcal{E}$ be a finite EI category. Then there is a finite free EI category $\hat{\mathcal{E}}$ and a full functor $\hat{F}: \hat{\mathcal{E}} \rightarrow \mathcal{E}$ such that $\hat{F}$ is the identity map restricted to objects, isomorphisms and unfactorizable morphisms. This finite free EI category is unique up to isomorphism.
\end{proposition}

We call $\hat{\mathcal{E}}$ the \textit{free EI cover} of $\mathcal{E}$.

\begin{proof} We mentioned before that every finite EI category $\mathcal{E}$ determines a finite EI quiver $\hat{Q}$ (see the paragraphs before Proposition 4.2.6), hence a finite free EI category $\mathcal{C_{\hat{Q}}}$ satisfying: Ob$(\mathcal{E}) = \Ob (\mathcal{C_{\hat{Q}}})$; $\Aut_{\mathcal{E}}(x) = \Aut \mathcal{C_{\hat{Q}}} (x)$ for every object $x$; the biset of unfactorizable morphisms from $x$ to $y$ in $\mathcal{E}$ is the same as that in $\mathcal{C_{\hat{Q}}}$ for every pair of different objects $x$ and $y$.

Define a functor $\hat{F}: \mathcal{C_{\hat{Q}}} \rightarrow \mathcal{E}$ in the following way: $\hat{F}$ is the identity map on objects, isomorphisms and unfactorizable morphisms. Now if $\delta: x \rightarrow y$ is neither an isomorphism nor an unfactorizable morphism, it can be decomposed as the composite
\begin{equation*}
\xymatrix{x=x_0 \ar[r]^{\beta_1} & x_1 \ar[r]^{\beta_2} & \ldots \ar[r]^{\beta_m} &x_m=y,}
\end{equation*}
where each $\beta_i$ is unfactorizable for $1 \leqslant i \leqslant m$. Define
\begin{equation*}
\hat{F}(\delta) = \hat{F}(\beta_m) \ldots \hat{F}(\beta_2) \hat{F}(\beta_1).
\end{equation*}

We want to verify that $\hat{F}$ is well defined for factorizable morphisms as well. That is, if $\delta$ has another decomposition into unfactorizable morphisms
\begin{equation*}
\xymatrix{x=x_0 \ar[r]^{\beta'_1} & z_1 \ar[r]^{\beta_2} & \ldots \ar[r]^{\beta'_n} &z_n=y,}
\end{equation*}
then
\begin{equation*}
\hat{F}(\beta_n) \hat{F}(\beta_{n-1}) \ldots \hat{F}(\beta_1) = \hat{F}(\beta'_n) \ldots \hat{F}(\beta'_2) \hat{F}(\beta'_1).
\end{equation*}
Since $\mathcal{E}$ satisfies the UFP, we have $m=n$ and $x_i = z_i$ for $1 \leqslant i \leqslant n$, and $\beta_1 = h_1 \beta'_1$, $\beta_2 =h_2 \beta'_2 h_1^{-1}, \ldots, \beta_{n-1} = h_{n-1} \beta'_{n-1} h_{n-2}^{-1}$, $\beta_n =\beta'_n h_{n-1}^{-1}$, where $h_i \in \Aut_{\mathcal{E}} (x_i)$. Since $\hat{F}$ is the identity map on automorphisms and unfactorizable morphisms, we get:
\begin{align*}
& \hat{F}(\beta_n) \hat{F}(\beta_{n-1}) \ldots \hat{F}(\beta_1) \\
& = \beta_n \circ \beta_{n-1} \circ \ldots \circ \beta_1 \\
& = (\beta'_n h_{n-1}^{-1}) \circ (h_{n-1} \beta'_{n-1} h_{n-2}^{-1}) \circ \ldots \circ (h_1 \beta'_1) \\
& = \beta'_n \circ \beta'_{n-1} \circ \ldots \beta'_1 \\
& = \hat{F}(\beta'_n) \hat{F}(\beta'_{n-1}) \ldots \hat{F}(\beta'_1).
\end{align*}

Therefore, $\hat{F}$ is a well defined functor. It is full since all automorphisms and unfactorizable morphisms in $\mathcal{E}$ are images of $\hat{F}$, and all other morphisms in $\mathcal{E}$ are their composites. Moreover, $\hat{\mathcal{E}}$ is unique up to isomorphism since it is completely determined by objects, isomorphisms and unfactorizable morphisms in $\mathcal{E}$.
\end{proof}

It turns out that the covering functor $F: \hat {\mathcal{E}} \rightarrow \mathcal{E}$ described in this proposition gives rise to a algebra quotient map $\varphi: k \hat {\mathcal{E}} \rightarrow k\mathcal{E}$.

\begin{proposition}
Let $\mathcal{E}$ be a finite EI category and $\hat{\mathcal{E}}$ be its free EI cover. Then the category algebra $k\mathcal{E}$ is a quotient algebra of $k \hat {\mathcal{E}}$. If $k\mathcal{E} \cong k \hat{\mathcal{E}} /I$, then the $k \hat{\mathcal{E}}$-ideal $I$ as a vector space is spanned by elements of the form $\hat{\alpha} - \hat{\beta}$, where $\hat{\alpha}$ and $\hat{\beta}$ are morphisms in $\hat{\mathcal{E}}$ with $\hat{F} (\hat{\alpha}) = \hat{F} (\hat{\beta})$.
\end{proposition}

\begin{proof}
Let $U$ be the vector space spanned by elements $\hat{\alpha} - \hat{\beta}$ such that $\hat{F} (\hat{\alpha}) = \hat{F} (\hat{\beta})$. Clearly, $U \subseteq I$ and we want to show the other inclusion. Let $x \in U$. By the definition of category algebras, $x$ can be expresses uniquely as $\sum_{i=1} ^n \lambda_i \alpha_i$ where $\alpha_i$ are pairwise different morphisms in $\hat {\mathcal{E}}$ and $\lambda_i \in k$. Then $\varphi (\sum _{i=1}^n \lambda_i \alpha_i) = \sum_{i=1}^n \lambda_i \hat{F} (\alpha_i) = 0$. Those $\hat{F} (\alpha_i)$ are probably not pairwise different in $\mathcal{E}$. By changing the indices if necessary, we can write the set $\{ \alpha_i \}_{i=1}^n$ as a disjoint union of $l$ subsets: $\{ \alpha_1, \ldots, \alpha_{s_1} \}$, $\{\alpha_{s_1+1}, \ldots, \alpha_{s_2} \}$ and so on, until $\{\alpha_{s_{l-1}+1}, \ldots, \alpha_{s_l} \}$ such that two morphisms have the same image under $\hat{F}$ if and only if they are in the same set.

Now we have:
\begin{align*}
\varphi(x) = (\lambda_1 + \ldots + \lambda_{s_1}) \hat{F} (\alpha_{s_1}) + \ldots + (\lambda _{s_{l-1}+1} + \ldots + \lambda_{s_l}) \hat{F} (\alpha_{s_l}) = 0.
\end{align*}
Therefore,
\begin{align*}
\lambda_1 + \ldots + \lambda_{s_1} = \ldots = \lambda _{s_{l-1}+1} + \ldots + \lambda_{s_l} = 0,
\end{align*}
and hence
\begin{align*}
x &= [\lambda_2 (\alpha_2 - \alpha_1) + \ldots + \lambda_{s_1}(\alpha_{s_1} - \alpha_1)] + \ldots \\
& + [\lambda _{s_{l-1}+2} (\alpha_{s_{l-1}+2} - \alpha_{s_{l-1}+1}) + \ldots + \lambda_{s_l} (\alpha_{s_l} - \alpha_{s_{l-1}+1})]
\end{align*}
is contained in $U$.
\end{proof}

\section{Koszul properties of finite EI categories}

In this section we study the Koszul properties of finite EI categories. First we discuss the possibility to put a length grading on a finite EI category.

If $\mathcal{E}$ is a finite free EI category, we can put a \textit{length grading} on its morphisms as follows: automorphisms and unfactorizable morphisms are given grades 0 and 1 respectively; if $\alpha$ is a factorizable morphism, then it can be expressed (probably not unique) as a composite $\alpha_n \alpha_{n-1} \ldots \alpha_2 \alpha_1$ with all $\alpha_i$ unfactorizable and we assign $\alpha$ grade $n$. This grading is well defined by the Unique Factorization Property of finite free EI categories. Obviously, this length grading cannot be applied to an arbitrary finite EI category. We say a finite EI category can be \textit{graded} if this length grading is well defined on it. The following proposition gives us criterions to determine whether an arbitrary finite EI category can be graded.

\begin{proposition}
Let $\mathcal{E}$ be a finite EI category. Then the following are equivalent:\
\begin{enumerate}
\item $\mathcal{E}$ is a graded finite EI category.
\item For each factorizable morphism $\alpha$ in $\mathcal{E}$, whenever it has two factorizations $\alpha_1 \circ \ldots \circ \alpha_m$ and $\beta_1 \circ \ldots \circ \beta_n$ into unfactorizable morphisms, we have $m=n$.
\item Let $\hat{\mathcal{E}}$ be the free EI cover of $\mathcal{E}$ and $\hat{F}: \hat{\mathcal{E}} \rightarrow \mathcal{E}$ be the covering functor. If two morphisms $\hat{\alpha}$ and $\hat{\beta}$ in $\hat{\mathcal{E}}$ have the same image under $\hat{F}$, then they have the same length in $\hat{\mathcal{E}}$.
\end{enumerate}
\end{proposition}

\begin{proof}
It is easy to see that if condition (2) holds, our grading works for $\mathcal{E}$, and hence (1) is true. Otherwise, if a factorizable morphism $\alpha$ has two decompositions $\alpha_n \circ \ldots \circ \alpha_1$ and $\beta_m \circ \ldots \circ \beta_1$ with $m \neq n$, then $\alpha$ should be assigned a grade $n$ by the first decomposition, and a grade $m$ by the second decomposition. Thus our grading cannot be applied to $\mathcal{E}$. This proves the equivalence of (1) and (2).

Now let $\alpha$ be an arbitrary morphism in $\mathcal{E}$ which has two different decompositions $\alpha_n \circ \ldots \circ \alpha_1$ and $\beta_m \circ \ldots \circ \beta_1$ into unfactorizable morphisms. Since $\hat{\mathcal{E}}$ is the free EI cover of $\mathcal{E}$, these unfactorizable morphisms are also unfactorizable morphisms in $\hat{\mathcal{E}}$. Let $\hat{\alpha}$ and $\hat{\beta}$ be the composite morphisms of these $\alpha_i$'s and $\beta_i$'s in $\hat{\mathcal{E}}$ respectively. Thus $\hat{\alpha} - \hat{\beta}$ is contained in $U$ since they have the same image $\alpha$ under $\hat{F}$. If (3) is true, then $m = n$ since $\hat{\alpha}$ and $\hat{\beta}$ have lengths $m$ and $n$ respectively. Therefore (3) implies (2). We can check that (2) implies (3) in a similar way.
\end{proof}

Theorem 3.1.6 has a corresponding version for finite EI categories.

\begin{proposition}
Let $\mathcal{E}$ be a graded finite EI category. Then $k \mathcal{E}$ is a Koszul algebra if and only if $k\mathcal{E}$ is a quasi-Koszul algebra and $\mathcal{E}$ is a standardly stratified category (in a sense defined in \cite{Webb3}) with respect to the canonical partial order on $\Ob \mathcal{E}$.
\end{proposition}

\begin{proof}
Let $\mathcal{C}$ the $k$-linearization of $\mathcal{E}$. Then $\mathcal{C}$ is a directed category. Actually, it is the associated category of the category algebra $k\mathcal{E}$. Applying Theorem 3.1.6 to $\mathcal{C}$ the conclusion follows immediately.
\end{proof}

We know that if a directed category $\mathcal{C}$ is standardly stratified (with respect to the preorder induced by the given linear order), then $\mathcal{C} (x,x)$ viewed as a $\mathcal{C}$-module has finite projective dimension for each $x \in \Ob \mathcal{C}$. For finite EI categories, the converse is true as well.

\begin{proposition}
Let $\mathcal{E}$ be a finite EI category which might not be graded. Then $\mathcal{E}$ is standardly stratified if and only if $M = \bigoplus _{x \in \Ob \mathcal{E}} k \Aut _{\mathcal{E}} (x)$ viewed as a  $k\mathcal{E}$-module has finite projective dimension.
\end{proposition}

\begin{proof}
One direction is clear since $M$ is precisely the direct sum of all standard modules, and hence has finite projective dimension.

Conversely, suppose that $\mathcal{E}$ is not standardly stratified. Then there is a non-isomorphism $\gamma: t \rightarrow y$ such that the order of $H_{\gamma}$ is not invertible in the field $k$ by Theorem 2.5 in \cite{Webb3}, where $H = \Aut _{\mathcal{E}} (y)$ and $H_{\gamma} = \Stab _H (\gamma)$. For this object $y$, define $S$ to be the set of objects $w$ such that there is a non-isomorphism $\beta: w \rightarrow y$ satisfying that $| H_{\beta}|$ is not invertible in $k$. This set $S$ is nonempty since $t \in S$. It is a poset equipped with the partial order inherited from the canonical partial order on $\Ob \mathcal{E}$. Take a fixed object $z$ which is maximal in this set and define $I_{ > z} = \{ x \in \Ob \mathcal{E} \mid x > z \}$.

By our definition, for an arbitrary object $x \in I_{>z}$ and a non-isomorphism $\alpha: x \rightarrow y$ (if it exists), the group $H _{\alpha} \leqslant H$ has an order invertible in $k$. Therefore, the $kH$-module $kH \alpha$ is projective. Since the value of $k \mathcal{E} 1_x$ on $y$ is 0 or is spanned by all non-isomorphisms from $x$ to $y$, and these non-isomorphisms form a disjoint union of $H$-orbits, we conclude that the value of $k\mathcal{E} 1_x$ on $y$ is a projective $kH$-module (notice that we always view 0 as a zero projective module). With the same reasoning, we know that the value of $k\mathcal{E} 1_z$ on $y$ is not a projective $kH$-module.

Consider the $k\mathcal{E}$-module $L = k \Aut _{\mathcal{E}} (z)$. We claim that $\pd L = \infty$. If this is true, then $\pd M = \infty$ since $L$ is a direct summand of $M$. We prove this claim by showing the following statement: for each $i \geqslant 1$, every projective cover of $\Omega^i(L)$ is supported on $I_{ > z}$; the value $\Omega^i(L) (y)$ of $\Omega(L)$ on $y$ is non-zero and is not a projective $kH$-module. Clearly, $\Omega(L)$ is spanned by all non-isomorphisms starting from $z$ and is supported on $I_{ > z}$; $\Omega(L) (y)$, spanned by all non-isomorphisms from $z$ to $y$, is non-zero. Moreover, $\Omega(L)(y)$ coincides with the value of $k\mathcal{E} 1_z$ on $y$ and is not a projective $kH$-module. Therefore our statement is true for $i=1$.

Suppose that this statement is true for $n$, and let $P$ be a projective cover of $\Omega^n (L)$. The exact sequence
\begin{equation*}
\xymatrix{ 0 \ar[r] & \Omega^{n+1} (L) \ar[r] & P \ar[r] & \Omega^n (L) \ar[r] & 0}
\end{equation*}
gives rise to an exact sequence
\begin{equation*}
\xymatrix{ 0 \ar[r] & \Omega^{n+1} (L) (y) \ar[r] & P(y) \ar[r] & \Omega^n (L)(y) \ar[r] & 0.}
\end{equation*}

Let us focus on the above sequences. Since $\Omega^n (L)$ is supported on $I_{ > z}$, so are $P$ and $\Omega^{n+1} (L)$. By the induction hypothesis $\Omega^n (L) (y) \neq 0$, Thus $P(y) \neq 0$. But $P$ is supported on $I_{ > z}$, so $P \in \text{add} (\bigoplus _{x \in I_{ > z}}  k\mathcal{E} 1_x)$. Notice that the value of each $k \mathcal{E} 1_x$ on $y$ is zero or a nontrivial projective $kH$-module. Therefore, $P(y)$ is a projective $kH$-module. Again by the induction hypothesis, $\Omega^n (L) (y)$ is not a projective $kH$-module, so $\Omega^n (L) (y) \ncong P(y)$, and $\Omega^{n+1} (L) (y)$ is non-zero. It cannot be a projective $kH$-module. Otherwise, $\Omega^{n+1} (L) (y)$ is also an injective $kH$-module and hence the above sequence splits, so $\Omega^n (L) (y)$ as a summand of $P(y)$ is a projective $kH$-module, too. But this contradicts the induction hypothesis.

We proved the induction hypothesis for $\Omega^{n+1} (L)$. Thus our statement and claim are proved. Consequently, $\pd M = \infty$.
\end{proof}

Now we can prove:

\begin{theorem}
Let $\mathcal{E}$ be a finite free EI category graded by the length grading. Then the following are equivalent:
\begin{enumerate}
\item $\pd k\mathcal{E}_0 \leqslant 1$;
\item $k \mathcal{E}$ is a Koszul algebra;
\item $\mathcal{E}$ is standardly stratified;
\item $\pd k\mathcal{E}_0 < \infty$.
\end{enumerate}
\end{theorem}

\begin{proof}
It is clear that $\pd k\mathcal{E}_0 = 0$ if and only if $\mathcal{E}$ is a finite EI category with a single object since we only consider connected categories. In this situation, $k \mathcal{E} = k \mathcal{E}_0$, and all statements are trivially true. Thus without loss of generality we suppose that $\pd k\mathcal{E}_0 \neq 0$.

Observe that $\pd k\mathcal{E} _0 =1$ if and only if $\Omega( k\mathcal{E} _0) = J = \bigoplus _{i \geqslant 1} k\mathcal{E}_i$ is projective. Since $J$ is spanned by all non-isomorphisms in $\mathcal{E}$ and each non-isomorphism can be written as a composition of unfactorizable morphisms, it is generated in degree 1. Thus $k \mathcal{E} _0$ is a linear $k\mathcal{E}$-module, and (1) implies (2). Clearly (2) implies (3). The statements (3) and (4) are equivalent by Proposition 4.3.3.

Now we prove that (3) implies (1). If $\mathcal{E}$ is standardly stratified, then for every morphism $\alpha : x \rightarrow y$ in $\mathcal{E}$, the order of $\Stab _H (\alpha)$ is invertible in $k$, where $H = \Aut _{\mathcal{E}} (y)$. By Lemma 4.2.10, $J$ is a direct sum of some $k\mathcal{E}$-modules $k\mathcal{E} \alpha_i$'s with each $\alpha_i$ unfactorizable. By Lemma 4.2.9, each $k\mathcal{E} \alpha_i$ is projective. Therefore, $J$ is also projective, i.e., $\pd k\mathcal{E}_0 = 1$.
\end{proof}

This theorem and Theorem 3.2.8 give us a way to construct Koszul algebras in the classical sense. Indeed, let $\mathcal{E}$ be a standardly stratified finite free EI category and define $\mathcal{D}$ to be the subcategory formed by removing all non-identity automorphisms. Then by Theorem 3.2.8 $k\mathcal{D}$ is a Koszul algebra in the classical sense since $k\mathcal{E}_0$ is a linear $k \mathcal{E}$-module in the generalized sense by the previous theorem.

Let us get more information about the projective resolutions of $k\mathcal{E} _0$ for arbitrary finite free EI categories. In general, $\Omega( k\mathcal{E} _0) \cong J = \bigoplus _{i \geqslant 1} k\mathcal{E}_i$ is not projective, but it is still a direct sum of some $k\mathcal{E}$-modules $k\mathcal{E} \alpha$'s with each $\alpha$ unfactorizable by Lemma 4.2.10. Thus the projective resolutions of $k\mathcal{E}_0$ is completely determined by the projective resolutions of those $k\mathcal{E} \alpha$'s.

\begin{lemma}
Let $\mathcal{E}$ be a finite free EI category and $\alpha: x \rightarrow y$ be an unfactorizable morphism. Grade the $k \mathcal{E}$-module $k\mathcal{E} \alpha$ by putting $\alpha$ in degree 1, namely, $(k \mathcal{E} \alpha)_1 = k \Aut _{\mathcal{E}} (y) \alpha$. Then $\Omega (k \mathcal{E} \alpha)$ is 0 or is generated in degree 1, and $\Omega (k \mathcal{E} \alpha)_1 = \Omega (k \mathcal{E} \alpha) (y)$, the value of $\Omega (k \mathcal{E} \alpha)$ on $y$.
\end{lemma}

\begin{proof}
Let $H = \Aut _{\mathcal{E}} (y)$ and $H_0 = \Stab _H (\alpha)$. If $| H_0|$ is invertible in $k$, then by Lemma 4.2.9, $k\mathcal{E} \alpha$ is a projective $k\mathcal{E}$-module, so $\Omega^i( k \mathcal{E} \alpha) =0$ for all $i \geqslant 1$, in particular $\Omega (k \mathcal{E} \alpha) =0$. The conclusion is trivially true. Thus we only need to deal with the case that $|H_0|$ is not invertible.

Consider the projective presentation
\begin{equation*}
\xymatrix {0 \ar[r] & N \ar[r] & k\mathcal{E} 1_y [1] \ar[r]^p & k \mathcal{E} \alpha \ar[r] & 0}
\end{equation*}
where $p$ maps $1_y$ to $\alpha$. Since $\Omega (k \mathcal{E} \alpha)$ is isomorphic to a direct summand of $N$, it is enough to show that $N$ is generated in degree 1, and $N_1 = N(y)$.

Notice that $k\mathcal{E} 1_y$ is spanned by all morphisms in $\mathcal{E}$ with source $y$, and $k\mathcal{E} \alpha$ is spanned by all morphisms in $\mathcal{E}$ of the form $\beta \alpha$ where $\beta$ is a morphism in $\mathcal{E}$ with source $y$. We claim that $N$ is spanned by vectors of the form $\beta_1 - \beta_2$ with $\beta_1 \alpha = \beta_2 \alpha$, where $\beta_1$ and $\beta_2$ are two morphisms with source $y$.

Clearly, every such difference is contained in $N$. Conversely, let $v \in N$. Then $v$ can be written as $\sum _{i=1} ^n \lambda_i \beta_i$ such that $\lambda_i \in k$ and $\beta_i$ are pairwise different morphisms with source $y$. By the definition of $p$, $\sum _{i=1} ^n \lambda_i \beta_i \alpha = 0$. Those $ \beta_i \alpha$ might not be pairwise different in $\mathcal{E}$. Now we apply the same technique used in the proof of Proposition 4.2.17. By changing the indices if necessary, we can group the same morphisms together and suppose that $ \beta_1 \alpha = \ldots = \beta_{s_1} \alpha$, $ \beta_{s_1+1} \alpha = \ldots = \beta_{s_2} \alpha$ and so on, until $ \beta_{s_{l-1}+1} \alpha = \ldots = \beta_{s_l} \alpha$.

We have:
\begin{align*}
p(v) = (\lambda_1 + \ldots + \lambda_{s_1}) \beta_{s_1} \alpha + \ldots + (\lambda _{s_{l-1}+1} + \ldots + \lambda_{s_l}) \beta_{s_l} \alpha = 0.
\end{align*}
Therefore,
\begin{align*}
\lambda_1 + \ldots + \lambda_{s_1} = \ldots = \lambda _{s_{l-1}+1} + \ldots + \lambda_{s_l} = 0,
\end{align*}
and hence
\begin{align*}
v &= [\lambda_2 (\beta_2 - \beta_1) + \ldots + \lambda_{s_1}(\beta_{s_1} - \beta_1)] + \ldots \\
& + [\lambda _{s_{l-1}+2} (\beta_{s_{l-1}+2} - \beta_{s_{l-1}+1}) + \ldots + \lambda_{s_l} (\beta_{s_l} - \beta_{s_{l-1}+1})].
\end{align*}
So $v$ can be written as a sum of these differences.

Now we can prove the lemma. Take an arbitrary object $z \in \Ob \mathcal{E}$ and consider the value $N (z)$. If it is 0, the conclusion holds trivially. Suppose that $N (z) \neq 0$. By the above description, $N (z)$ is spanned by vectors $\beta_1 - \beta_2$ such that $\beta_1, \beta_2$ are two morphisms from $y$ to $z$, and $\beta_1 \alpha = \beta_2 \alpha$. By the equivalent definition of UFP described in Remark 6.5, there is an automorphism $h \in \Aut _{\mathcal{E}} (y)$ such that $\beta_1 = \beta_2 h$ and $\alpha = h^{-1} \alpha$. Therefore $h\alpha = \alpha$, and $1-h \in N (y)$. Thus $\beta_1 - \beta_2 = \beta (1-h) \in k\mathcal{E} \cdot N (y)$. Since $z$ is taken to be an arbitrary object, $N$ is generated by $N (y)$, which is clearly equal to $N_1$.
\end{proof}

From this lemma we can get:

\begin{proposition}
Let $\mathcal{E}$ be a finite free EI category, then $\Ext _{k\mathcal{E}} ^2 (k\mathcal{E}_0, k\mathcal{E}_0) = 0$.
\end{proposition}

\begin{proof}
Since $\Omega (k \mathcal{E} _0) \cong J$, it is enough to show $\Ext _{k \mathcal{E}} ^1 (J, k \mathcal{E} _0) =0$. The conclusion holds trivially if $J$ is projective. Otherwise, since $J$ is the direct sum of some $k\mathcal{E} \alpha$'s with $\alpha$ unfactorizable, by the above lemma we know that $\Omega J$ is generated in degree 1.

Applying the functor $\Hom _{k\mathcal{E}} (-, k\mathcal{E} _0)$ to $0 \rightarrow \Omega J \rightarrow P \rightarrow J \rightarrow 0$ we get
\begin{equation*}
0 \rightarrow \Hom _{k\mathcal{E} } (J, k\mathcal{E}_0) \rightarrow \Hom _{k\mathcal{E} } (P, k\mathcal{E} _0) \rightarrow \Hom _{k\mathcal{E}} (\Omega J, k\mathcal{E}_0) \rightarrow \Ext _{k\mathcal{E}} ^1 (J, k\mathcal{E}_0) \rightarrow 0.
\end{equation*}
Since all modules are generated in degree 1, the sequence
\begin{equation}
0 \rightarrow \Hom _{k\mathcal{E} } (J, k\mathcal{E}_0) \rightarrow \Hom _{k\mathcal{E}} (P, k\mathcal{E} _0) \rightarrow \Hom _{k\mathcal{E}} (\Omega J, k\mathcal{E}_0)
\end{equation}
is isomorphic to the sequence
\begin{equation*}
0 \rightarrow \Hom _{k\mathcal{E} _0} (J_1, k\mathcal{E}_0) \rightarrow \Hom _{k\mathcal{E} _0} (P_1, k\mathcal{E} _0) \rightarrow \Hom _{k\mathcal{E} _0} ((\Omega J)_1, k\mathcal{E}_0)
\end{equation*}
obtained by applying the exact functor $\Hom _{k \mathcal{E}_0} (-, k\mathcal{E} _0)$ to the exact sequence $0 \rightarrow (\Omega J)_1 \rightarrow P_1 \rightarrow J_1 \rightarrow 0$. Thus the last map in sequence (4.3.1) is surjective, so $\Ext _{k \mathcal{E}} ^1 (J, k \mathcal{E} _0) =0$.
\end{proof}

The fact that $\Omega J$ is generated in degree 1 implies $\Ext^2 _{k\mathcal{E}} (k\mathcal{E}_0, k\mathcal{E}_0) =0$. Actually the converse statement is also true. Indeed, consider the exact sequence $0 \rightarrow \Omega J \rightarrow P \rightarrow J \rightarrow 0$. If $\Ext^2 _{k\mathcal{E}} (k\mathcal{E}_0, k\mathcal{E}_0) =0$, applying the exact functor $\Hom _{k\mathcal{E} _0} (-, k\mathcal{E}_0)$ we get the exact sequence
\begin{equation*}
0 \rightarrow \Hom _{k\mathcal{E} _0} (J, k\mathcal{E} _0) \rightarrow \Hom _{k\mathcal{E} _0} (P, k\mathcal{E} _0) \rightarrow \Hom _{k\mathcal{E} _0} (\Omega J, k\mathcal{E} _0) \rightarrow 0,
\end{equation*}
which is isomorphic to
\begin{equation*}
0 \rightarrow \Hom _{k\mathcal{E} _0} (J_1, k\mathcal{E} _0) \rightarrow \Hom _{k\mathcal{E} _0} (P_1, k\mathcal{E} _0) \rightarrow \Hom _{k\mathcal{E} _0} (\Omega J / J (\Omega J), k\mathcal{E} _0) \rightarrow 0
\end{equation*}
since both $J$ and $P$ are generated in degree $1$. Applying the functor $\Hom _{ k\mathcal{E} _0} (-, k \mathcal{E}_0 )$ again, we recover $0 \rightarrow \Omega J / J (\Omega J) \rightarrow P_1 \rightarrow J_1 \rightarrow  0$. Therefore, $\Omega J / J (\Omega J) \cong (\Omega J)_1$, so $\Omega J$ is generated in degree 1.

Finite free EI categories with quasi-Koszul category algebras have very special homological properties. For example:

\begin{proposition}
Let $\mathcal{E}$ be a finite free EI category. Then the following are equivalent:
\begin{enumerate}
\item $\Ext _{k\mathcal {E}} ^i (k \mathcal{E} _0, k\mathcal{E} _0) = 0$ for all $i \geqslant 2$;
\item for every unfactorizable morphism $\alpha: x \rightarrow y$ and $i \geqslant 0$, either $\Omega ^i (k \mathcal{E} \alpha) $ are all 0, or they are all generated in degree 1 (in which case it is generated by $\Omega ^i (k \mathcal{E} \alpha) (y)$);
\item $k \mathcal{E}$ is a quasi-Koszul algebra.
\end{enumerate}
\end{proposition}

\begin{proof}
If $k \mathcal{E}$ is a quasi-Koszul algebra, then
\begin{equation*}
\Ext _{k\mathcal {E}} ^i (k \mathcal{E} _0, k\mathcal{E} _0) = \Ext _{k\mathcal{E}} ^2 (k \mathcal{E} _0, k\mathcal{E} _0) \cdot \Ext _{k\mathcal{E}} ^{i-2} (k \mathcal{E} _0, k\mathcal{E} _0)
\end{equation*}
for every $i \geqslant 2$. But $\Ext _{k\mathcal{E}} ^2 (k \mathcal{E} _0, k\mathcal{E} _0) =0$ by Proposition 4.3.6, so (3) implies (1). Clearly, (1) implies (3).

Notice that $k \mathcal{E} \alpha$ is a isomorphic to a direct summand of $J \cong \Omega (k \mathcal{E}_0)$. Thus we only need to prove the equivalence of the following two statements:
\begin{enumerate}[(1')]
\item $\Ext _{k\mathcal {E}} ^i (J, k\mathcal{E} _0) = 0$ for every $i \geqslant 1$;
\item $\Omega ^i (J) =0 $ or is generated in degree 1 for every $i \geqslant 1$.
\end{enumerate}

Since the technique we use is similar to that in the proof of Proposition 4.3.6, we only give a sketch of the proof. In the case that $J$ is projective, i.e., $\mathcal{E}$ is standardly stratified, then (1') and (2') are trivially true, hence they are equivalent. Now suppose that $J$ is not projective. From the proof of the previous proposition and the paragraph after it we conclude that $\Omega J$ is generated in degree 1 if and only if $\Ext _{k\mathcal{E}} ^1 (J, k\mathcal{E} _0) =0$. Replacing $J$ by $\Omega J$ (which is also generated in degree 1 either by the induction hypothesis or by the hypothesis $\Ext _{k\mathcal{E}} ^1 (J, k\mathcal{E} _0) =0$) and using the same technique, we get $\Omega ^2 (J)$ is generated in degree 1 if and only if $\Ext _{k\mathcal{E}} ^2 (J, k\mathcal{E} _0) =0$. The equivalence of (1') and (2'), and hence the equivalence of (1) and (2), come from induction.
\end{proof}

The reader may guess that the category algebra of a finite free EI category is always quasi-Koszul in our sense because of the following reasons: finite free EI categories generalize finite groups and acyclic quivers, for which the associated algebras are all quasi-Koszul;  by Proposition 4.3.2 and Theorem 4.3.4, for an arbitrary finite EI category $\mathcal{E}$, $k\mathcal{E}$ is Koszul if $\mathcal{E}$ is standardly stratified and one of the following condition holds: $k \mathcal{E}$ is quasi-Koszul, or $\mathcal{E}$ is a finite free EI category; and we have proved that $\Ext _{k\mathcal{E}} ^2 (k\mathcal{E} _0, k\mathcal{E} _0) =0$ if $\mathcal{E}$ is a finite free EI category. Unfortunately, this conjecture is false, as shown by the following example.

\begin{example}
Let $\mathcal{E}$ be the following finite EI category where: $\Aut _{\mathcal{E}} (x) = \langle 1_x \rangle$, $\Aut _{\mathcal{E}} (z) = \langle 1_z \rangle$, $\Aut _{\mathcal{E}} (y) = \langle h \rangle$ is a group of order 2; $\mathcal{E} (x, y) = \{ \alpha \}$, $\mathcal{E} (y, z) = \{ \beta \}$ and $\mathcal{E} (x, z) = \{ \beta \alpha \}$. The reader can check that $\mathcal{E}$ is a finite free EI category and then the length grading can be applied on it. Let $k$ be an algebraically closed field with characteristic 2.

\begin{equation*}
\xymatrix { x \ar[r] ^{\alpha} & y \ar[r] ^{\beta} & z.}
\end{equation*}

The indecomposable direct summands of $k \mathcal{E}$ and $k \mathcal{E}_0$ are:
\begin{equation*}
P_x = \begin{matrix}   x_0 \\ y_1 \\ z_2 \end{matrix}, \qquad P_y = \begin{matrix} & y_0 & \\ y_0 & & z_1 \end{matrix}, \qquad P_z = z_0, \qquad k \mathcal{E} _0 \cong x_0 \oplus z_0 \oplus \begin{matrix} y_0 \\ y_0 \end{matrix}.
\end{equation*}
We use indices to mark the degrees of composition factors. The reader should bear in mind that the two simple modules $y$ appearing in $P_y$ have the same degree.

Take the summand $x_0$ of $k\mathcal{E} _0$. By computation, we get
\begin{equation*}
\Omega (x_0) = \begin{matrix}   y_1 \\ z_2 \end{matrix}, \qquad \Omega ^2 (x_0) = \begin{matrix} y_1 \end{matrix}, \qquad \Omega ^3 (x_0) = y_1 \oplus z_2.
\end{equation*}
Applying $\Hom _{k\mathcal{E}} (-, k\mathcal{E} _0)$ to the exact sequence
\begin{equation*}
\xymatrix{0 \ar[r] & \Omega^3 (x_0) \ar[r] & P_y [1] \ar[r] & \Omega^2 (x_0) \ar[r] &0}
\end{equation*}
we get $\Ext^3 _{k\mathcal{E}} (k\mathcal{E}_0, k\mathcal{E}_0 ) \neq 0$. Consequently, $k \mathcal{E}$ is not a quasi-Koszul algebra in our sense by the previous proposition.
\end{example}

We aim to characterize finite free EI categories with quasi-Koszul category algebras. For this goal, we make the following definition:

\begin{definition}
Let $\mathcal{E}$ be a finite EI category. An object $x \in \Ob \mathcal{E}$ is called left regular if for every morphism $\alpha$ with target $x$, the stabilizer of $\alpha$ in $\Aut _{\mathcal{E}} (x)$ has an order invertible in $k$. Similarly, $x$ is called right regular if for every morphism $\beta$ with source $x$, the stabilizer of $\beta$ in $\Aut _{\mathcal{E}} (x)$ has an order invertible in $k$.
\end{definition}

\begin{remark}
We make some comments for this definition.
\begin{enumerate}
\item If $x \in \Ob \mathcal{E}$ is maximal, i.e., there is no non-isomorphisms with source $x$, then $x$ is right regular by convention; similarly, if $x$ is minimal, then it is trivially left regular.
\item The category $\mathcal{E}$ is standardly stratified if and only if every object $x \in \Ob \mathcal{E}$ is left regular; similarly, $\mathcal{E} ^{\textnormal{op}}$ is standardly stratified if and only if every object $x \in \Ob \mathcal{E}$ is right regular
\item If $\mathcal{E}$ is a finite free EI category and $x \in \Ob \mathcal{E}$. Then $x$ is left regular if and only if for every $\alpha$ with target $x$, the $k\mathcal{E}$-module $k \mathcal{E} \alpha$ is a left projective $k \mathcal{E}$-module. Similarly, $x$ is right regular if and only if for every $\beta$ with source $x$, the right $k\mathcal{E}$-module $\beta (k \mathcal{E})$ is a right projective $k \mathcal{E}$-module.
\end{enumerate}
\end{remark}

\begin{lemma}
Let $\mathcal{E}$ be a finite free EI category and $\beta: x \rightarrow y$ be a morphism with $x \in \Ob \mathcal{E}$ right regular. Then there exists some idempotent $e$ in $k \mathcal{E}$ such that $\beta (k \mathcal{E}) \cong e (k \mathcal{E})$ as right $k \mathcal{E}$-modules by sending $e$ to $\beta$. In particular, $\beta (k G) \alpha \cong e (kG) \alpha$ as vector spaces for every morphism $\alpha$ with target $x$, where $G = \Aut _{\mathcal{E}} (x)$.
\end{lemma}

\begin{proof}
Let $G_0 = \Stab _G (\alpha)$ and $e = \sum _{g \in G_0} g / | G_0 |$. This is well defined since $x$ is right regular. Then we can prove $\beta (k \mathcal{E}) \cong e (k \mathcal{E})$ as right $k\mathcal{E}$-modules in a way similar to the proof of Lemma 4.2.9. The isomorphism is given by sending $er$ to $\beta r$ for $r \in k\mathcal{E}$. Since the image of $e (kG) \alpha \subseteq k\mathcal{E}$ is exactly $\beta (k \mathcal{E}) \alpha$, we deduce that $e (kG) \alpha \cong \beta (kG) \alpha$ as vector spaces.
\end{proof}

Using these concepts, we can get a sufficient condition for the category algebra of a finite free EI category to be quasi-Koszul.

\begin{theorem}
Let $\mathcal{E}$ be a finite free EI category such that every object $x \in \Ob \mathcal{E}$ is either left regular or right regular. Then $k \mathcal{E}$ is quasi-Koszul.
\end{theorem}

\begin{proof}
By the second statement of Proposition 4.3.7, it is enough to show that for each unfactorizable $\alpha: x \rightarrow y$ and every $i \geqslant 1$, $\Omega ^i (k \mathcal{E} \alpha)$ is 0 or generated by $\Omega ^i (k \mathcal{E} \alpha) (y)$. Let $H = \Aut _{\mathcal{E}} (y)$ and $H_0 = \Stab _H (\alpha)$. If $| H_0 |$ is invertible in $k$, then $k \mathcal{E} \alpha$ is a projective $k\mathcal{E}$-module, and the conclusion follows. So we only need to deal with the case that the order of $H_0$ is not invertible in $k$.

By Lemma 4.3.5, $\Omega (k \mathcal{E} \alpha)$ is generated in degree 1, or equivalently, generated by its value $\Omega (k \mathcal{E} \alpha) (y) = 1_y \Omega (k \mathcal{E} \alpha)$ on $y$. Now suppose that $\Omega ^i (k \mathcal{E} \alpha)$ is also generated in degree 1, or equivalently, generated by its value $\Omega ^i (k \mathcal{E} \alpha) (y) = 1_y \Omega ^i (k \mathcal{E} \alpha)$ on $y$, where $i \geqslant 1$. We claim that $\Omega ^{i+1} (k \mathcal{E} \alpha)$ is generated by $\Omega ^{i+1} (k \mathcal{E} \alpha) (y)$, which is clearly equal to $\Omega ^{i+1} (k \mathcal{E} \alpha)_1$. If this is true, then conclusion follows from Proposition 4.3.7.

Take an arbitrary object $z \in \Ob \mathcal{E}$ such that $\mathcal{E} (y, z) \neq \emptyset$. (In the case $\mathcal{E} (y, z) = \emptyset$, $\Omega ^s (k \mathcal{E} \alpha) (z) =0$ for $s \geqslant 0$, and the claim is trivially true.) The morphisms in $\mathcal{E} (y, z)$ form a disjoint union of orbits under the right action of $H$. By taking a representative $\beta_i$ from each orbit we have $\mathcal{E} (y, z) = \bigsqcup _{i=1} ^n \beta_i H$. Since $|H _0|$ is not invertible, $y$ is not left regular. By the assumption, $y$ must be right regular. Therefore, by the previous lemma, for each representative morphism $\beta_s$, $1 \leqslant s \leqslant n$, there exist some idempotent $e_i$ such that $\beta_s (k \mathcal{E}) \cong e_s (k \mathcal{E})$ as right projective $k \mathcal{E}$-modules, and $\beta_s (k \mathcal{E}) \alpha \cong e_s (k \mathcal{E}) \alpha$ as vector spaces.

Consider the exact sequence
\begin{equation*}
\xymatrix{ 0 \ar[r] & \Omega ^{i+1} (k \mathcal{E} \alpha) \ar[r] & P^i \ar[r] & \Omega^i (k \mathcal{E} \alpha) \ar[r] & 0,}
\end{equation*}
where we assume inductively that $\Omega ^i (k \mathcal{E} \alpha)$ is generated in degree 1, or equivalently generated by its value on $y$. Thus $P^i \in \text{add} (k \mathcal{E} 1_y [1])$. Observe that the segment of a minimal projective resolution of the $k\mathcal{E}$-module $k \mathcal{E} \alpha$
\begin{equation*}
\xymatrix{ P^{i+1} \ar[r] & P^i \ar[r] & \ldots \ar[r] & P^0 \ar[r] & k\mathcal{E} \alpha \ar[r] & 0}
\end{equation*}
induces a minimal projective resolution of the $kH$-module $kH \alpha$:
\begin{equation*}
\xymatrix{ P^{i+1} (y) \ar[r] & P^i (y) \ar[r] & \ldots \ar[r] & P^0 (y) \ar[r] & kH \alpha \ar[r] & 0.}
\end{equation*}
Thus $\Omega ^j (k\mathcal{E} \alpha)_1 = \Omega^j (k\mathcal{E} \alpha) (y) = \Omega ^j _{kH} (kH \alpha)$ for $1 \leqslant j \leqslant i+1$.

Applying the exact functor $\Hom _{k \mathcal{E}} (k \mathcal{E} 1_y, -)$ to the exact sequence
\begin{equation}
\xymatrix {0 \ar[r] & \Omega ^{i+1} (k \mathcal{E} \alpha) \ar[r] & P^i \ar[r] & \Omega^i (k \mathcal{E} \alpha) \ar[r] & 0},
\end{equation}
we get an exact sequence
\begin{equation*}
\xymatrix {0 \ar[r] & \Omega ^{i+1} (k \mathcal{E} \alpha) (y) \ar[r] & 1_y P^i \ar[r] & \Omega^i (k \mathcal{E} \alpha) (y) \ar[r] & 0},
\end{equation*}
which can be identified with
\begin{equation*}
\xymatrix {0 \ar[r] & \Omega ^{i+1} _{kH} (kH \alpha) \ar[r] & P^i(y) \ar[r] & \Omega^i _{kH} (kH \alpha) \ar[r] & 0}.
\end{equation*}
Applying the exact functor $\Hom_{kH} (\bigoplus _{s=1}^n kHe_s, -)$ to the above sequence, we have another exact sequence
\begin{equation}
0 \rightarrow \bigoplus _{s=1}^n e_s \Omega ^{i+1} _{kH} (kH \alpha) \rightarrow \bigoplus _{s=1}^n e_s P^i(y) \rightarrow \bigoplus _{s=1}^n e_s \Omega^i _{kH} (kH \alpha) \rightarrow 0.
\end{equation}

Since $\Omega^i (k \mathcal{E} \alpha)$ is generated by $\Omega ^i (k \mathcal{E} \alpha) (y) = \Omega^i _{kH} (kH \alpha)$ by the induction hypothesis, the value of $\Omega ^i (k\mathcal{E} \alpha )$ on $z$ is $\sum _{s=1}^n \beta_s \cdot \Omega_{kH} ^i (kH \alpha)$ (this is well defined as $\Omega ^i _{kH} (kH \alpha) \subseteq (kH) ^{\oplus m}$ for some $m \geqslant 0$). We check that this sum is actually direct by the UFP of $\mathcal{E}$. In conclusion,
\begin{equation}
\Omega ^i (k\mathcal{E} \alpha ) (z) = \bigoplus _{s=1}^n \beta_s \cdot \Omega_{kH} ^i (kH \alpha) \cong \bigoplus _{s=1}^n e_s \Omega_{kH} ^i (kH \alpha).
\end{equation}
Similarly, the value of $P^i$ on $z$ is
\begin{equation}
P^i (z) = \bigoplus _{s=1}^n \beta_s \cdot P^i(y) \cong \bigoplus _{s=1}^n e_s P^i(y).
\end{equation}

Restricted to $z$, sequence (4.3.2) gives rise to
\begin{equation}
\xymatrix {0 \ar[r] & \Omega ^{i+1} (k \mathcal{E} \alpha) (z) \ar[r] & P^i(z) \ar[r] & \Omega^i (k \mathcal{E} \alpha)(z) \ar[r] & 0}.
\end{equation}

On one hand, $\bigoplus _{s=1}^n \beta_s \Omega ^{i+1} _{kH} (kH \alpha) \subseteq \Omega ^{i+1} (k \mathcal{E} \alpha) (z)$. On the other hand, we have:
\begin{align*}
& \dim _k \bigoplus _{s=1}^n \beta_s \Omega ^{i+1} _{kH} (kH \alpha) = \dim _k \bigoplus _{s=1}^n e_s \Omega ^{i+1} _{kH} (kH \alpha) \quad \text{ by Lemma 4.3.11} \\
& = \dim_k \bigoplus _{s=1}^n e_s P^i(y) - \dim_k \bigoplus _{s=1}^n e_s \Omega^i _{kH} (kH \alpha) \quad \text{ by sequence (4.3.3)} \\
& = \dim_k P^i (z) - \dim_k \Omega ^i (k\mathcal{E} \alpha ) (z) \quad \text{ by identities (4.3.4) and (4.3.5)} \\
& = \dim_k \Omega ^{i+1} (k \mathcal{E} \alpha) (z) \quad \text{ by sequence (4.3.6)}
\end{align*}
Therefore, we have $\Omega ^{i+1} (k\mathcal{E} \alpha ) (z) = \bigoplus _{s=1}^n \beta_s \Omega ^{i+1} _{kH} (kH \alpha) = \bigoplus _{s=1}^n \beta_s \Omega ^{i+1} _{k \mathcal{E}} (k \mathcal{E} \alpha) (y)$ since $\Omega ^{i+1} _{kH} (kH \alpha) = \Omega ^{i+1} _{k \mathcal{E}} (k \mathcal{E} \alpha) (y)$. That is, the value of $\Omega ^{i+1} (k \mathcal{E} \alpha)$ on $z$ is generated by $\Omega ^{i+1} (k \mathcal{E} \alpha) (y)$. Since $z$ is arbitrary, our claim holds, and the conclusion follows from induction.
\end{proof}

\chapter{Extension algebras of standard modules}
\label{Extension algebras of standard modules}

Let $A$ be a basic finite-dimensional $k$-algebra standardly stratified with respect to a preordered set $(\Lambda, \leqslant)$ indexing all simple modules (up to isomorphism), $\Delta$ be the direct sum of all standard modules, and $\mathcal{F} (\Delta)$ be the category of finitely generated $A$-modules with $\Delta$-filtrations. That is, for each $M \in \mathcal{F} (\Delta)$, there is a chain $0 = M_0 \subseteq M_1 \subseteq \ldots \subseteq M_n = M$ such that $M_i / M_{i-1}$ is isomorphic to an indecomposable summand of $\Delta$, $1 \leqslant i \leqslant n$. Since standard modules of $A$ are relative simple in $\mathcal{F} (\Delta)$, we are motivated to exploit the extension algebra $\Gamma = \Ext _A^{\ast} (\Delta, \Delta)$ of standard modules. These extension algebras were studied in \cite{Abe, Drozd, Klamt, Mazorchuk1, Miemietz}. In this chapter, we are interested in the stratification property of $\Gamma$ with respect to $(\Lambda, \leqslant)$ and $(\Lambda, \leqslant ^{\textnormal{op}})$, and its Koszul property since $\Gamma$ has a natural grading. A particular question is that in which case it is a \textit{generalized Koszul algebra}, i.e., $\Gamma_0$ has a linear projective resolution.

By Gabriel's construction (see Section 3.1), we associate a locally finite $k$-linear category $\mathcal{E}$ to the extension algebra $\Gamma$ such that the category $\Gamma$-mod is equivalent to the category of finitely generated $k$-linear representations of $\mathcal{E}$. We show that the category $\mathcal{E}$ is a directed category with respect to $\leqslant$. That is, the morphism space $\mathcal{E} (x,y) = 0$ whenever $x \nleqslant y$. With this observation, we characterize the stratification property of $\Gamma$ in the first section. Then we define \textit{linearly filtered modules}, and use this terminology to obtain a sufficient condition for $\Gamma$ to be generalized Koszul.

Throughout this chapter $A$ is a finite-dimensional basic associative $k$-algebra with identity 1, where $k$ is algebraically closed.

\section{Stratification property of extension algebras}

The definition of standardly stratified algebras can be found in Section 3.1 of this thesis, and we do not repeat it here. Suppose that $A$ is standardly stratified with respect to a poset $(\Lambda, \leqslant)$. \footnote{For some authors the preordered set $(\Lambda, \leqslant)$ is supposed to be a poset (\cite{Dlab1, Dlab2}) or even a linearly ordered set (\cite{Agoston1, Agoston2}). Algebras standardly stratified in this sense are called \textit{strongly standardly stratified} (\cite{Frisk1, Frisk2}). In this chapter $(\Lambda, \leqslant)$ is supposed to be a poset. Example 5.1.5 in this section explains why we should work with partial orders instead of the more general preorders.} Then standard modules can be described as:
\begin{equation*}
\Delta_{\lambda} = P_{\lambda} / \sum _{\mu > \lambda} \text{tr} _{P_{\mu}} (P_{\lambda}),
\end{equation*}
where tr$_{P_{\mu}} (P_{\lambda})$ is the trace of $P_{\mu}$ in $P _{\lambda}$. Let $\Delta$ be the direct sum of all standard modules and $\mathcal{F} (\Delta)$ be the full subcategory of $A$-mod such that each object in $\mathcal{F} (\Delta)$ has a filtration by standard modules. Clearly, since $A$ is standardly stratified for $\leqslant$, $_AA \in \mathcal{F} (\Delta)$, or equivalently, every indecomposable projective $A$-module has a filtration by standard modules.

Given $M \in \mathcal{F} (\Delta)$ and a fixed filtration $0 = M_0 \subseteq M_1 \subseteq \ldots \subseteq M_n =M$, we define the \textit{filtration multiplicity} $m_{\lambda} = [ M: \Delta_{\lambda}]$ to be the number of factors isomorphic to $\Delta_{\lambda}$ in this filtration. By Lemma 1.4 of \cite{Erdmann}, The filtration multiplicities defined above are independent of the choice of a particular filtration. Moreover, since each standard module has finite projective dimension, we deduce that every $A$-module contained in $\mathcal{F} (\Delta)$ has finite projective dimension. Therefore, the extension algebra $\Gamma = \Ext _A^{\ast} (\Delta, \Delta)$ is finite-dimensional.

\begin{lemma}
Let $\Delta_{\lambda}$, $\Delta_{\mu}$ be standard modules. Then $\Ext _A^n (\Delta _{\lambda}, \Delta_{\mu}) =0$ if $\lambda \nleqslant \mu$ for all $n \geqslant 0$.
\end{lemma}

\begin{proof}
First, we claim $[\Omega^i (\Delta_{\lambda}) : \Delta_{\nu}] = 0$ whenever $\lambda \nleqslant \nu$ for all $i \geqslant 0$, where $\Omega$ is the Heller operator. Indeed, for $i =0$ the conclusion holds clearly. Suppose that it is true for all $i \leqslant n$ and consider $\Omega ^{n+1} (\Delta _{\lambda})$. We have the following exact sequence:
\begin{equation*}
\xymatrix {0 \ar[r] & \Omega ^{n+1} (\Delta _{\lambda}) \ar[r] & P \ar[r] & \Omega^n (\Delta _{\lambda}) \ar[r] & 0.}
\end{equation*}
By the induction hypothesis, $[\Omega ^n (\Delta _{\lambda}) : \Delta_{\nu}] = 0$ whenever $\lambda \nleqslant \nu$. Therefore, $[P : \Delta_{\nu}] = 0$ whenever $\lambda \nleqslant \nu$, and hence $[\Omega ^{n+1} (\Delta _{\lambda}) : \Delta_{\nu}] = 0$ whenever $\lambda \nleqslant \nu$. The claim is proved by induction.

The above exact sequence induces a surjection $\Hom _A (\Omega ^n (\Delta _{\lambda}), \Delta_{\mu}) \rightarrow \Ext ^n_A (\Delta _{\lambda}, \Delta_{\mu})$. Thus it suffices to show $\Hom _A (\Omega ^n (\Delta _{\lambda}), \Delta_{\mu}) =0$ for all $n \geqslant 0$ if $\lambda \nleqslant \mu$. By the above claim, all filtration factors $\Delta_{\nu}$ of $\Omega ^n (\Delta _{\lambda})$ satisfy $\nu \geqslant \lambda$, and hence $\nu \nleqslant \mu$. But $\Hom _A (\Delta_{\nu}, \Delta_{\mu}) = 0$ whenever $\nu \nleqslant \mu$. The conclusion follows.
\end{proof}

Now let $\Gamma = \Ext _A^{\ast} (\Delta, \Delta)$. This is a graded finite-dimensional algebra equipped with a natural grading. In particular, $\Gamma_0 = \End _A (\Delta)$. For each $\lambda \in \Lambda$, $\Delta_{\lambda}$ is an indecomposable $A$-module. Therefore, up to isomorphism, the indecomposable projective $\Gamma$-modules are exactly those $\Ext _A^{\ast} (\Delta_{\lambda}, \Delta)$, $\lambda \in \Lambda$.

The associated $k$-linear category $\mathcal{E}$ of $\Gamma$ has the following structure: $\Ob \mathcal{E} = \{ \Delta _{\lambda} \} _{\lambda \in \Lambda}$; the morphism space $\mathcal{E} (\Delta _{\lambda}, \Delta_{\mu}) = \Ext _A^{\ast} (\Delta_{\lambda}, \Delta_{\mu})$. The partial order $\leqslant$ induces a partial order on $\Ob \mathcal{E}$ which we still denote by $\leqslant$, namely, $\Delta _{\lambda} \leqslant \Delta_{\mu}$ if and only if $\lambda \leqslant \mu$.

\begin{proposition}
The associated category $\mathcal{E}$ of $\Gamma$ is directed with respect to $\leqslant$. In particular, $\Gamma$ is standardly stratified with respect to $\leqslant ^{\textnormal{op}}$ and all standard modules are projective.
\end{proposition}

\begin{proof}
The first statement follows from the previous lemma. The second statement is also clear. Indeed, since $\Gamma$ is directed with respect to $\leqslant$, $e_{\mu} \Gamma e_{\lambda} \cong \Hom _{\Gamma} (Q_{\mu}, Q_{\lambda}) = 0$ if $\mu \ngeqslant \lambda$, where $Q_{\mu}, Q_{\lambda}$ are projective $\Gamma$-modules. Thus tr$ _{Q_{\mu}} (Q_{\lambda}) = 0$ whenever $\mu \ngeqslant \lambda$, or equivalently, tr$ _{Q_{\mu}} (Q_{\lambda}) = 0$ whenever $\mu \nleqslant ^{\textnormal{op}} \lambda$. Therefore, all standard modules with respect to $\leqslant ^{\textnormal{op}}$ are projective.
\end{proof}

The following theorem characterize the stratification property of $\Gamma$.

\begin{theorem}
If $A$ is standardly stratified for $(\Lambda, \leqslant)$, then $\mathcal{E}$ is a directed category with respect to $\leqslant$ and is standardly stratified for $\leqslant ^{\textnormal{op}}$. Moreover, $\mathcal{E}$ is standardly stratified for $\leqslant$ if and only if for all $\lambda, \mu \in \Lambda$ and $s \geqslant 0$, $\Ext _A^s (\Delta_{\lambda}, \Delta_{\mu})$ is a projective $\End _A (\Delta_{\mu})$-module.
\end{theorem}

\begin{proof}
The first statement follows from Proposition 5.1.2 and the second statement follows from Theorem 3.1.6.
\end{proof}

In the case that $A$ is quasi-hereditary, we have:

\begin{corollary}
If $A$ is a quasi-hereditary algebra with respect to $\leqslant$, then $\Gamma$ is quasi-hereditary with respect to both $\leqslant$ and $\leqslant ^{\textnormal{op}}$.
\end{corollary}

\begin{proof}
We have shown that $\Gamma$ is standardly stratified with respect to $\leqslant ^{\textnormal{op}}$ and the corresponding standard modules $_{\Gamma} \Delta _{\lambda} \cong \Gamma 1_{\lambda}$ for $\lambda \in \Lambda$. Therefore,
\begin{equation*}
\End _{\Gamma} (_{\Gamma} \Delta _{\lambda}) = \End _{\Gamma} (\Gamma 1_{\lambda}) \cong 1_{\lambda} \Gamma 1_{\lambda} = \Ext _A^{\ast} (\Delta_{\lambda}, \Delta_{\lambda}) = \End_A (\Delta_{\lambda}) \cong k
\end{equation*}
since $A$ is quasi-hereditary. So $\Gamma$ is also quasi-hereditary with respect to $\leqslant^{\textnormal{op}}$.

Now consider the stratification property of $\Gamma$ with respect to $\leqslant$. The associated category $\mathcal{E}$ is directed with respect to $\leqslant$. Since $\Ext _A^{\ast} (\Delta_{\mu}, \Delta_{\mu}) = \End _A (\Delta_{\mu}) \cong k$ for all $\mu \in \Lambda$, $\mathcal{E} (\Delta_{\lambda}, \Delta_{\mu}) = \Ext _A^{\ast} (\Delta_{\lambda}, \Delta_{\mu})$ is a projective $k$-module for each pair $\lambda, \mu \in \Lambda$. Therefore, $\mathcal{E}$ is standardly stratified for $\leqslant$ by the previous theorem. Moreover, by Proposition 3.1.5, the standard modules of $\mathcal{E}$ (or the standard modules of $\Gamma$) are precisely indecomposable summands of $\bigoplus _{\lambda \in \Lambda} \Ext _A^{\ast} (\Delta_{\lambda}, \Delta_{\lambda}) \cong \bigoplus _{\lambda \in \Lambda} k_{\lambda}$. Clearly, for $\lambda \in \Lambda$, $\End _{\Gamma} (k_{\lambda}, k_{\lambda}) \cong k$, so $\Gamma$ is quasi-hereditary with respect to $\leqslant$.
\end{proof}

The following example from 8.2 in \cite{Frisk2} illustrates why we should assume that $\leqslant$ is a partial order rather than a preorder. Indeed, in a preordered set $(\Lambda, \leqslant)$ we cannot deduce $x = y$ if $x \leqslant y$ and $y \leqslant x$.

\begin{example}
Let $A$ be the path algebra of the following quiver with relations $\alpha_1 \beta_1 = \alpha_2 \beta_2 = \alpha_2 \alpha_1 = \beta_1 \beta_2 =0$. Define a preorder $\leqslant$ by letting $x \leqslant y < z$ and $y \leqslant x < z$.
\begin{equation*}
\xymatrix{x \ar@/^/[r] ^{\alpha_1} & y \ar@/^/[r] ^{\alpha_2} \ar@/^/[l] ^{\beta_1} & z \ar@/^/[l] ^{\beta_2}}.
\end{equation*}
Projective modules and standard modules are described as follows:
\begin{equation*}
P_x \cong \Delta_x = \begin{matrix} x \\ y \\ x \end{matrix} \qquad P_y = \begin{matrix} & y & \\ x & & z \\ & & y \end{matrix} \qquad \Delta_y = \begin{matrix} y \\ x \end{matrix} \qquad P_z \cong \Delta_z = \begin{matrix} z \\ y \end{matrix}
\end{equation*}
Then the associated category $\mathcal{E}$ of $\Gamma = \Ext _A^{\ast} (\Delta, \Delta)$ is not a directed category since both $\Hom_A (\Delta_x, \Delta_y)$ and $\Hom _A (\Delta_y, \Delta_x)$ are nonzero.
\end{example}

Now we generalize the above results to \textit{Ext-Projective Stratifying Systems} (EPSS). From now on the algebra $A$ is finite-dimensional and basic, but we do not assume that it is standardly stratified for some partial order, as we did before. The EPSS we describe in this chapter is indexed by a finite poset $(\Lambda, \leqslant)$ rather than a linearly ordered set as in \cite{Marcos1, Marcos2}. However, this difference is not essential and all properties described in \cite{Marcos1, Marcos2} can be applied to our situation with suitable modifications.

\begin{definition}
(Definition 2.1 in \cite{Marcos2}) Let $\underline{\Theta} = \{ \Theta_{\lambda} \}_{\lambda \in \Lambda}$ be a set of nonzero $A$-modules and $\underline{Q} = \{ Q_{\lambda} \} _{\lambda \in \Lambda}$ be a set of indecomposable $A$-modules, both of which are indexed by a finite poset $(\Lambda, \leqslant)$. We call $(\underline {\Theta}, \underline{Q})$ an EPSS if the following conditions are satisfied:
\begin{enumerate}
\item $\Hom _A (\Theta_{\lambda}, \Theta_{\mu}) =0$ if $\lambda \nleqslant \mu$;
\item for each $\lambda \in \Lambda$, there is an exact sequence $0 \rightarrow K_{\lambda} \rightarrow Q_{\lambda} \rightarrow \Theta_{\lambda} \rightarrow 0$ such that $K_{\lambda}$ has a filtration only with factors isomorphic to $\Theta_{\mu}$ satisfying $\mu > \lambda$;
\item for every $A$-module $M \in \mathcal{F} (\underline {\Theta})$ and $\lambda \in \Lambda$, $\Ext _A^1 (Q_{\lambda}, M) =0$.
\end{enumerate}
\end{definition}

We denote $\Theta$ and $Q$ the direct sums of all $\Theta _{\lambda}$'s and $Q_{\lambda}$'s respectively, $\lambda \in \Lambda$.

Given an EPSS $(\underline {\Theta}, \underline{Q})$ indexed by $(\Lambda, \leqslant)$, $(\underline {\Theta}, \leqslant)$ is a \textit{stratifying system} (SS) since $\Hom _A (\Theta_{\lambda}, \Theta_{\mu}) =0$ if $\lambda \nleqslant \mu$, and $\Ext _A^1 (\Theta_{\lambda}, \Theta_{\mu}) =0$ if $\lambda \nless \mu$. Conversely, given a stratifying system $(\underline {\Theta}, \leqslant)$, we can construct an EPSS $(\underline {\Theta}, \underline{Q})$ unique up to isomorphism. See \cite{Marcos2} for more details. Moreover, as described in \cite{Marcos2}, the algebra $B = \End _A (Q) ^{\textnormal{op}}$ is standardly stratified, and the functor $e_Q = \Hom _A (Q, -)$ gives an equivalence of exact categories between $\mathcal{F} (\Theta)$ and $\mathcal{F} (_B \Delta)$.

To study the extension algebra $\Gamma = \Ext _A^{\ast} (\Theta, \Theta)$, one may want to use projective resolutions of $\Theta$. However, different from the situation of standardly stratified algebras, the regular module $_AA$ in general might not be contained in $\mathcal{F} (\Theta)$. If we suppose that $_AA$ is contained in $\mathcal{F} (\Theta)$ (in this case the stratifying system $(\underline {\Theta}, \leqslant)$ is said to be \textit{standard}) and $\mathcal{F} (\Theta)$ is closed under the kernels of surjections, then by Theorem 2.6 in \cite{Marcos1} $A$ is standardly stratified for $\leqslant$ and those $\Theta_{\lambda}$'s coincide with standard modules of $A$. This situation has been completely discussed previously. Alternately, we use the \textit{relative projective resolutions} whose existence is guaranteed by the following proposition.

\begin{proposition}
(Corollary 2.11 in \cite{Marcos2}) Let $(\underline {\Theta}, \underline{Q})$ be an EPSS indexed by a finite poset $(\Lambda, \leqslant)$. Then for each $M \in \mathcal{F} (\Theta)$, there is a finite resolution
\begin{equation*}
\xymatrix {0 \ar[r] & Q^d \ar[r] & \ldots \ar[r] & Q^0 \ar[r] & M \ar[r] & 0}
\end{equation*}
such that each kernel is contained in $\mathcal{F} (\Theta)$, where $0 \neq Q^i \in \text{add} (Q)$ for $0 \leqslant i \leqslant d$.
\end{proposition}

The number $d$ in this resolution is called the \textit{relative projective dimension} of $M$.

\begin{proposition}
Let $(\underline {\Theta}, \underline{Q})$ be an EPSS indexed by a finite poset $(\Lambda, \leqslant)$ and $d$ be the relative projective dimension of $\Theta$. If $\Ext _A^s (Q, \Theta) = 0$ for all $s \geqslant 1$, then for $M, N \in \mathcal{F} (\Theta)$ and $s > d$, $\Ext _A^s (M, N) =0$.
\end{proposition}

\begin{proof}
Since $M$ and $N$ are contained in $\mathcal{F} (\Theta)$, it is enough to show that $\Ext _A^s (\Theta, \Theta) = 0$ for all $s > d$. If $d =0$, then $Q = \Theta$ and the conclusion holds trivially. So we suppose $d \geqslant 1$. Applying the functor $\Hom _A (-, \Theta)$ to the exact sequence
\begin{equation*}
\xymatrix {0 \ar[r] & K_1 \ar[r] & Q \ar[r] & \Theta \ar[r] & 0}
\end{equation*}
we get a long exact sequence. In particular, from the segment
\begin{equation*}
\xymatrix {\Ext _A^{s-1} (Q, \Theta) \ar[r] & \Ext _A^{s-1} (K_1, \Theta) \ar[r] & \Ext _A^s (\Theta, \Theta) \ar[r] & \Ext _A^s (Q, \Theta)}
\end{equation*}
of this long exact sequence we deduce that $\Ext _A^s (\Theta, \Theta) \cong \Ext _A^{s-1} (K_1, \Theta)$ since the first and last terms are 0. Now applying $\Hom _A (-, \Theta)$ to the exact sequence
\begin{equation*}
\xymatrix {0 \ar[r] & K_2 \ar[r] & Q^1 \ar[r] & K_1 \ar[r] & 0}
\end{equation*}
we get $\Ext _A^{s-1} (K_1, \Theta) \cong \Ext _A^{s-2} (K_2, \Theta)$. Thus $\Ext _A^s (\Theta, \Theta) \cong \Ext _A^{s-d} (K_d, \Theta)$ by induction. But $K_d \cong Q^d \in \text{add} (Q)$. The conclusion follows.
\end{proof}

Thus $\Gamma = \Ext _A^{\ast} (\Theta, \Theta)$ is a finite-dimensional algebra under the given assumption.

There is a natural partition on the finite poset $(\Lambda, \leqslant)$ as follows: let $\Lambda_1$ be the subset of all minimal elements in $\Lambda$, $\Lambda_2$ be the subset of all minimal elements in $\Lambda \setminus \Lambda_1$, and so on. Then $\Lambda = \sqcup_{i \geqslant 1} \Lambda_i$. With this partition, we can introduce a \textit{height function} $h: \Lambda \rightarrow \mathbb{N}$ in the following way: for $\lambda \in \Lambda_i \subseteq \Lambda$, $i \geqslant 1$, we define $h(\lambda) = i$.

For each $M \in \mathcal{F} (\Theta)$, we define supp$(M)$ to be the set of elements $\lambda \in \Lambda$ such that $M$ has a $\Theta$-filtration in which there is a factor isomorphic to $\Theta_{\lambda}$. For example, supp$ (\Theta_{\lambda}) = \{ \lambda \}$. By Lemma 2.6 in \cite{Marcos2}, the multiplicities of factors of $M$ is independent of the choice of a particular $\Theta$-filtration. Therefore, supp$(M)$ is well defined. We also define $\min (M) = \min (\{ h(\lambda) \mid \lambda \in \text{supp} (M) \})$.

\begin{lemma}
Let $(\underline {\Theta}, \underline{Q})$ be an EPSS indexed by a finite poset $(\Lambda, \leqslant)$. For each $M \in \mathcal{F} (\Theta)$, there is an exact sequence $0 \rightarrow K_1 \rightarrow Q^0 \rightarrow M$ such that $K_1 \in \mathcal{F} (\Theta)$ and $\min (K_1) > \min (M)$, where $Q^0 \in \text{add} (Q)$.
\end{lemma}

\begin{proof}
This is Proposition 2.10 in \cite{Marcos2} which deals with the special case that $\Lambda$ is a linearly ordered set. The general case can be proved similarly by observing the fact that $\Ext _A^1 (\Theta_{\lambda}, \Theta_{\mu}) = 0$ if $h (\lambda) = h (\mu)$.
\end{proof}

By this lemma, the relative projective dimension of every $M \in \mathcal{F} (\Theta)$ cannot exceed the length of the longest chain in $\Lambda$.

As before, we let $\mathcal{E}$ be the $k$-linear category associated to $\Gamma = \Ext _A^{\ast} (\Theta, \Theta)$.

\begin{theorem}
Let $(\underline {\Theta}, \underline{Q})$ be an EPSS indexed by a finite poset $(\Lambda, \leqslant)$. such that $\Ext_A ^i (Q, \Theta) =0$ for all $i \geqslant 1$. Then $\mathcal{E}$ is a directed category with respect to $\leqslant$ and is standardly stratified for $\leqslant ^{\textnormal{op}}$. Moreover, it is standardly stratified for $\leqslant$ if and only if for all $s \geqslant 0$, $\Ext _A^s (\Theta_{\lambda}, \Theta_{\mu})$ is a projective $\End _A (\Theta_{\mu})$-module, $\lambda, \mu \in \Lambda$.
\end{theorem}

\begin{proof}
We only need to show that $\mathcal{E}$ is a directed category with respect to $\leqslant$ since the other statements can be proved as in Theorem 5.1.3. We know $\Hom _A (\Theta_{\lambda}, \Theta_{\mu}) = 0$ if $\lambda \nleqslant \mu$ and $\Ext _A^1 (\Theta_{\lambda}, \Theta_{\mu}) = 0$ for all $\lambda \nless \mu$. Therefore, it suffices to show that for all $s \geqslant 2$, $\Ext _A^s (\Theta_{\lambda}, \Theta_{\mu}) = 0$ if $\lambda \nless \mu$.

By Proposition 5.1.7 and Lemma 5.1.9, $\Theta_{\lambda}$ has a relative projective resolution
\begin{equation*}
\xymatrix {0 \ar[r] & Q^d \ar[r] ^{f_d} & \ldots \ar[r] ^{f_1} & Q^0 \ar[r] ^{f_0} & \Theta_{\lambda} \ar[r] & 0}
\end{equation*}
such that for each map $f_t$, min($K_t) > \text{min} (K_{t-1})$, where $K_t = \text{Ker} (f_t)$ and $1 \leqslant t \leqslant d$. By Proposition 5.1.8, $\Ext _A^s (\Theta_{\lambda}, \Theta_{\mu}) = 0$ if $s > d$; if $2 \leqslant s \leqslant d$, we have $\Ext _A^s (\Theta_{\lambda}, \Theta_{\mu}) \cong \Ext _A^1 (K_{s-1}, \Theta_{\mu})$. But we have chosen
\begin{equation*}
\min (K_{s-1}) > \min (K_{s-2}) > \ldots > \min (\Theta_{\lambda}) = h(\lambda) \geqslant h(\mu).
\end{equation*}
Thus each factor $\Theta_{\nu}$ appearing in a $\Theta$-filtration of $K_{s-1}$ satisfies $h(\nu) > h(\mu)$, and hence $\nu \nleqslant \mu$. Since $\Ext _A^1 (\Theta_{\nu}, \Theta_{\mu}) =0$ for all $\nu \nleqslant \mu$, we deduce
\begin{equation*}
\Ext _A^s (\Theta_{\lambda}, \Theta_{\mu}) \cong \Ext _A^1 (K_{s-1}, \Theta_{\mu}) =0.
\end{equation*}
This finishes the proof.
\end{proof}

The following corollary is a generalization of Corollary 5.1.4.

\begin{corollary}
Let $(\underline {\Theta}, \underline{Q})$ be an EPSS indexed by a finite poset $(\Lambda, \leqslant)$. If for all $s \geqslant 1$ and $\lambda \in \Lambda$ we have $\Ext _A ^s (Q, \Theta) = 0$ and $\End_A (\Theta_{\lambda}, \Theta_{\lambda}) \cong k$, then $\Gamma$ is quasi-hereditary with respect to both $\leqslant$ and $\leqslant ^{\textnormal{op}}$.
\end{corollary}

\begin{proof}
This can be proved as Corollary 5.1.4.
\end{proof}

\section{Koszul property of extension algebras}

If $A$ is a graded standardly stratified algebra with $A_0$ a semisimple algebra, then $B = \Ext _A^{\ast} (A_0, A_0)$ is a Koszul algebra, too. Moreover, the functor $\Ext _A^{\ast} (-, A_0)$ gives an equivalence between the category of linear $A$-modules and the category of linear $B$-modules. However, even if $A$ is quasi-hereditary with respect to a partial order $\leqslant$, $B$ might not be quasi-hereditary with respect to $\leqslant$ or $\leqslant ^{\textnormal{op}}$ (compared to Corollary 5.1.4). This problem has been considered in \cite{Agoston1, Agoston2, Mazorchuk3, Mazorchuk4}.

On the other hand, if $A$ is quasi-hereditary with respect to $\leqslant$, we have shown that the extension algebra $\Gamma = \Ext _A^{\ast} (\Delta, \Delta)$ is quasi-hereditary with respect to both $\leqslant$ and $\leqslant ^{\textnormal{op}}$. But $\Gamma$ is in general not a Koszul algebra (even in the generalized sense). In this section we want to get a sufficient condition for $\Gamma$ to be a generalized Koszul algebra.

We work in the context of EPSS described in last section. Let $(\underline {\Theta}, \underline{Q})$ be an EPSS indexed by a finite poset $(\Lambda, \leqslant)$; $Q = \bigoplus _{\lambda \in \Lambda} Q_{\lambda}$ and $\Theta = \bigoplus _{\lambda \in \Lambda} \Theta_{\lambda}$. \textbf{We insist the following conditions: $\Ext \mathbf{_A^s (Q, \Theta) =0}$ for all $\mathbf{s \geqslant 1}$; each $\Theta_{\lambda}$ has a simple top $S_{\lambda}$; and $S_{\lambda} \ncong S_{\mu}$ for $\lambda \neq \mu$.} These conditions are always true for the classical stratifying system of a standardly stratified basic algebra. In particular, in that case $Q = _AA$.

\begin{proposition}
Let $0 \neq M \in \mathcal{F} (\Theta)$ and $i = \text{min (M)}$. Then there is an exact sequence
\begin{equation}
\xymatrix{ 0 \ar[r] & M[1] \ar[r] & M \ar[r] & \bigoplus _{h (\lambda) = i } \Theta_{\lambda} ^{\oplus m_{\lambda}} \ar[r] & 0}
\end{equation}
such that $M[1] \in \mathcal{F} (\Theta)$ and $\text{min} (M[1]) > \text{min} (M)$.
\end{proposition}

\begin{proof}
This is Proposition 2.9 in \cite{Marcos2} which deals with the special case that $\Lambda$ is a linearly ordered set. The general case can be proved similarly by observing the fact that $\Ext _A^1 (\Theta_{\lambda}, \Theta_{\mu}) = 0$ if $h (\lambda) = h (\mu)$.
\end{proof}

It is clear that $m_{\lambda} = [M: \Theta_{\lambda}]$. Based on this proposition, we make the following definition:

\begin{definition}
Let $M \in \mathcal{F} (\Theta)$ with $\text{min} (M) = i$. We say $M$ is generated in height $i$ if in sequence (5.2.1) we have
\begin{equation*}
\Top (M) = M / \rad M \cong \Top \Big{(}\bigoplus _{h (\lambda) = i } \Theta_{\lambda} ^{\oplus m_{\lambda}} \Big{)} = \bigoplus _{h (\lambda) = i } S_{\lambda} ^{\oplus m_{\lambda}}.
\end{equation*}
\end{definition}

We introduce some notation: if $M \in \mathcal{F} (\Theta)$ is generated in height $i$, then define $M_i = \bigoplus _{h (\lambda) = i } \Theta_{\lambda} ^{\oplus m_{\lambda}}$ in sequence (5.2.1). If $M[1]$ is generated in some height $j$, we can define $M[2] = M[1][1]$ and $M[1]_j$ in a similar way. This procedure can be repeated.

\begin{proposition}
Let $0 \rightarrow L \rightarrow M \rightarrow N \rightarrow 0$ be an exact sequence in $\mathcal{F} (\Theta)$. If $M$ is generated in height $i$, so is $N$. Conversely, if both $L$ and $N$ are generated in height $i$, then $M$ is generated in height $i$ as well.
\end{proposition}

\begin{proof}
We always have $\Top (N) \subseteq \Top (M)$ and $\Top (M) \subseteq \Top (L) \oplus \Top (N)$. The conclusion follows from these inclusions and the rightmost identity in the above definition.
\end{proof}

Notice that $[Q_{\lambda} : \Theta_{\lambda}] = 1$ and $[Q_{\lambda} : \Theta_{\mu}] = 0$ for all $\mu \ngeqslant \lambda$. We claim that $Q_{\lambda}$ is generated in height $h(\lambda)$ for $\lambda \in \Lambda$. Indeed, the algebra $B = \End _A (Q)^{\textnormal{op}}$ is a standardly stratified algebra, with projective modules $\Hom _A (Q, Q_{\lambda})$ and standard modules $\Hom _A (Q, \Theta_{\lambda})$, $\lambda \in \Lambda$. Moreover, the functor $\Hom _A (Q, -)$ gives an equivalence between $\mathcal{F} (\Theta) \subseteq A$-mod and $\mathcal{F} (_B \Delta) \subseteq B$-mod. Using this equivalence and the standard filtration structure of projective $B$-modules we deduce the conclusion.

\begin{lemma}
If $M \in \mathcal{F} (\Theta)$ is generated in height $i$ with $[M: \Theta_{\lambda} ] = m_{\lambda}$, then $M$ has a relative projective cover $Q^i \cong \bigoplus _{h (\lambda) = i} Q_{\lambda} ^{\oplus m_{\lambda}}$.
\end{lemma}

\begin{proof}
There is a surjection $f: M \rightarrow \bigoplus _{h(\lambda) = i} \Theta_{\lambda} ^{\oplus m_{\lambda}}$ by Proposition 5.2.1. Consider the following diagram:
\begin{equation*}
\xymatrix{ & \bigoplus _{h (\lambda) = i} Q_{\lambda} ^{\oplus m_{\lambda}} \ar[d]^p \ar@{-->}[dl]_q \\
M \ar[r]^f & \bigoplus _{h (\lambda) = i} \Theta_{\lambda} ^{\oplus m_{\lambda}} \ar[r] & 0.}
\end{equation*}
Since $Q^i$ is projective in $\mathcal{F} (\Theta)$, the projection $p$ factors through the surjection $f$. In particular, $\Top \Big{(} \bigoplus _{h (\lambda) = i} \Theta_{\lambda} ^{\oplus m_{\lambda}} \Big{)} = \bigoplus _{h (\lambda) = i} S_{\lambda} ^{\oplus m_{\lambda}}$ is in the image of $fq$. Since $M$ is generated in height $i$, $f$ induces an isomorphism between $\Top (M)$ and $\Top \Big{(} \bigoplus _{h (\lambda) = i} \Theta_{\lambda} ^{\oplus m_{\lambda}} \Big{)}$. Thus $\Top (M)$ is in the image of $q$, and hence $q$ is surjective. It is clear that $q$ is minimal, so $Q^1 = \bigoplus _{h (\lambda) = i} Q_{\lambda} ^{\oplus m_{\lambda}}$ is a relative projective cover of $M$. The uniqueness follows from Proposition 8.3 in \cite{Webb1}.
\end{proof}

We use $\Omega _{\Theta} ^i (M)$ to denote the \textit{i-th relative syzygy} of $M$. Actually, for every $M \in \mathcal{F} (\Theta)$ there is always a relative projective cover by Proposition 8.3 in \cite{Webb1}.

The following definition is an analogue of that of Koszul modules.

\begin{definition}
An $A$-module $M \in \mathcal{F} (\Theta)$ is said to be linearly filtered if there is some $i \in \mathbb{N}$ such that $\Omega _{\Theta} ^s (M)$ is generated in height $i + s$ for $s \geqslant 0$.
\end{definition}

Equivalently, $M \in \mathcal{F} (\Theta)$ is linearly filtered if and only if it is generated in height $i$  and has a relative projective resolution
\begin{equation*}
\xymatrix{ 0 \ar[r] & Q^l \ar[r] & Q^{l-1} \ldots \ar[r] & Q^{i+1} \ar[r] & Q^i \ar[r] & M \ar[r] & 0}
\end{equation*}
such that each $Q^s$ is generated in height $s$, $i \leqslant s \leqslant l$.

We remind the reader that there is a common upper bound for the relative projective dimensions of modules contained in $\mathcal{F} (\Theta)$, which is the length of the longest chains in the finite poset $(\Lambda, \leqslant)$. It is also clear that if $M$ is linearly filtered, so are all relative syzygies and direct summands. In other words, the subcategory $\mathcal{LF} (\Theta)$ constituted of linearly filtered modules contains all relative projective modules, and is closed under summands and relative syzygies. But in general it is not closed under extensions, kernels of epimorphisms and cokernels of monomorphisms.

The following proposition is an analogue of Proposition 2.1.8

\begin{proposition}
Let $0 \rightarrow L \rightarrow M \rightarrow N \rightarrow 0$ be an exact sequence in $\mathcal{F} (\Theta)$ such that all terms are generated in height $i$. If $L$ is linearly filtered, then $M$ is linearly filtered if and only if $N$ is linearly filtered.
\end{proposition}

\begin{proof}
Let $m_{\lambda} = [M : \Theta_{\lambda}]$, $l_{\lambda} = [L : \Theta_{\lambda}]$ and $n_{\lambda} = [N : \Theta_{\lambda}]$. By the previous lemma, we get the following commutative diagram:

\begin{equation*}
\xymatrix{ & 0 \ar[d] & 0 \ar[d] & 0 \ar[d] & \\
0 \ar[r] & \Omega _{\Theta}(L) \ar[r] \ar[d] & \Omega _{\Theta}(M) \ar[r] \ar[d] & \Omega _{\Theta}(N) \ar[r] \ar[d] & 0 \\
0 \ar[r] & \bigoplus _{h(\lambda) = i} Q_{\lambda}^{\oplus l_{\lambda}} \ar[r] \ar[d] & \bigoplus _{h(\lambda) = i} Q_{\lambda} ^{\oplus m_{\lambda}} \ar[r] \ar[d] & \bigoplus _{h(\lambda) = i} Q_{\lambda} ^{\oplus n_{\lambda}} \ar[r] \ar[d] & 0 \\
0 \ar[r] & L \ar[r] \ar[d] & M \ar[r] \ar[d] & N \ar[r] \ar[d] & 0 \\
 & 0 & 0 & 0 &}
\end{equation*}
Since $\Omega _{\Theta} (L)$ is generated in height $i+1$, by Proposition 5.2.3, $\Omega _{\Theta} (M)$ is generated in height $i+1$ if and only if $\Omega _{\Theta} (N)$ is generated in height $i+1$. Replacing $L$, $M$ and $N$ by $\Omega _{\Theta} (L)$, $\Omega _{\Theta}(M)$ and $\Omega _{\Theta} (N)$ respectively, we conclude that $\Omega _{\Theta}^2 (M)$ is generated in height $i+2$ if and only if $\Omega _{\Theta}^2 (N)$ is generated in height $i+2$. The conclusion follows from induction.
\end{proof}

The following corollary is an analogue of Proposition 2.1.9.

\begin{corollary}
Suppose that $M \in \mathcal{F} (\Theta)$ is generated in height $i$ and linearly filtered. If $\bigoplus _{h(\lambda) = i} \Theta_{\lambda}$ is linearly filtered, then $M[1]$ is generated in height $i+1$ and linearly filtered.
\end{corollary}

\begin{proof}
Clearly $\bigoplus _{h(\lambda) = i} \Theta_{\lambda}$ is generated in height $i$. Let $m_{\lambda} = [M : \Theta_{\lambda}]$. Notice that both $M$ and $\bigoplus _{h(\lambda) = i} \Theta_{\lambda} ^{\oplus m_{\lambda}}$ have projective cover $\bigoplus _{h(\lambda) = i} Q_{\lambda} ^{\oplus m_{\lambda}}$. Thus the exact sequence
\begin{equation*}
\xymatrix{ 0 \ar[r] & M[1] \ar[r] & M \ar[r] & \bigoplus _{h(\lambda) = i} \Theta_{\lambda} ^{\oplus m_{\lambda}} \ar[r] & 0}
\end{equation*}
induces the following diagram:
\begin{equation*}
\xymatrix{ & \Omega _{\Theta}(M) \ar@{^{(}->} [r] \ar[d] & \Omega _{\Theta} \Big{(} \bigoplus _{h(\lambda) = i} \Theta_{\lambda} ^{\oplus m_{\lambda}} \Big{)} \ar@{->>}[r] \ar[d] & M[1] \\
& \bigoplus _{h(\lambda) = i} Q_{\lambda} ^{\oplus m_{\lambda}} \ar@{=}[r] \ar[d] & \bigoplus _{h(\lambda) = i} Q_{\lambda} ^{\oplus m_{\lambda}} \ar[d] \\
M[1] \ar@{^{(}->}[r] & M \ar@{->>}[r] & \bigoplus _{h(\lambda) = i} \Theta_{\lambda} ^{\oplus m_{\lambda}}.}
\end{equation*}

Consider the top sequence. Since both $\Omega _{\Theta} (\bigoplus _{h(\lambda) = i} \Theta_{\lambda} ^{\oplus m_{\lambda}})$ and $\Omega _{\Theta} (M)$ are generated in height $i+1$ and linearly filtered, $M[1]$ is also generated in height $i+1$ and linearly filtered by Propositions 5.2.3 and 5.2.6.
\end{proof}

These results tell us that linearly filtered modules have properties similar to those of linear modules of graded algebras.

\begin{lemma}
Let $M \in \mathcal{F} (\Theta)$ be generated in height $i$ and $m_{\lambda} = [M : \Theta_{\lambda}]$. If $\Hom _A (Q, \Theta) \cong \Hom _A (\Theta, \Theta)$ as $\End _A (\Theta)$-modules, then
\begin{equation*}
\Hom _A (M, \Theta) \cong \Hom _A (\bigoplus _{h(\lambda) = i} Q_{\lambda} ^{\oplus m_{\lambda}}, \Theta) \cong \Hom _A (\bigoplus _{h(\lambda) = i} \Theta_{\lambda} ^{\oplus m_{\lambda}}, \Theta).
\end{equation*}
\end{lemma}

\begin{proof}
We claim that $\Hom_A (Q, \Theta) \cong \End _A (\Theta)$ implies $\Hom_A (Q_{\lambda}, \Theta) \cong \Hom _A (\Theta_{\lambda}, \Theta)$ for every $\lambda \in \Lambda$. Then the second isomorphism follows immediately. First, note that $\End _A (\Theta)$ is a basic algebra with $n$ non-zero indecomposable summands $\Hom _A (\Theta_{\lambda}, \Theta)$, $\lambda \in \Lambda$, where $n$ is the cardinal number of $\Lambda$. But $\Hom _A (Q, \Theta) \cong \bigoplus _{\lambda \in \Lambda} \Hom _A (Q_{\lambda}, \Theta)$ has at least $n$ non-zero indecomposable summands. If $\Hom_A (Q, \Theta) \cong \End _A (\Theta)$, by Krull-Schmidt theorem, $\Hom _A (Q_{\lambda}, \Theta)$ must be indecomposable, and is isomorphic to some $\Hom _A (\Theta_{\mu}, \Theta)$.

If $\lambda \in \Lambda$ is maximal, $\Hom _A (\Theta_{\lambda}, \Theta) \cong \Hom _A (Q_{\lambda}, \Theta)$ since $Q_{\lambda} \cong \Theta_{\lambda}$. Define $\Lambda_1$ to be the subset of maximal elements in $\Lambda$ and consider $\lambda_1 \in \Lambda \setminus \Lambda_1$ which is maximal. We have
\begin{equation*}
\Hom _A (Q_{\lambda_1}, \Theta) \ncong \Hom _A (\Theta_{\lambda}, \Theta) \cong \Hom _A (Q_{\lambda}, \Theta)
\end{equation*}
for every $\lambda \in \Lambda_1$ since $\End _A (\Theta)$ is a basic algebra. Therefore, $\Hom _A (Q_{\lambda_1}, \Theta)$ must be isomorphic to some $\Hom _A (\Theta_{\mu}, \Theta)$ with $\mu \in \Lambda \setminus \Lambda_1$. But $\Hom _A (\Theta_{\mu}, \Theta)$ contains a surjection from $\Theta_{\mu}$ to the direct summand $\Theta_{\mu}$ of $\Theta$, and $\Hom _A (Q _{\lambda_1}, \Theta)$ contains a surjection from $Q_{\lambda_1}$ to $\Theta_{\mu}$ if and only if $\lambda_1 = \mu$. Thus we get $\lambda_1 = \mu$. Repeating the above process, we have $\Hom_A (Q_{\lambda}, \Theta) \cong \Hom _A (\Theta_{\lambda}, \Theta)$ for every $\lambda \in \Lambda$.

Applying $\Hom _A (-, \Theta)$ to the surjection $M \rightarrow \bigoplus _{h(\lambda) = i} \Theta_{\lambda} ^{\oplus m_{\lambda}}$ we get
\begin{equation*}
\Hom _A (\bigoplus _{h(\lambda) = i} \Theta_{\lambda} ^{\oplus m_{\lambda}}, \Theta) \subseteq \Hom _A (M, \Theta).
\end{equation*}
Similarly, from the relative projective covering map $\bigoplus _{h(\lambda) = i} Q_{\lambda} ^{\oplus m_{\lambda}} \rightarrow M$ we have
\begin{equation*}
\Hom _A (M, \Theta) \subseteq \Hom _A (\bigoplus _{h(\lambda) = i} Q_{\lambda} ^{\oplus m_{\lambda}}, \Theta).
\end{equation*}
Comparing these two inclusions and using the second isomorphism, we deduce the first isomorphism.
\end{proof}

The reader may aware that the above lemma is an analogue to the following result in representation theory of graded algebras: if $A$ is a graded algebra and $M$ is a graded module generated in degree 0, then $\Hom _A (M, A_0) \cong \Hom _A (M_0, A_0)$.

\begin{lemma}
Suppose that $\Hom _A (Q, \Theta) \cong \Hom _A (\Theta, \Theta)$. If $M \in \mathcal{F} (\Theta)$ is generated in height $i$, then $\Ext _A^s (M, \Theta) \cong \Ext _A^{s-1} (\Omega _{\Theta}(M), \Theta)$ for all $s \geqslant 1$.
\end{lemma}

\begin{proof}
Let $m_{\lambda} = [M: \Theta_{\lambda}]$. Applying $\Hom _A (-, \Theta)$ to the exact sequence
\begin{equation*}
\xymatrix{ 0 \ar[r] & \Omega _{\Theta} (M) \ar[r] & \bigoplus _{h(\lambda) = i} Q_{\lambda} ^{\oplus m_{\lambda}} \ar[r] & M \ar[r] & 0}
\end{equation*}
we get a long exact sequence. In particular, for all $s \geqslant 2$, by observing the segment
\begin{align*}
& 0 = \Ext _A^{s-1} (\bigoplus _{h(\lambda) = i} Q ^{\oplus m_{\lambda}}, \Theta) \rightarrow \Ext _A^{s-1} (\Omega _{\Theta} (M), \Theta) \\
& \rightarrow \Ext _A^s (M, \Theta) \rightarrow \Ext _A^s (\bigoplus _{h(\lambda) = i} Q_{\lambda} ^{\oplus m_{\lambda}}, \Theta) = 0
\end{align*}
we conclude $\Ext _A^{s-1} (\Omega _{\Theta}(M), \Theta) \cong \Ext _A^s (M, \Theta)$.

For $s=1$, we have
\begin{align*}
& 0 \rightarrow \Hom _A (M, \Theta) \rightarrow \Hom _A (\bigoplus _{h(\lambda) = i} Q_{\lambda} ^{\oplus m_{\lambda}}, \Theta)\\
& \rightarrow \Hom _A (\Omega _{\Theta} (M), \Theta) \rightarrow \Ext _A^1 (M, \Theta) \rightarrow 0.
\end{align*}
By the previous lemma, the first inclusion is an isomorphism. Thus $\Ext _A^1 (M, \Theta) \cong \Hom _A (\Omega _{\Theta} (M), \Theta)$.
\end{proof}

The following proposition tells us that if all standard modules are linearly filtered, then the extension algebra is generated in degrees 0 and 1.

\begin{proposition}
Suppose that $\Hom _A (Q, \Theta) \cong \Hom _A (\Theta, \Theta)$, and $\Theta_{\lambda}$ are linearly filtered for all $\lambda \in \Lambda$. If $M \in \mathcal{F} (\Theta)$ is linearly filtered, then
\begin{equation*}
\Ext _A^{i+1} (M, \Theta) = \Ext _A^1 (\Theta, \Theta) \cdot \Ext _A^i (M, \Theta)
\end{equation*}
for all $i \geqslant 0$, i.e., $\Ext _A^{\ast} (M, \Theta)$ as a graded $\Gamma = \Ext _A^{\ast} (\Theta, \Theta)$-module is generated in degree 0.
\end{proposition}

\begin{proof}
The proof is similar to that of Theorem 2.1.16. For the sake of completeness we include it here.

Suppose that $M$ is generated in height $d$ and linearly filtered. By Lemma 5.2.9,
\begin{equation*}
\Ext _A^{i+1} (M, \Theta) \cong \Ext _A^i (\Omega _{\Theta} (M), \Theta).
\end{equation*}
But $\Omega _{\Theta}$ is generated in height $d+1$ and linearly filtered. Thus by induction
\begin{equation*}
\Ext _A^{i+1} (M, \Theta) \cong \Ext _A^1 (\Omega^i _{\Theta} (M), \Theta), \quad \Ext _A^i (M, \Theta) \cong \Hom _A (\Omega^i _{\Theta} (M), \Theta).
\end{equation*}
Therefore, it suffices to show
\begin{equation*}
\Ext _A^1 (M, \Theta) = \Ext _A^1 (\Theta, \Theta) \cdot \Hom _A (M, \Theta)
\end{equation*}
since we can replace $M$ by $\Omega _{\Theta} ^i (M)$, which is linearly filtered as well.\

Let $m_{\lambda} = [M: \Theta_{\lambda}]$ and define $Q^0 = \bigoplus _{h(\lambda) = d} Q_{\lambda} ^{\oplus m_{\lambda}}$, $M_0 = \bigoplus _{h(\lambda) = d} \Theta_{\lambda} ^{\oplus m_{\lambda}}$. We have the following commutative diagram:
\begin{equation*}
\xymatrix {0 \ar[r] & \Omega _{\Theta} (M) \ar@{=}[d] \ar[r] & Q^0[1] \ar[r] \ar[d] & M[1] \ar[d] \ar[r] & 0\\
0 \ar[r] & \Omega _{\Theta} (M) \ar[r] & Q^0 \ar[r] \ar[d] & M \ar[d] \ar[r] & 0\\
 & & M_0 \ar@{=}[r] & M_0}
\end{equation*}
where $Q^0[1] = \Omega _{\Theta} (M_0)$, see Proposition 5.2.1.

Observe that all terms in the top sequence are generated in height $d+1$ and linearly filtered. For every $\lambda \in \Lambda$ with $h(\lambda) = d+1$, we have
\begin{equation*}
[\Omega _{\Theta} (M): \Theta_{\lambda}] + [M[1]: \Theta_{\lambda}] = [Q^0[1]: \Theta_{\lambda}].
\end{equation*}
Let $r_{\lambda}, s_{\lambda}$ and $t_{\lambda}$ be the corresponding numbers in the last equality. Then we get a split short exact sequence
\begin{equation*}
\xymatrix{ 0 \ar[r] & \bigoplus _{h(\lambda) = d+1} \Theta_{\lambda} ^{\oplus r_{\lambda}} \ar[r] & \bigoplus _{h(\lambda) = d+1} \Theta_{\lambda} ^{\oplus t_{\lambda}} \ar[r] & \bigoplus _{h(\lambda) = d+1} \Theta_{\lambda} ^{\oplus s_{\lambda}} \ar[r] & 0}.
\end{equation*}
Applying $\Hom _A (-, \Theta)$ to this sequence and using Lemma 5.2.8, we obtain the exact sequence
\begin{equation*}
0 \rightarrow \Hom _A (M[1], \Theta) \rightarrow \Hom _A (Q^0[1], \Theta) \rightarrow \Hom _A (\Omega _{\Theta} (M), \Theta) \rightarrow 0.
\end{equation*}
Therefore, each map $\Omega _{\Theta} (M) \rightarrow \Theta$ can extend to a map $Q^0[1] \rightarrow \Theta$.

To prove $\Ext_A^1 (M, \Theta) = \Ext _A^1 (\Theta, \Theta) \cdot \Hom _A (M, \Theta)$, by Lemma 5.2.9 we first identify $\Ext _A^1 (M, \Theta)$ with $\Hom _A (\Omega _{\Theta} (M), \Theta)$. Take an element $x \in \Ext _A^1 (M, \Theta)$ and let $g: \Omega _{\Theta} (M) \rightarrow \Theta$ be the corresponding homomorphism. As we just showed, $g$ can extend to $Q^0[1]$, and hence there is a homomorphism $\tilde{g}: Q^0[1] \rightarrow \Theta$ such that $g = \tilde{g} \iota$, where $\iota$ is the inclusion.
\begin{align*}
\xymatrix { \Omega _{\Theta} (M) \ar[r]^{\iota} \ar[d]^g & Q^0[1] \ar[dl]^{\tilde{g}}\\
\Theta.}
\end{align*}

We have the following commutative diagram:
\begin{align*}
\xymatrix{ 0 \ar[r] & \Omega _{\Theta} (M) \ar[d]^{\iota} \ar[r] & Q^0 \ar@{=}[d] \ar[r] & M \ar[d]^p \ar[r] & 0\\
0 \ar[r] & Q^0[1] \ar[r] & Q^0 \ar[r] & M_0 \ar[r] & 0,}
\end{align*}
where $p$ is the projection of $M$ onto $M_0$. The map $\tilde{g}: Q^0[1] \rightarrow \Theta$ gives a push-out of the bottom sequence:
\begin{align*}
\xymatrix{ 0 \ar[r] & \Omega _{\Theta} (M) \ar[d]^{\iota} \ar[r] & Q^0 \ar@{=}[d] \ar[r] & M \ar[d]^p \ar[r] & 0\\
0 \ar[r] & Q^0[1] \ar[r] \ar[d] ^{\tilde{g}} & Q^0 \ar[d] \ar[r] & M_0 \ar[r] \ar@{=}[d] & 0\\
0 \ar[r] & \Theta \ar[r] & E \ar[r] & M_0 \ar[r] & 0.}
\end{align*}

Since $M_0 \cong \bigoplus _{h(\lambda) = d} \Theta_{\lambda} ^{\oplus m_{\lambda}}$, the bottom sequence represents some
\begin{equation*}
y \in \Ext _A^1 (\Theta ^{\oplus m}, \Theta) = \bigoplus _{i=1} ^m \Ext _A^1 (\Theta, \Theta)
\end{equation*}
where $m = \sum _{h(\lambda) = d} m_{\lambda}$. Therefore, we can write $y = y_1 + \ldots + y_m$ where $y_i \in \Ext _A^1 (\Theta, \Theta)$ is represented by the sequence
\begin{equation*}
\xymatrix{ 0 \ar[r] & \Theta \ar[r] & E_i \ar[r] & \Theta \ar[r] & 0}.
\end{equation*}
Composed with the inclusions $\epsilon_{\lambda}: \Theta_{\lambda} \rightarrow \Theta$, we get the map $(p_1, \ldots, p_m)$ where each component $p_i$ is defined in an obvious way. Consider the pull-backs:
\begin{equation*}
\xymatrix {0 \ar[r] & \Theta \ar[r] \ar@{=}[d] & F_i \ar[r] \ar[d] & M \ar[r] \ar[d]^{p_i} & 0 \\
0 \ar[r] & \Theta \ar[r] & E_i \ar[r] & \Theta \ar[r] & 0.}
\end{equation*}
Denote by $x_i$ the top sequence. Then
\begin{equation*}
x = \sum_{i=1}^m x_i = \sum_{i=1}^m y_i \cdot p_i \in \Ext _A^1 (\Theta, \Theta) \cdot \Hom _A (M, \Theta),
\end{equation*}
so $\Ext _A^1 (M, \Theta) \subseteq \Ext _A^1 (\Theta, \Theta) \cdot \Hom _A (M, \Theta)$. The other inclusion is obvious.
\end{proof}

Now we can prove the main result of this chapter.

\begin{theorem}
Let $(\underline {\Theta}, \underline{Q})$ be an EPSS indexed by a finite poset $(\Lambda, \leqslant)$ such that $\Ext_A ^i (Q, \Theta) =0$ for all $i \geqslant 1$ and $\Hom _A (Q, \Theta) \cong \Hom _A (\Theta, \Theta)$. Suppose that all $\Theta_{\lambda}$ are linearly filtered for $\lambda \in \Lambda$. If $M \in \mathcal{F} (\Theta)$ is linearly filtered, then the graded module $\Ext _A^{\ast} (M, \Theta)$ has a linear projective resolution. In particular, $\Gamma = \Ext _A^{\ast} (\Theta, \Theta)$ is a generalized Koszul algebra.
\end{theorem}

\begin{proof}
The proof is similar to that of Theorem 2.3.1. For the completeness we include it here.

Suppose that $M$ is generated in height $d$. Define $m_{\lambda} = [M: \Theta_{\lambda}]$ for $\lambda \in \Lambda$, $Q^0 = \bigoplus _{h(\lambda) = d} Q_{\lambda} ^{\oplus m_{\lambda}}$, and $M_0 = \bigoplus _{h(\lambda) = d} \Theta_{\lambda} ^{\oplus m_{\lambda}}$. As in the proof of the previous lemma, we have the following short exact sequence of linearly filtered modules generated in height $d+1$:
\begin{align*}
\xymatrix{ 0 \ar[r] & \Omega _{\Theta} (M) \ar[r] & \Omega _{\Theta} (M_0) \ar[r] & M[1] \ar[r] & 0}
\end{align*}
where $\Omega _{\Theta} (M_0) = Q^0[1]$. This sequence induces exact sequences recursively (see the proof of Proposition 5.2.6):
\begin{align*}
\xymatrix{ 0 \ar[r] & \Omega^i _{\Theta} (M) \ar[r] & \Omega^i _{\Theta} (M_0) \ar[r] & \Omega^{i-1} _{\Theta} (M[1]) \ar[r] & 0,}
\end{align*}
where all modules are linearly filtered and generated in height $d+i$. Again as in the proof of the previous lemma, we get an exact sequence
\begin{align*}
0 \rightarrow \Hom _A (\Omega _{\Theta} ^{i-1} (M[1]), \Theta) \rightarrow \Hom _A (\Omega _{\Theta} ^i (M_0), \Theta) \rightarrow \Hom _A (\Omega _{\Theta} ^i (M), \Theta) \rightarrow 0.
\end{align*}
According to Lemma 5.2.9, the above sequence is isomorphic to:
\begin{align*}
0 \rightarrow \Ext ^{i-1} _A (M[1], \Theta) \rightarrow \Ext^i _A (M_0, \Theta) \rightarrow \Ext^i _A (M, \Theta) \rightarrow 0.
\end{align*}
Now let the index $i$ vary and put these sequences together. We have:
\begin{align*}
\xymatrix{ 0 \ar[r] & E(M[1]) \langle 1 \rangle \ar[r] & E(M_0) \ar[r]^p & E(M) \ar[r] & 0,}
\end{align*}
where $E = \Ext _A^{\ast} (-, \Theta)$ and $\langle - \rangle$ is the degree shift functor of graded modules. That is, for a graded module $T = \bigoplus _{i \geqslant 0} T_i$, $T \langle 1 \rangle _i$ is defined to be $T_{i-1}$.

Since $M_0 \in \text{add} (\Theta)$, $E(M_0)$ is a projective $\Gamma$-module. It is generated in degree 0 by the previous lemma. Similarly, $E(M[1])$ is generated in degree 0, so $E(M[1]) \langle 1 \rangle$ is generated in degree 1. Therefore, the map $p$ is a graded projective covering map. Consequently, $\Omega(E(M)) \cong E(M[1]) \langle 1 \rangle$ is generated in degree 1.

Replacing $M$ by $M[1]$ (since it is also linearly filtered), we have
\begin{equation*}
\Omega^2 (E(M)) \cong \Omega (E(M[1]) \langle 1 \rangle) \cong \Omega( E(M[1]) \langle 1 \rangle \cong E(M[2]) \langle 2 \rangle,
\end{equation*}
which is generated in degree 2. By recursion, $\Omega^i (E(M)) \cong E(M [i]) \langle i \rangle$ is generated in degree $i$ for all $i \geqslant 0$. Thus $E(M)$ is a linear $\Gamma$-module.

In particular let $M = Q_{\lambda}$ for a certain $\lambda \in \Lambda$. We get
\begin{equation*}
E(Q_{\lambda}) = \Ext _A^{\ast} (Q_{\lambda}, \Theta) = \Hom _A (Q_{\lambda}, \Theta)
\end{equation*}
is a linear $\Gamma$-module. Therefore,
\begin{align*}
\bigoplus _{\lambda \in \Lambda} E (Q_{\lambda}, \Theta) & = \bigoplus _{\lambda \in \Lambda} \Hom _A (Q_{\lambda}, \Theta) \cong \Hom _A (\bigoplus _{\lambda \in \Lambda} Q_{\lambda}, \Theta)\\
& = \Hom _A (Q, \Theta) \cong \Hom _A (\Theta, \Theta) = \Gamma_0
\end{align*}
is a linear $\Gamma$-module. So $\Gamma$ is a generalized Koszul algebra.
\end{proof}

\begin{remark}
To get the above result we made some assumptions on the EPSS $(\underline {\Theta}, \underline{Q})$. Firstly, each $\Theta_{\lambda}$ has a simple top $S_{\lambda}$ and $S_{\lambda} \ncong S_{\mu}$ for $\lambda \neq \mu$; secondly, $\Ext _A^s (Q, \Theta) = 0$ for every $s \geqslant 1$. These two conditions always hold for standardly stratified basic algebras. We also suppose that $\Hom _A (\Theta, \Theta) \cong \Hom _A (Q, \Theta)$. This may not be true even if $A$ is a quasi-hereditary algebra.
\end{remark}

Although $\Gamma$ is proved to be a generalized Koszul algebra, in general it does not have the Koszul duality. Consider the following example:

\begin{example}
Let $A$ be the path algebra of the following quiver with relation $\alpha \cdot \beta = 0$. Put an order $x < y < z$.
\begin{equation*}
\xymatrix{ x \ar@/^/ [rr]^{\alpha} & & y \ar@ /^/ [ll]^{\beta} \ar[r] ^{\gamma} & z.}
\end{equation*}
The projective modules and standard modules of $A$ are described as follows:
\begin{equation*}
P_x = \begin{matrix}  & x & \\ & y & \\ x & & z \end{matrix} \qquad P_y = \begin{matrix} & y & \\ x & & z \end{matrix} \qquad P_z = z
\end{equation*}
\begin{equation*}
\Delta_x = x \qquad \Delta_y = \begin{matrix} y \\ x \end{matrix} \qquad \Delta_z = z \cong P_z.
\end{equation*}
This algebra is quasi-hereditary. Moreover, $\Hom _A (\Delta, \Delta) \cong \Hom _A (A, \Delta)$, and all standard modules are linearly filtered. Therefore, $\Gamma = \Ext _A^{\ast} (\Delta, \Delta)$ is a generalized Koszul algebra by the previous theorem.\

We explicitly compute the extension algebra $\Gamma$. It is the path algebra of the following quiver with relation $\gamma \cdot \alpha = 0$.
\begin{equation*}
\xymatrix{ x \ar@/^/ [rr]^{\alpha} \ar@ /_/ [rr]_{\beta} & & y \ar[r] ^{\gamma} & z.}
\end{equation*}
\begin{equation*}
_{\Gamma} P_x = \begin{matrix}  & x_0 & \\ y_0 & & y_1 \\ z_1 & & \end{matrix} \qquad _{\Gamma} P_y = \begin{matrix} y_0 \\ z_1 \end{matrix} \qquad _{\Gamma} P_z = z_0
\end{equation*}
and
\begin{equation*}
_{\Gamma} \Delta_x = x_0 \qquad _{\Gamma} \Delta_y = y_0 \qquad _{\Gamma} \Delta_z = z_0 \qquad \Gamma_0 = \begin{matrix} x_0 \\ y_0 \end{matrix} \oplus y_0 \oplus z_0 \ncong _{\Gamma} \Delta.
\end{equation*}
Here we use indices to mark the degrees of simple composition factors. As asserted by the theorem, $\Gamma_0$ has a linear projective resolution. But $_{\Gamma} \Delta$ is not a linear $\Gamma$-module (we remind the reader that the two simple modules $y$ appearing in $_{\Gamma} P_x$ lie in different degrees!).

By computation, we get the extension algebra $\Gamma' = \Ext _{\Gamma} ^{\ast} (\Gamma_0, \Gamma_0)$, which is the path algebra of the following quiver with relation $\beta \cdot \alpha =0$.
\begin{equation*}
\xymatrix {x \ar[r] ^{\alpha} & y \ar[r] ^{\beta} & z.}
\end{equation*}
Since $\Gamma'$ is a Koszul algebra in the classical sense, the Koszul duality holds in $\Gamma'$. It is obvious that the Koszul dual algebra of $\Gamma'$ is not isomorphic to $\Gamma$. Therefore, as we claimed, the Koszul duality does not hold in $\Gamma$.
\end{example}

Let $\mathfrak{r} = \rad \Gamma_0$ and $\bar{\Gamma} = \Gamma / \Gamma \mathfrak{r} \Gamma$. Then we have an immediate corollary of the previous theorem.

\begin{corollary}
Suppose that $\Delta \cong \Gamma_0$ as a $\Gamma_0$-module and $\Delta_{\lambda}$ is linearly filtered for each $\lambda \in \Lambda$. Then the quotient algebra $\bar{\Gamma}$ is a classical Koszul algebra.
\end{corollary}

\begin{proof}
We only need to show that $\Gamma$ is a projective $\Gamma_0$-module, and the conclusion follows from the correspondence between the generalized Koszul theory and the classical theory (see Theorem 2.4.8). With the conditions in Theorem 5.2.11, $\Ext _A^s (\Theta_{\lambda}, \Theta) \cong \Hom _A (\Omega ^s _{\Theta _{\lambda}} (\Theta), \Theta)$ for all $s \geqslant 0$ and $\lambda \in \Lambda$ by Lemma 5.2.9. Notice that $\Omega ^s _{\Theta} (\Theta _{\lambda})$ is linearly filtered. Suppose that min$(\Omega ^s _{\Theta} (\Theta _{\lambda}) ) = d$ and $m_{\mu} = [\Omega ^s _{\Theta} (\Theta _{\lambda}): \Theta_{\mu}]$. Then
\begin{equation}
\Ext _A^s (\Theta _{\lambda}, \Theta) \cong \Hom _A (\Omega ^s _{\Theta} (\Theta _{\lambda}), \Theta) \cong \Hom _A ( \bigoplus _{h(\mu) = d} \Theta_{\mu} ^{\oplus m_{\mu}}, \Theta),
\end{equation}
which is a projective $\Gamma_0 = \End_A (\Theta)$-module.
\end{proof}

\begin{example}
Let $A$ be the path algebra of the following quiver with relations $\delta^2 = \delta \alpha = \beta \delta = \beta \alpha = \gamma \beta =0$. Let $x > z > y$.
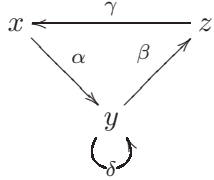
\begin{equation*}
\xymatrix{ x \ar[dr] ^{\alpha} & & z \ar[ll] _{\gamma}\\
& y \ar[ur] ^{\beta} \ar@(dl,dr)[]|{\delta} }
\end{equation*}
Indecomposable projective modules and standard modules of $A$ are described below:
\begin{align*}
& P_x = \begin{matrix} x \\ y \end{matrix} \quad P_y = \begin{matrix} & y & \\ y & & z \end{matrix} \quad P_z = \begin{matrix} z \\ x \\ y \end{matrix} \\
& \Delta_x = P_x = \begin{matrix} x \\ y \end{matrix} \quad \Delta_y = \begin{matrix} y \\ y \end{matrix} \quad \Delta_z = z.
\end{align*}
Clearly, $A$ is standardly stratified. Moreover, all standard modules have projective dimension 1 and are linearly filtered. By direct computation we check that $\Delta \cong \End _A (\Delta)$ as $\Gamma_0 = \End _A (\Delta)$-modules.

Now we compute the extension algebra $\Gamma$: $\Gamma_s = 0$ for $s \geqslant 2$; $\Ext _A^1 (\Delta_x, \Delta) = 0$; $\Ext _A^1 (\Delta_y, \Delta) \cong \End_A (\Delta_z)$; and $\Ext _A^1 (\Delta_z, \Delta) \cong \End_A (\Delta_x)$. Therefore, we find $\Gamma$ is the path algebra of the following quiver with relations $\delta^2 = \beta \delta = \alpha \delta = \gamma \beta =0$.
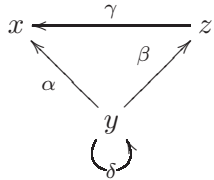
\begin{equation*}
\xymatrix{ x & & z \ar[ll] _{\gamma}\\
& y \ar[ur] ^{\beta} \ar[ul] ^{\alpha} \ar@(dl,dr)[]|{\delta} }
\end{equation*}

We remind the reader that $\alpha$ is in the degree 0 part of $\Gamma$. Indeed, $\Gamma_0 = \langle 1_x, 1_y, 1_z, \delta, \alpha \rangle$ and $\Gamma_1 = \langle \beta, \gamma \rangle$. In this case $\Gamma_0$ does not satisfy the splitting condition (S) since we can find the following non-splitting exact sequence:
\begin{equation*}
0 \rightarrow  x \rightarrow \begin{matrix} & y & \\ y & & x \end{matrix} \rightarrow \begin{matrix} y \\ y \end{matrix} \rightarrow 0.
\end{equation*}

Since $\mathfrak{r} = \rad \Gamma_0 = \langle \delta, \alpha \rangle = \Gamma \mathfrak{r} \Gamma$, the quotient algebra $\Gamma$ is the path algebra of the following quiver with relation $\gamma \beta = 0$, which is clearly a classical Koszul algebra.
\begin{equation*}
\xymatrix {x & z \ar[l]_{\gamma} & y \ar[l] _{\beta}}.
\end{equation*}
\end{example}

Let us return to the question of whether $\Gamma = \Ext _A ^{\ast} (\Theta, \Theta)$ is standardly stratified with respect to $\leqslant$. According to Theorem 5.1.3, this happens if and only if for each pair $\Theta_{\lambda}, \Theta_{\mu}$ and $s \geqslant 0$, $\Ext _A^s (\Theta_{\lambda}, \Theta_{\mu})$ is a projective $\End _A (\Theta_{\mu})$-module. Putting direct summands together, we conclude that $\Gamma$ is standardly stratified with respect to $\leqslant$ if and only if $\Ext _A^s (\Theta, \Theta)$ is a projective $\bigoplus _{\lambda \in \Lambda} \End _A (\Theta_{\lambda})$-module.

With this observation, we have:

\begin{corollary}
Let $(\underline {\Theta}, \underline {Q})$ be an EPSS indexed by a finite poset $(\Lambda, \leqslant)$. Suppose that all $\Theta_{\lambda}$ are linearly filtered for $\lambda \in \Lambda$, and $\Hom _A (Q, \Theta) \cong \Hom _A (\Theta, \Theta)$. Then $\Gamma = \Ext _A^{\ast} (\Theta, \Theta)$ is standardly stratified for $\leqslant$ if and only if $\End _A (\Theta)$ is a projective $\bigoplus _{\lambda \in \Lambda} \End _A (\Theta_{\lambda})$-module.
\end{corollary}

\begin{proof}
If $\Gamma$ is standardly stratified for $\leqslant$, then in particular $\Gamma_0 = \End _A (\Theta)$ is a projective $\bigoplus _{\lambda \in \Lambda} \End _A (\Theta_{\lambda})$-module by Theorem 5.1.3. Conversely, if $\Gamma_0 = \End _A (\Theta)$ is a projective $\bigoplus _{\lambda \in \Lambda} \End _A (\Theta_{\lambda})$-module, then by the isomorphism in (5.2.2) we know that $\Ext _A^s (\Theta, \Theta) = \bigoplus _{\lambda \in \Lambda} \Ext _A^s (\Theta_{\lambda}, \Theta)$ is a projective $\Gamma_0$-module for all $s \geqslant 0$, so it is a projective $\bigoplus _{\lambda \in \Lambda} \End _A (\Theta_{\lambda})$-module as well. Again by Theorem 5.1.3, $\Gamma$ is standardly stratified with respect to $\leqslant$.
\end{proof}

If $A$ is quasi-hereditary with respect to $\leqslant$ such that all standard module are linearly filtered, then $\Gamma = \Ext _A^{\ast} (\Delta, \Delta)$ is again quasi-hereditary for this partial order by Corollary 5.1.4, and $\Gamma_0$ is Koszul by the previous theorem. Let $_{\Gamma} \Delta$ be the direct sum of all standard modules of $\Gamma$ with respect to $\leqslant$. The reader may wonder whether $_{\Gamma} \Delta$ is Koszul as well. The following proposition gives a partial answer to this question.

\begin{proposition}
With the above notation, if $_{\Gamma} \Delta$ is Koszul, then $\Gamma_0 \cong  _{\Gamma}\Delta$, or equivalently $\Hom _A (\Delta_{\lambda}, \Delta_{\mu}) \neq 0$ only if $\lambda = \mu$, $\lambda, \mu \in \Lambda$. If furthermore $\Hom _A (A, \Delta) \cong \End _A (\Delta)$, then $\Delta \cong A / \rad A$.
\end{proposition}

\begin{proof}
We have proved that the $k$-linear category associated to $\Gamma$ is directed with respect to $\leqslant$. By Theorem 5.1.3, standard modules of $\Gamma$ for $\leqslant$ are exactly indecomposable summands of $\bigoplus _ {\lambda \in \Lambda} \End _A (\Delta_{\lambda})$, i.e., $_{\Gamma} \Delta \cong \bigoplus _ {\lambda \in \Lambda} \End _A (\Delta_{\lambda}) \cong \bigoplus _{\lambda \in \Lambda} k_{\lambda}$. Clearly, $_{\Gamma} \Delta \subseteq \Gamma_0 = \End _A (\Delta)$. If $_{\Gamma} \Delta$ is Koszul, then by Corollary 2.1.5, it is a projective $\Gamma_0$-module. Consequently, every summand $k_{\lambda}$ is a projective $\Gamma_0$-module. Since both $_{\Gamma} \Delta$ and $\Gamma_0$ have exactly $| \Lambda |$ pairwise non-isomorphic indecomposable summands, we deduce $_{\Gamma} \Delta \cong \Gamma_0 \cong \bigoplus _{\lambda \in \Lambda} k_{\lambda}$, or equivalently $\Hom _A (\Delta_{\lambda}, \Delta_{\mu}) = 0$ if $\lambda \neq \mu$.

If furthermore $\Hom _A (A, \Delta) \cong \End _A (\Delta)$, then
\begin{equation*}
\Delta \cong \Hom _A (A, \Delta) \cong \End _A (\Delta) \cong \bigoplus _{\lambda \in \Lambda} k_{\lambda} \cong A / \rad A.
\end{equation*}
\end{proof} 
\chapter{Algebras stratified for all linear orders}
\label{Algebras stratified for all linear orders}

Dlab and Ringel showed in \cite{Dlab2} that a finite-dimensional algebra $A$ is quasi-hereditary for all linear orders if and only if $A$ is a hereditary algebra. In this chapter we characterize and classify algebras standardly stratified for all linear orders. We will find that these algebras includes hereditary algebras as special cases, and have many properties similar to hereditary algebras.

In the first section we describe several characterizations of algebras stratified for all linear orders, and then classify these algebras in the second section. We also consider the problem of whether $\mathcal{F} (\Delta)$ is closed under cokernels of monomorphisms. The main content of this chapter comes from \cite{Li5}.

\section{Several characterizations}
Let $A$ be a basic finite-dimensional $k$-algebra with a chosen set of orthogonal primitive idempotents $\{ e_{\lambda} \} _{\lambda \in \Lambda}$ indexed by a set $\Lambda$ such that $\sum _{\lambda \in \Lambda} e_{\lambda} = 1$. Let $P_{\lambda} = A e_{\lambda}$ and $S_{\lambda} = P_{\lambda} / \rad P_{\lambda}$. For definitions of \textit{standardly stratified algebras, properly stratified algebras, standard modules, costandard modules, proper standard modules, proper costandard modules,} please refer to Section 3.1. Since by Proposition 6.1.1 $A$ is standardly stratified for all preorders on $\Lambda$ if and only if it is a direct sum of local algebras, in this chapter we only deal with linear orders on $\Lambda$ instead of general preorders as in \cite{Frisk2, Frisk3, Webb3}.

\begin{proposition}
The algebra $A$ is standardly stratified for all preorders if and only if $A$ is a direct sum of local algebras.
\end{proposition}

\begin{proof}
If $A$ is a direct sum of local algebras, it is standardly stratified for all preorders such that all standard modules coincide with indecomposable projective modules. Conversely, suppose that $A$ is standardly stratified for all preorders. To show that $A$ is a direct sum of local algebras, it is sufficient to show that $\Hom _A (P_{\lambda}, P_{\mu}) \cong e_{\lambda} A e_{\mu} = 0$ for all pairs of distinct elements $\lambda \neq \mu \in \Lambda$.

Consider the preorder on $\Lambda$ such that every two different elements cannot be compared. By the second condition in the definition of standardly stratified algebras, $\Delta_{\lambda} \cong P_{\lambda}$ for all $\lambda \in \Lambda$. By the first condition, $P_{\lambda}$ only has composition factors $S_{\lambda}$. Therefore, $\Hom _A (P_{\lambda}, P_{\mu}) \cong e_{\lambda} A e_{\mu} = 0$ for all $\lambda \neq \mu \in \Lambda$.
\end{proof}

The following statement is an immediate corollary of Dlab's theorem (\cite{Frisk3, Webb3}).

\begin{proposition}
The algebra $A$ is properly stratified for a linear order $\preccurlyeq$ on $\Lambda$ if and only if both $A$ and $A^{\textnormal{op}}$ are standardly stratified with respect to $\preccurlyeq$, in other words, $A$ is both left and right standardly stratified.
\end{proposition}

\begin{proof}
The algebra $A$ is properly stratified if and only if $\mathcal{F} (\Delta) \cap \mathcal{F} (\overline {\Delta})$ contains all left projective $A$-modules. By duality, this is true if and only if $\mathcal{F} (\Delta)$ contains all left projective $A$-modules and $\mathcal{F} (\overline {\nabla} _A)$ contains all right injective $A$-modules, where $\overline {\nabla} _A$ is a right $A$-module. By Theorem 3.4 in \cite{Webb3}, we conclude that $A$ is properly stratified if and only if $\mathcal{F} (\Delta)$ contains all left projective $A$-module and $\mathcal{F} (\Delta_A)$ contains all right projective $A$-modules, here $\Delta_A$ is the direct sum of all right standard modules. That is, $A$ is properly stratified if and only if it is both left and right standardly stratified.
\end{proof}

Recall the associated category $\mathcal{A}$ of $A$ is a \textit{directed} category if there is a partial order $\leqslant$ on $\Lambda$ such that $\mathcal{A} (e_{\lambda}, e_{\mu}) = e_{\mu} A e_{\lambda} \neq 0$ implies $\lambda \leqslant \mu$.

\begin{proposition}
If $A$ is standardly stratified for all linear orders on $\Lambda$, then the associated category $\mathcal{A}$ is a directed category.
\end{proposition}

\begin{proof}
Suppose that the conclusion is not true. Then there is an oriented cycle $e_0 \rightarrow e_1 \rightarrow \ldots \rightarrow e_n = e_0$ with $n > 1$ such that $\mathcal{A} (e_i, e_{i+1}) = e_{i+1} A e_i \neq 0$ for all $0 \leqslant i \leqslant n-1$. Take some $e_s$ such that $\dim _k P_s \geqslant \dim _k P_i$ for all $0 \leqslant i \leqslant n-1$, where $P_s = A e_s$ coincides with the vector space formed by all morphisms starting from $e_s$. Then for an arbitrary linear order $\preccurlyeq$ with respect to which $e_s$ is maximal, we claim that $A$ is not standardly stratified.

Indeed, if $A$ is standardly stratified with respect to $\preccurlyeq$, then tr$ _{P_s} (P_i) \cong P_s^{m_i}$ for some $m_i \geqslant 0$. Consider $P_{s-1}$ (if $s=0$ we consider $P_{n-1}$). Since $\mathcal{A} (e_{s-1}, e_s) = e_s A e_{s-1} \neq 0$, tr$ _{P_s} (P_{s-1}) \neq 0$, so $m_{s-1} \geqslant 1$. But on the other hand, since tr$ _{P_s} (P_{s-1}) \subseteq \rad P_{s-1}$, and $\dim _k P_s \geqslant \dim _k P_{s-1}$, we should have $m_{s-1} = 0$. This is absurd. Therefore, there is no oriented cycle in $\mathcal{A}$, and the conclusion follows.
\end{proof}

The following proposition motivates us to study the problem of whether $\mathcal{F} (\Delta)$ is closed under cokernels of monomorphisms.

\begin{proposition}
Suppose that the associated category $\mathcal{A}$ of $A$ is a directed category with respect to a linear order $\leqslant$. Then $\mathcal{F} (_{\leqslant} \Delta)$ is closed under cokernels of monomorphisms.
\end{proposition}

\begin{proof}
Take an arbitrary exact sequence with $L, M \in \mathcal{F} (_{\leqslant}  \Delta)$, we need to show $N \in \mathcal{F} (_{\leqslant} \Delta)$ as well:
\begin{equation*}
\xymatrix{ 0 \ar[r] & L \ar[r] & M \ar[r] & N \ar[r] & 0.}
\end{equation*}
By the structure of standard modules, it is clear that an $A$-module $K \in \mathcal{F} (_{\leqslant} \Delta)$ if and only if $e_{\lambda} K$ is a free $\mathcal{A} (e_{\lambda},e_{\lambda}) = e_{\lambda} A e_{\lambda}$-module for all $\lambda \in \Lambda$. For an arbitrary $\lambda \in \Lambda$, the above sequence induces an exact sequence of $e_{\lambda} A e_{\lambda}$-modules
\begin{equation*}
\xymatrix{ 0 \ar[r] & e_{\lambda} L \ar[r] & e_{\lambda} M \ar[r] & e_{\lambda} N \ar[r] & 0.}
\end{equation*}
Since $L, M \in \mathcal{F} (_{\leqslant} \Delta)$, we know that $e_{\lambda} L \cong (e_{\lambda} A e_{\lambda})^l$ and $e_{\lambda} M \cong (e_{\lambda} A e_{\lambda})^m$ for some $l, m \geqslant 0$.

Notice that $e_{\lambda} A e_{\lambda}$ is a local algebra. The regular module over $e_{\lambda} A e_{\lambda}$ is indecomposable, has a simple top and finite length. Therefore, the top of $e_{\lambda} L$ is embedded into the top of $e_{\lambda} M$ since otherwise the first map cannot be injective. Therefore, the top of $e_{\lambda} N$ is isomorphic to $S_{\lambda}^{m-l}$ where $S_{\lambda} \cong (e_{\lambda} A e_{\lambda}) / \rad (e_{\lambda} A e_{\lambda})$ and $e_{\lambda} N$ is the quotient module of $(e_{\lambda} A e_{\lambda}) ^{m-l}$. But by comparing dimensions we conclude that $e_{\lambda} N \cong (e_{\lambda} A e_{\lambda}) ^{m-l}$. Consequently, $e_{\lambda} N$ and hence $N$ are contained in $\mathcal{F} (_{\leqslant} \Delta)$ as well. This finishes the proof.
\end{proof}

Actually, This also proves that if $A_0$ of a positively graded algebra $A$ is a direct sum of local algebras, then $A_0$ has the splitting property (S).

In some sense the property of being standardly stratified for all linear orders is inherited by quotient algebras. Explicitly,

\begin{lemma}
If $A$ is standardly stratified for all linear orders, then for an arbitrary primitive idempotent $e_{\lambda}$, the quotient algebra $\bar{A} = A / A e_{\lambda} A$ is standardly stratified for all linear orders.
\end{lemma}

\begin{proof}
Let $\preccurlyeq$ be a linear order on $\Lambda \setminus \{\lambda \}$. Then we can extend it to a linear order $\tilde {\preccurlyeq}$ on $\Lambda$ by letting $\lambda$ be the unique maximal element. Since $A$ is standardly stratified for $\tilde {\preccurlyeq}$, we conclude that $\bar {A}$ is standardly stratified for $\preccurlyeq$ as well.
\end{proof}

Now suppose that the associated category $\mathcal{A}$ of $A$ is directed. Then we define
\begin{equation*}
J = \bigoplus _{ \lambda \neq \mu \in \Lambda} \mathcal{A} (e_{\lambda}, e_{\mu} ) = \bigoplus _{ \lambda \neq \mu \in \Lambda} e_{\mu} A e_{\lambda}, \qquad A_0 = \bigoplus _{ \lambda \in \Lambda} \mathcal{A} (e_{\lambda}, e_{\lambda}) = \bigoplus _{ \lambda \in \Lambda} e_{\lambda} A e_{\lambda}.
\end{equation*}
Clearly, $J$ is a two-sided ideal of $A$ with respect to this chosen set of orthogonal primitive idempotents, and $A_0 \cong A /J$ as $A$-modules.

\begin{proposition}
If $A$ is standardly stratified for all linear orders on $\Lambda$, then $J$ is a projective $A$-module.
\end{proposition}

\begin{proof}
Note that the associated category $\mathcal{A}$ is directed by Proposition 6.1.3, so $J$ is a well defined $A$-module. Let $\leqslant$ be a linear order on $\Lambda$ with respect to which $\mathcal{A}$ is directed, i.e., $\mathcal{A} (e_{\lambda}, e_{\mu}) = e_{\mu} A e_{\lambda} \neq 0$ implies $\lambda \leqslant \mu$. We prove this conclusion by induction on $| \Lambda |$, the size of $\Lambda$. It holds for $|\Lambda| = 1$ since $J = 0$. Now suppose that it holds for $|\Lambda| = s$ and consider $|\Lambda| = s+1$.

Let $\lambda$ be the minimal element in $\Lambda$ with respect to $\leqslant$. Therefore, $\mathcal{A} (e_{\mu}, e_{\lambda}) = e_{\lambda} A e_{\mu} = 0$ for all $\mu \neq \lambda \in \Lambda$. Let $\preccurlyeq$ be a linear order on $\Lambda$ such that $\lambda$ is the unique maximal element with respect to $\preccurlyeq$. Consider the quotient algebra $\bar{A} = A / A e_{\lambda} A$, which is standardly stratified for all linear orders on $\Lambda \setminus \{\lambda \}$ by the previous lemma. Thus the associated category $\bar {\mathcal{A}}$ of $\bar{A}$ is directed, and we can define $\bar{J}$ as well.

For $\mu \neq \lambda$, we have $\Hom_A (P_{\lambda}, P_{\mu}) \cong e_{\lambda} A e_{\mu} = 0$ since $\lambda$ is minimal with respect to $\leqslant$ and $\mathcal{A}$ is directed. Therefore, tr$ _{P_{\lambda}} (P _{\mu}) = 0$, and the corresponding indecomposable projective $\bar{A}$-module $\bar{P}_{\mu}$ is isomorphic to $P_{\mu}$. Let $J_{\mu} = J e_{\mu}$. Then we have:
\begin{align*}
J_{\mu} & = \bigoplus _{\nu \neq \mu \in \Lambda} e_{\nu} A e_{\mu} = \bigoplus _{\nu > \mu} e_{\nu} A e_{\mu} \cong \bigoplus _{\nu > \mu} \Hom _A (P_{\nu}, P_{\mu}) \\
& \cong \bigoplus _{\nu > \mu} \Hom _{\bar{A}} (\bar{P}_{\nu}, \bar{P}_{\mu}) \cong \bigoplus _{\nu > \mu} e_{\nu} \bar{A} e_{\mu} = \bar{J}_{\mu}
\end{align*}
By the induction hypothesis, $J_{\mu} \cong \bar{J}_{\mu}$ is a projective $\bar {A}$-module. By our choice of $\lambda$, it is actually a projective $A$-module since $e_{\lambda}$ acts on $J_{\mu}$ as 0. Therefore, $J_{\mu}$ is a projective $A$ module.

Since $J = \bigoplus _{\nu \in \Lambda} J_{\nu}$ and we have proved that all $J_{\nu}$ are projective for $\nu \neq \lambda$, it remains to prove that $J_{\lambda}$ is a projective module. To achieve this, we take the element $\mu \in \Lambda$ which is minimal in $\Lambda \setminus \{\lambda\}$ with respect to $\leqslant$. Therefore, $\mathcal{A} (e_{\nu}, e_{\mu}) = e_{\mu} A e_{\nu} \neq 0$ only if $\nu = \mu$ or $\nu = \lambda$.

Now define another linear order $\preccurlyeq '$ on $\Lambda$ such that $\mu$ is the unique maximal element with respect to $\preccurlyeq'$. Similarly, for all $\nu \neq \mu, \lambda$, we have $\bar{P}_{\nu}' \cong P_{\nu}$, where $\bar{A}' = A / A e_{\mu} A$ and $\bar{P}_{\nu}' = \bar{A}' e_{\nu}$. As before, the associated category of $\bar{A}'$ is directed so we can define $\bar{J}'$. Moreover, we have $M = \text{tr} _{P_{\mu}} (P_{\lambda}) \subseteq J_{\lambda}$. Thus we get the following commutative diagram:
\begin{equation*}
\xymatrix{ & & 0 \ar[d] & 0 \ar[d] & \\
0 \ar[r] & M \ar[r] \ar@{=}[d] & J_{\lambda} \ar[r] \ar[d] & \bar{J} _{\lambda}' \ar[r] \ar[d] & 0 \\
0 \ar[r] & M \ar[r] & P_{\lambda} \ar[r] \ar[d] & \bar{P} _{\lambda}' \ar[r] \ar[d] & 0 \\
 & & e_{\lambda} A e_{\lambda} \ar[d] \ar@{=}[r] & e_{\lambda} A e_{\lambda} \ar[d] \ar[r] & 0 \\
 & & 0 & 0 &
}
\end{equation*}
where $\bar{J}_{\lambda}' = \bar{J'} e_{\lambda}$. By the induction hypothesis on the quotient algebra $\bar {A'}$, $\bar{J}_{\lambda}'$ is a projective $\bar {A'}$-module. Clearly, each indecomposable summands of $\bar{J}_{\lambda}'$ is isomorphic to a certain $\bar{P}_{\nu}'$ with $\nu \neq \lambda, \mu$. But $\bar{P}_{\nu}' \cong P_{\nu}$.  Therefore, $\bar {J}_{\lambda}'$ is actually a projective $A$-module, and the top sequence in the above diagram splits, i.e., $J_{\lambda} \cong M \oplus \bar{J}_{\lambda}'$. But $M = \text{tr} _{P_{\mu}} (P_{\lambda})$ is also a projective $A$-module (which is actually isomorphic to a direct sum of copies of $P_{\mu}$) since $A$ is standardly stratified with respect to $\preccurlyeq'$ and $\mu$ is maximal with respect to this order. Thus $J_{\lambda}$ is a projective $A$-module as well. The proof is completed.
\end{proof}

\begin{proposition}
Suppose that the associated category $\mathcal{A}$ of $A$ is directed. If $J$ is a projective $A$-module then for an arbitrary pair $\lambda, \mu \in \Lambda$, tr$_{P_{\lambda}} (P_{\mu}) \cong P_{\lambda} ^{m_{\mu}}$ for some $m_{\mu} \geqslant 0$.
\end{proposition}

\begin{proof}
Let $\leqslant$ be a linear order on $\Lambda$ with respect to which $\mathcal{A}$ is directed. We can index all orthogonal primitive idempotents: $e_n > e_{n-1} > \ldots > e_1$. Let $P_i = A e_i$. Suppose that $e_{\lambda} = e_s$ and $e_{\mu} = e_t$. If $s < t$, tr$_{P_s} (P_t) = 0$ and the conclusion is trivially true. For $s = t$, the conclusion holds as well. So we assume $s > t$ and conduct induction on the difference $d = s -t$. Since it has been proved for $d =0$, we suppose that the conclusion holds for all $d \leqslant l$.

Now suppose $d = s - t = l+1$. Let $E_t = \mathcal{A} (e_t, e_t) \cong P_t /J_t \cong e_t A e_t$, which can be viewed as an $A$-module. We have the following exact sequence:
\begin{equation*}
\xymatrix{0 \ar[r] & J_t \ar[r] & P_t \ar[r] & E_t \ar[r] & 0}.
\end{equation*}
Since $J$ is projective, so is $J_t = Je_t$. Moreover, since $J_t = \bigoplus _{m \neq t} e_m A e_t$ and $e_m A e_t = 0$ if $m \ngtr t$, we deduce that $J_t$ has no summand isomorphic to $P_m$ with $m < t $. Therefore, $J_t \cong \bigoplus _{t+1 \leqslant i \leqslant n} P_i^{m_i}$, where $m_i$ is the multiplicity of $P_i$ in $J_t$.

The above sequence induces an exact sequence
\begin{equation*}
\xymatrix{0 \ar[r] & \bigoplus _{t+1 \leqslant i \leqslant n} \text{tr} _{P_s} (P_i)^{m_i} \ar[r] & \text{tr} _{P_s} (P_t) \ar[r] & \text{tr} _{P_s} (E_t) \ar[r] & 0}.
\end{equation*}
Clearly, $\text{tr} _{P_s} (E_t) = 0$, so
\begin{equation*}
\bigoplus _{t+1 \leqslant i \leqslant n} \text{tr} _{P_s} (P_i)^{m_i} \cong \text{tr} _{P_s} (P_t).
\end{equation*}
But for each $t+1 \leqslant i \leqslant n$, we get $s - i < s - t = l+1$. Thus by the induction hypothesis, each $\text{tr} _{P_s} (P_i)$ is a projective module isomorphic to a direct sum of $P_s$. Therefore, $\text{tr} _{P_s} (P_t)$ is a direct sum of $P_s$. The conclusion follows from induction.
\end{proof}

The condition that $\mathcal{A}$ is directed is required in the previous propositions since otherwise $J$ might not be well defined.

\begin{proposition}
If for each pair $\lambda, \mu \in \Lambda$, tr$ _{P_{\lambda}} (P_{\mu})$ is a projective module, then the associated category $\mathcal{A}$ is a directed category. Moreover, $A$ is standardly stratified for all linear orders.
\end{proposition}

\begin{proof}
The first statement comes from the proof of Proposition 6.1.3. Actually, if $\mathcal{A}$ is not directed, in that proof we find some $P_{\lambda}$ and $P_{\mu}$ such that tr$_{P_{\lambda}} (P_{\mu}) \neq 0$ is not projective.

To prove the second statement, we use induction on $|\Lambda|$. It is clearly true if $|\Lambda| = 1$. Suppose that the conclusion holds for $|\Lambda| \leqslant l$ and suppose that $|\Lambda| = l+1$. Let $\preccurlyeq$ be an arbitrary linear order on $\Lambda$ and take a maximal element $\lambda$ with respect to this linear order. Since tr$ _{P_{\lambda}} (P_{\mu})$ is a projective $A$-module for all $\mu \in \Lambda$ by the given condition, it is enough to show that the quotient algebra $\bar {A} =  A /A e_{\lambda} A$ has the same property. That is, tr$ _{\bar{P_{\mu}}} (\bar{P_{\nu}})$ is a projective $\bar {A}$-module for all $\mu, \nu \in \Lambda \setminus \{ \lambda \}$, where $\bar{P_{\mu}} = \bar {A} e_{\mu}$ and $\bar{P_{\nu}}$ is defined similarly. Then the conclusion will follow from induction hypothesis.

Let $\leqslant$ be a linear order on $\Lambda$ with respect to which $\mathcal{A}$ is directed. It restriction on $\Lambda \setminus \{ \lambda \}$ gives a linear order with respect to which $\bar{\mathcal{A}}$, the associated category of $\bar{A}$, is directed. Consider tr$ _{\bar{P_{\mu}}} (\bar{P_{\nu}})$. If $\mu \leqslant \nu$, this trace is 0 or $\bar{P_{\nu}}$. Thus we assume that $\mu > \nu$. Since tr$ _{P_{\mu}} (P_{\nu})$ is projective, it is isomorphic to a direct sum of of $P_{\mu}$, and we get the following exact sequence:
\begin{equation*}
\xymatrix{0 \ar[r] & \text{tr} _{P_{\mu}} (P_{\nu}) \cong P_{\mu}^m \ar[r] & P_{\nu} \ar[r] & M \ar[r] & 0}
\end{equation*}
with $e_{\mu} M = 0$.

Since tr$ _{P_{\lambda}} (P_{\mu})$ and tr$ _{P_{\lambda}} (P_{\nu})$ are also isomorphic to direct sums of $P_{\lambda}$, by considering the traces of $P_{\lambda}$ in the modules in the above sequence, we get a commutative diagram as follows:
\begin{equation*}
\xymatrix{ & 0 \ar[d] & 0 \ar[d] & 0 \ar[d] & \\
0 \ar[r] & P_{\lambda}^s \cong \text{tr} _{P_{\lambda}} (P_{\mu}^m) \ar[r] \ar[d] & P_{\lambda}^t \cong \text{tr} _{P_{\lambda}} (P_{\nu}) \ar[r] \ar[d] & P_{\lambda}^{t-s} \cong \text{tr} _{P_{\lambda}} (M) \ar[r] \ar[d] & 0 \\
0 \ar[r] & P_{\mu}^m \ar[r] \ar[d] & P_{\nu} \ar[r] \ar[d] & M \ar[r] \ar[d] & 0 \\
0 \ar[r] & \bar{P_{\mu}}^m \ar[r] \ar[d] & \bar{P_{\nu}} \ar[r] \ar[d] & \bar{M} \ar[r] \ar[d] & 0 \\
 & 0 & 0 & 0 &.}
\end{equation*}
Since $e_{\mu} M =0$, we get $e_{\mu} \bar{M} =0$ as well. Thus from the bottom row we conclude that tr$ _{\bar {P_{\mu}}} (\bar {P_{\nu}}) \cong \bar{P_{\mu}}^m$ is projective. This finishes the proof.
\end{proof}

Now we can give several characterizations for algebras standardly stratified for all linear orders.

\begin{theorem}
Let $A$ be a basic finite-dimensional algebra and $\mathcal{A}$ be its associated category. Then the following are equivalent:
\begin{enumerate}
\item $A$ is standardly stratified for all linear orders.
\item $\mathcal{A}$ is a directed category and $J = \bigoplus _{\lambda \neq \mu \in \Lambda} e_{\mu} A e_{\lambda}$ is a projective $A$-module.
\item The trace tr$ _{P_{\lambda}} (P_{\mu})$ is a projective module for $\lambda, \mu \in \Lambda$.
\item The projective dimension $\pd _A M \leqslant 1$ for all $M \in \mathcal{F} (_{\preccurlyeq} \Delta)$ and linear orders $\preccurlyeq$.
\end{enumerate}
\end{theorem}

\begin{proof}
$(1) \Rightarrow (2)$: by Propositions 6.1.3 and 6.1.6.

$(2) \Rightarrow (3)$: by Proposition 6.1.7.

$(3) \Rightarrow (1)$: by Proposition 6.1.8.

$(2) \Rightarrow (4)$: By the assumption, there is a linear order $\leqslant$ on $\Lambda$ such that $\mathcal{A}$ is directed with respect to it. By Proposition 6.1.4, we know that $_{\leqslant} \Delta \cong A_0 = \bigoplus _{\lambda \in \Lambda} e_{\lambda} A e_{\lambda}$. Thus the projective dimension of $_{\leqslant} \Delta$ is at most 1 since we have the exact sequence $0 \rightarrow J \rightarrow A \rightarrow A_0 \rightarrow 0$ and $J$ is projective. Therefore, every $M \in \mathcal{F} (_{\leqslant} \Delta)$ has projective dimension at most 1.

Note that $\mathcal{F} (_{\leqslant} \Delta)$ is closed under cokernels of monomorphisms by Proposition 6.1.4. Let $\preccurlyeq$ be an arbitrary linear order on $\Lambda$. By the second statement of Theorem 6.3.4 (which will be proved later), $\mathcal{F} (_{\preccurlyeq} \Delta) \subseteq \mathcal{F} (_{\leqslant} \Delta)$. Thus every $M \in \mathcal{F} (_{\preccurlyeq} \Delta) \subseteq \mathcal{F} (_{\leqslant} \Delta)$ has projective dimension at most 1.

$(4) \Rightarrow (3)$: Take an arbitrary pair $\lambda, \mu \in \Lambda$. We want to show that tr$ _{P_{\lambda}} (P_{\mu})$ is a projective $A$-module. Clearly, there exists a linear order $\preccurlyeq$ on $\Lambda$ such that $\lambda$ is the maximal element in $\Lambda$ and $\mu$ is the maximal element in $\Lambda \setminus \{ \lambda \}$. Therefore, by the definition, we have a short exact sequence
\begin{equation*}
\xymatrix{ 0 \ar[r] & \text{tr} _{P_{\lambda}} (P_{\mu}) \ar[r] & P_{\mu} \ar[r] & _{\preccurlyeq} \Delta_{\mu} \ar[r] & 0}.
\end{equation*}
Clearly, $_{\preccurlyeq} \Delta_{\mu} \in \mathcal{F} (_{\preccurlyeq} \Delta)$. Therefore, it has projective dimension at most 1, which means that $\text{tr} _{P_{\lambda}} (P_{\mu})$ is projective.
\end{proof}

This immediately gives us the following characterization of algebras properly stratified for all linear orders.

\begin{corollary}
Let $A, \mathcal{A}$ and $J$ be as in the previous theorem. Then $A$ is properly stratified for all linear orders if and only if $\mathcal{A}$ is a directed category and $J$ is a left and right projective $A$-module.
\end{corollary}

\begin{proof}
By Proposition 6.1.2 and Theorem 6.1.9.
\end{proof}

The following example describes an algebra which is standardly stratified (but not properly stratified) for all linear orders.

\begin{example}
Let $A$ be the algebra defined by the following quiver with relations: $\delta^2 = \beta \delta =0$, $\alpha \delta = \gamma \beta$.
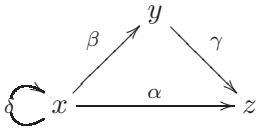
\begin{equation*}
\xymatrix { & y \ar[dr] ^{\gamma} & \\
x \ar@(ld, lu)|[] {\delta} \ar[ur] ^{\beta} \ar[rr] ^{\alpha} & & z}
\end{equation*}
The indecomposable left projective modules are described as follows:
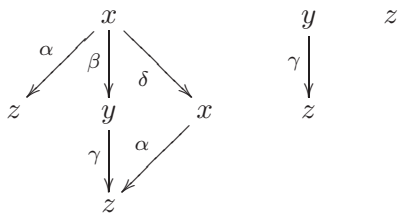
\begin{equation*}
\xymatrix{ & x \ar[dl] _{\alpha} \ar[dr]_{\delta} \ar[d]_{\beta} & \\
z & y \ar[d]_{\gamma} & x \ar[dl] _{\alpha} \\
 & z & }
\qquad \xymatrix{ y \ar[d] _{\gamma} \\ z}
\qquad {z}
\end{equation*}
It is easy to see that $J \cong P_y \oplus P_z \oplus P_z$ is a left projective $A$-module, so $A$ is standardly stratified for all linear orders.

On the other hand, the indecomposable right projective modules have the following structures:
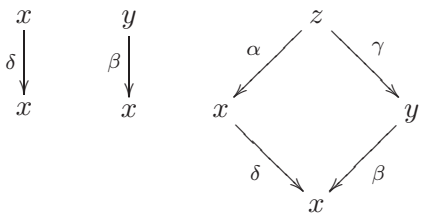
\begin{equation*}
\xymatrix{ x \ar[d] _{\delta} \\ x}
\qquad \xymatrix{ y \ar[d] _{\beta} \\ x}
\qquad \xymatrix{ & z \ar[dl] _{\alpha} \ar[dr] ^{\gamma} & \\
x \ar[dr] _{\delta} &  & y \ar[dl] ^{\beta} \\
 & x & }
\end{equation*}
Thus $J$ is not a right projective $A$-module, and hence $A$ is not properly stratified for all linear orders.
\end{example}

\section{The classification}

Let $A$ be a finite-dimensional basic algebra with a chosen set of orthogonal primitive idempotents $\{ e _{\lambda} \} _{\lambda \in \Lambda}$ such that $\sum _{\lambda \in \Lambda} e_{\lambda} = 1$. We also suppose that the associated category $\mathcal{A}$ is directed. Thus we can define $A_0$ and $J$ as before, and consider its associated graded algebra $\check{A} = \bigoplus _{i \geqslant 0} J^i / J^{i+1}$, where we set $J^0 = A$. Correspondingly, for a finitely generated $A$-module $M$, its associated graded $\check{A}$-module $\check{M} = \bigoplus _{i \geqslant 0} J^iM/ J^{i+1} M$. Clearly, we have $\check{A}_i \cdot \check{A}_j = \check{A} _{i+j}$ for $i, j \geqslant 0$.

\begin{lemma}
Let $M$ be an $A$-module and $\check{M}$ be the associated graded $\check{A}$-module. Then $\check{M}$ is a graded projective $\check{A}$-module if and only if $M$ is a projective $A$-module.
\end{lemma}

\begin{proof}
Without loss of generality we can assume that $M$ is indecomposable. Since the free module $_AA$ is sent to the graded free module $_{\check{A}} \check{A}$ by the grading process, we know that $\check{M}$ is a graded projective $\check{A}$-module generated in degree 0 if $M$ is a projective $A$-module. Now suppose that $M$ is not a projective $A$-module but $\check{M}$ is a projective $\check{A}$-module, and we want to get a contradiction.

Let $p: P \rightarrow M$ be a projective covering map of $M$. Then $\dim _k M < \dim _k P$ since $M$ is not projective. Moreover, $\Top (P) = P / \rad P \cong \Top (M) = M / \rad M$. The surjective map $p$ gives a graded surjective map $\check{p}: \check{P} \rightarrow \check{M}$. Since $\check{M}$ is supposed to be projective and $\check{P}$ is projective, we get a splitting exact sequence of graded projective $\check{A}$-modules generated in degree 0 as follows:
\begin{equation*}
\xymatrix{0 \ar[r] & \check{L} \ar[r] & \check{P} \ar[r] & \check{M} \ar[r] & 0}.
\end{equation*}
Notice that $\check{L} \neq 0$ since $\dim_k \check{M} = \dim_k M < \dim_k P = \dim_k \check{P}$. Consider the degree 0 parts. We have a splitting sequence of $A_0$-modules
\begin{equation*}
\xymatrix{0 \ar[r] & \check{L}_0 \ar[r] & \check{P}_0 \ar[r] & \check{M}_0 \ar[r] & 0}
\end{equation*}
which by definition is isomorphic to
\begin{equation*}
\xymatrix{0 \ar[r] & \check{L}_0 \ar[r] & P/JP \ar[r] & M/JM \ar[r] & 0}.
\end{equation*}
View them as $A$-modules on which $J$ acts as 0. Observe that $J$ is contained in the radical of $A$. Thus $\Top (P) \cong \Top (P/JP)$ and $\Top (M) \cong \Top (M/JM)$. But the above sequence splits, hence $\Top (P) \cong \Top (\check{L}_0) \oplus \Top (M)$, contradicting the fact that $\Top (P) \cong \Top (M)$. This finishes the proof.
\end{proof}

The above lemma still holds if we replace left modules by right modules. It immediately implies the following result:

\begin{proposition}
Let $A$ be a finite-dimensional basic algebra whose associated category is directed. Let $\check{A}$ be the associated graded algebra. Then $A$ is standardly (resp., properly) stratified for all linear orders if and only if so is $\check{A}$.
\end{proposition}

\begin{proof}
The algebra $A$ is standardly stratified for all linear orders if and only if the associated category $\mathcal{A}$ is a directed category and $J$ is a projective $A$-module. This happens if and only if the graded category $\check {\mathcal{A}}$ is a directed category and $\check{J}$ is a projective $\check{A}$-module by the previous lemma, and if and only if $\check{A}$ is standardly stratified for all linear orders.
\end{proof}

In general $\check{A}$ (when viewed as a non-graded algebra) is not isomorphic to $A$, as shown by the following example.

\begin{example}
Let $A$ be the algebra described in Example 6.1.10, where we proved that it is standardly stratified for all linear orders. It is easy to see $J/J^2 = \langle \bar{\alpha}, \bar{\beta}, \bar{\gamma} \rangle$, $J^2 / J^3 = \langle \bar{\gamma} \bar{\beta} \rangle$ and $J^3 = 0$. Therefore, $\check{A}$ is defined by the following quiver with relations $\bar{\alpha} \delta = \bar{\beta} \bar{\delta} = 0$, which is not isomorphic to $A$.
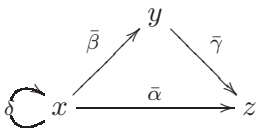
\begin{equation*}
\xymatrix { & y \ar[dr] ^{\bar{\gamma}} & \\
x \ar@(ld, lu)|[] {\delta} \ar[ur] ^{\bar{\beta}} \ar[rr] ^{\bar{\alpha}} & & z}
\end{equation*}
The indecomposable left projective $\check{A}$-modules have the following structures:
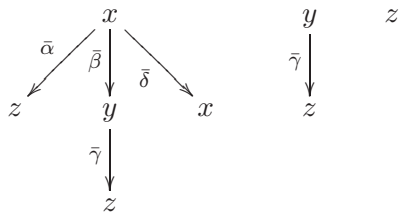
\begin{equation*}
\xymatrix{ & x \ar[dr] _{\bar{\delta}} \ar[dl] _{\bar{\alpha}} \ar[d] _{\bar{\beta}} & \\
z & y \ar[d] _{\bar{\gamma}} & x \\
 & z & }
\qquad \xymatrix{ y \ar[d] _{\bar{\gamma}} \\ z}
\qquad {z}
\end{equation*}
It still holds that $\check{J} \cong P_y \oplus P_z \oplus P_z$, so $\check{A}$ is standardly stratified for all linear orders.

The indecomposable right projective $\check{A}$-modules are as follows:
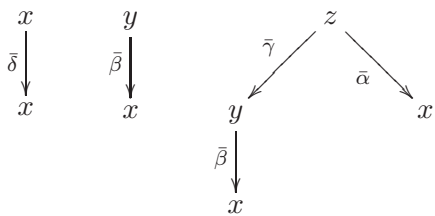
\begin{equation*}
\xymatrix{ x \ar[d] _{\bar{\delta}} \\ x} \qquad \xymatrix{ y \ar[d] _{\bar{\beta}} \\ x} \qquad \xymatrix{ & z \ar[dl] _{\bar{\gamma}} \ar[dr] _{\bar{\alpha}} & \\ y \ar[d] _{\bar{\beta}} & & x \\ x}
\end{equation*}
We deduce that $\check{A}$ is not properly stratified for all linear orders since $\check{J}$ is not a right projective module.
\end{example}

Let $X$ be an $(A_0, A_0)$-bimodule. We denote the tensor algebra generated by $A_0$ and $X$ to be $A_0[X]$. That is, $A_0[X] = A_0 \oplus X \oplus (X \otimes _{A_0} X) \oplus \ldots$. With this notation, we have:

\begin{lemma}
Let $A = \bigoplus _{i \geqslant 1} A_i$ be a finite-dimensional graded algebra with $A_i \cdot A_j = A_{i+j}$, $i, j \geqslant 0$. Then $J = \bigoplus _{i \geqslant 1} A_i$ is a projective $A$-module if and only if $A \cong A_0 [A_1]$, and $A_1$ is a projective $A_0$-module.
\end{lemma}

\begin{proof}
Suppose that $A = A_0[A_1]$ and $A_1$ is a projective $A_0$-module. Observe that $A$ is a finite-dimensional algebra, so there exists a minimal number $n \geqslant 0$ such that $A_{n+1} \cong \underbrace{ A_1 \otimes _{A_0} \ldots \otimes _{A_0} A_1}_{n+1} = 0$. Without loss of generality we assume that $n >0$ since otherwise $J = 0$ is a trivial projective module. Therefore,
\begin{align*}
J & = \bigoplus _{i = 1}^n A_i = A_1 \oplus (A_1 \otimes _{A_0} A_1) \oplus  \ldots \oplus (\underbrace{ A_1 \otimes _{A_0} \ldots \otimes _{A_0} A_1}_n) \\
& = A_1 \oplus (A_1 \otimes _{A_0} A_1) \oplus  \ldots \oplus (\underbrace{ A_1 \otimes _{A_0} \ldots \otimes _{A_0} A_1}_{n+1}) \\
& \cong \big{(} A_0 \oplus A_1 \oplus  \ldots \oplus (\underbrace{ A_1 \otimes _{A_0} \ldots \otimes _{A_0} A_1}_n) \big{)} \otimes _{A_0} A_1 \\
& = A \otimes _{A_0} A_1,
\end{align*}
which is projective since $A_1$ is a projective $A_0$-module.

Conversely, suppose that $J = \bigoplus _{i \geqslant 1} A_i$ is a projective $A$-module. Since $J$ is generated in degree 1, $A_1$ must be a projective $A_0$-module. Moreover, we have a minimal projective resolution of $A_0$ as follows
\begin{equation*}
\xymatrix {\ldots \ar[r] & A \otimes _{A_0} A_1 \ar[r] & A \ar[r] & A_0 \ar[r] & 0}.
\end{equation*}
But we also have the short exact sequence $0 \rightarrow J \rightarrow A \rightarrow A_0 \rightarrow 0$. Since $J$ is projective, we deduce that $J \cong A \otimes _{A_0} A_1$. Therefore, $A_2 \cong A_1 \otimes _{A_0} A_1$, $A_3 \cong A_1 \otimes _{A_0} A_1 \otimes _{A_0} A_1$, and so on. Thus $A \cong A_0[A_1]$ as claimed.
\end{proof}

Now we describe a classification for algebras stratified for all linear orders.

\begin{theorem}
Let $A$ be a basic finite-dimensional $k$-algebra whose associated category $\mathcal{A}$ is directed. Then the following are equivalent:
\begin{enumerate}
\item $A$ is standardly stratified (resp., properly stratified) for all linear orders;
\item the associated graded algebra $\check{A}$ is standardly stratified (resp., properly stratified) for all linear orders;
\item $\check{A}$ is the tensor algebra generated by $A_0 = \bigoplus _{\lambda \in \Lambda} \mathcal{A} (e_{\lambda}, e_{\lambda}) = e_{\lambda} A e_{\lambda}$ and a left (resp., left and right) projective $A_0$-module $\check{A}_1$.
\end{enumerate}
\end{theorem}

\begin{proof}
For standardly stratified algebras, the equivalence of (1) and (2) has been established in Proposition 6.2.2, and the equivalence of (2) and (3) comes from the previous lemma and Theorem 6.1.9. Since all arguments work for right modules, we also have the equivalence of these statements for properly stratified algebras.
\end{proof}

We end this section with a combinatorial description of $\check {\mathcal{A}}$, the associated graded category of $\check{A}$ stratified for all linear orders. Let $Q = (Q_0, Q_1)$ be a finite acyclic quiver, where both the vertex set $Q_0$ and the arrow set $Q_1$ are finite. We then define a \textit{quiver of bimodules} $\tilde{Q} = (Q_0, Q_1, f, g)$: to each vertex $v \in Q_0$ the map $f$ assigns a finite-dimensional local algebra $A_v$, i.e., $f(v) = A_v$; for each arrow $\alpha: v \rightarrow w$, $g(\alpha)$ is a finite-dimensional $(A_w, A_v)$-bimodule.

The quiver of bimodules $\tilde{Q}$ determines a category $\mathcal{C}$. Explicitly, $\Ob \mathcal{C} = Q_0$. The morphisms between an arbitrary pair of objects $v, w \in Q_0$ are defined as follows. Let
\begin{equation*}
\xymatrix {\gamma: v = v_0 \ar[r] ^{\alpha_1} \ar[r] & v_1 \ar[r] ^{\alpha_2} & \ldots \ar[r] ^{\alpha_{n-1}} & v_{n-1} \ar[r] ^{\alpha_n} & v_n = w}
\end{equation*}
be an oriented path in $Q$. We define
\begin{equation*}
M_{\gamma} = g(\alpha_n) \otimes _{f (v_{n-1})} g(\alpha_{n-1}) \otimes _{f (v_{n-2})} \ldots \otimes_{f (v_1)} g(\alpha_1).
\end{equation*}
This is a $(f(w), f(v))$-bimodule. Then
\begin{equation*}
\mathcal{C} (v, w) = \bigoplus _{\gamma \in P(v, w)} M_{\gamma},
\end{equation*}
Where $P(v, w)$ is the set of all oriented paths from $v$ to $w$. The composite of morphisms is defined by tensor product. We call a category defined in this way a \textit{free directed} category. It is \textit{left regular} if for every arrow $\alpha: v \rightarrow w$, $g(\alpha)$ is a left projective $f(w)$-module. Similarly, we define \textit{right regular} categories. It is \textit{regular} if this category is both left regular and right regular.

Using these terminologies, we get the following combinatorial description of $\check{\mathcal{A}}$.

\begin{theorem}
Let $A$ be a finite-dimensional basic algebra whose associated category is directed. Then $A$ is standardly (resp., properly) stratified for all linear orders if and only if the graded category $\check {\mathcal{A}}$ is a left regular (resp., regular) free directed category.
\end{theorem}

\begin{proof}
It is straightforward to check that if $\check {\mathcal{A}}$ is a left regular free directed category, then $\check{A}$ satisfies (3) in Theorem 6.2.5. So $A$ is standardly stratified for all linear orders. Conversely, if $A$ is standardly stratified for all linear orders, then $\check {\mathcal{A}}$ is a directed category, and $\check{A}$ is a tensor algebra generated by $A_0 = \bigoplus _{\lambda \in \Lambda} e_{\lambda} A e_{\lambda}$ and a projective left $A_0$-module $\check{A}_1$. Define $Q = (Q_0, Q_1, f, g)$ in the following way: $Q_0 = \Lambda$, and $f(\lambda) = e_{\lambda} A_0 e_{\lambda}$. Arrows and the map $g$ are defined as follows: for $\lambda \neq \mu \in Q_0$, we put an arrow $\phi: \lambda \rightarrow \mu$ if $e_{\mu} \check{A}_1 e_{\lambda} \neq 0$ and define $g(\phi) = e_{\mu} \check{A}_1 e_{\lambda}$. In this way $Q$ defines a left regular free directed category which is isomorphic to $\check {\mathcal{A}}$. Since all arguments work for right modules, the proof is completed.
\end{proof}

\section{ Whether $\mathcal{F} (_{\preccurlyeq} \Delta)$ is closed under cokernels of monomorphisms?}

In Proposition 6.1.4 we proved that if $\mathcal{A}$ is directed and standardly stratified with respect to a linear order $\leqslant$, then the corresponding category $\mathcal{F} (_{\leqslant} \Delta)$ is closed under cokernels of monomorphisms. This result motivates us to study the general situation. Suppose $A$ is standardly stratified with respect to a fixed linear ordered set $(\Lambda, \preccurlyeq)$. In the following lemmas we describe several equivalent conditions for $\mathcal{F} (_{\preccurlyeq} \Delta)$ to be closed under cokernels of monomorphisms.

\begin{lemma}
The category $\mathcal{F} (_{\preccurlyeq} \Delta)$ is closed under cokernels of monomorphisms if and only if for each exact sequence $0 \rightarrow L \rightarrow M \rightarrow N \rightarrow 0$, where $L, M \in \mathcal{F} (_{\preccurlyeq} \Delta)$ and $L$ is indecomposable, $N$ is also contained in $\mathcal{F} (_{\preccurlyeq} \Delta)$.
\end{lemma}

\begin{proof}
The only if direction is trivial, we prove the other direction: for each exact sequence $0 \rightarrow \tilde{L} \rightarrow M \rightarrow N \rightarrow 0$ with $\tilde{L}, M \in \mathcal{F} (_{\preccurlyeq} \Delta)$, we have $N \in \mathcal{F} (_{\preccurlyeq} \Delta)$ as well.

We use induction on the number of indecomposable direct summands of $\tilde{L}$. If $\tilde{L}$ is indecomposable, the conclusion holds obviously. Now suppose that the if part is true for $\tilde{L}$ with at most $l$ indecomposable summands. Assume that $\tilde{L}$ has $l+1$ indecomposable summands. Taking an indecomposable summand $L_1$ of $\tilde{L}$ we have the following diagram:
\begin{equation*}
\xymatrix{ & 0 \ar[d] & 0 \ar[d] \\
0 \ar[r] & L_1 \ar[d] \ar@{=}[r] & L_1 \ar[r] \ar[d] & 0\\
0 \ar[r] & \tilde{L} \ar[r] \ar[d] & M \ar[r] \ar[d] & N \ar[r] \ar@{=}[d] & 0 \\
0 \ar[r] & \bar{L} \ar[d] \ar[r] & \bar{M} \ar[d] \ar[r] & N \ar[r] & 0\\
 & 0 & 0}
\end{equation*}
Considering the middle column, $\bar{M} \in \mathcal{F} (_{\preccurlyeq} \Delta)$ by the given condition. Therefore, we conclude that $N \in \mathcal{F} (_{\preccurlyeq} \Delta)$ by using the induction hypothesis on the bottom row.
\end{proof}

\begin{lemma}
The category $\mathcal{F} (_{\preccurlyeq} \Delta)$ is closed under cokernels of monomorphisms if and only if for each exact sequence $0 \rightarrow L \rightarrow P \rightarrow N \rightarrow 0$, where $L, P \in \mathcal{F} (_{\preccurlyeq} \Delta)$ and $P$ is projective, $N$ is also contained in $\mathcal{F} (_{\preccurlyeq} \Delta)$.
\end{lemma}

\begin{proof}
It suffices to show the if part. Let $0 \rightarrow L \rightarrow M \rightarrow N \rightarrow 0$ be an exact sequence with $L, M \in \mathcal{F} (_{\preccurlyeq} \Delta)$. We want to show $N \in \mathcal{F} (_{\preccurlyeq} \Delta)$ as well. Let $P$ be a projective cover of $M$. Then we have a commutative diagram by the Snake Lemma:
\begin{equation*}
\xymatrix{ & 0 \ar[r] & \Omega(M) \ar[r] \ar[d] & N' \ar[r] \ar[d] & L \ar[r] & 0\\
 & & P \ar@{=}[r] \ar[d] & P \ar[d] \\
0 \ar[r] & L \ar[r] & M \ar[r] & N \ar[r] & 0,}
\end{equation*}
where all rows and columns are exact. Clearly, $\Omega(M) \in \mathcal{F} (_{\preccurlyeq} \Delta)$, so is $N'$ since $\mathcal{F} (_{\preccurlyeq} \Delta)$ is closed under extension. Considering the last column, we conclude that $N \in \mathcal{F} (_{\preccurlyeq} \Delta)$ by the given condition.
\end{proof}

Every $M \in \mathcal{F} (_{\preccurlyeq} \Delta)$ has a $_{\preccurlyeq} \Delta$-filtration $\xi$ and we define $[M: \Delta_{\lambda}]$ to be the number of factors isomorphic to $\Delta_{\lambda}$ in $\xi$. This number is independent of the particular $\xi$ (see \cite{Dlab1, Erdmann}). We then define $l(M) = \sum _{\lambda \in \Lambda} [M: \Delta_{\lambda}]$ and call it the \textit{filtration length} of $M$, which is also independent of the choice of $\xi$.

\begin{lemma}
The category $\mathcal{F} (_{\preccurlyeq} \Delta)$ is closed under cokernels of monomorphisms if and only if for each exact sequence $0 \rightarrow \Delta_{\lambda} \rightarrow M \rightarrow N \rightarrow 0$, where $M \in \mathcal{F} (_{\preccurlyeq} \Delta)$ and $\lambda \in \Lambda$, $N$ is also contained in $\mathcal{F} (_{\preccurlyeq} \Delta)$.
\end{lemma}

\begin{proof}
We only need to show the if part. Let $0 \rightarrow L \rightarrow M \rightarrow N \rightarrow 0$ be an exact sequence with $L, M \in \mathcal{F} (_{\preccurlyeq} \Delta)$. We use induction on the filtration length of $L$.

If $l(L) = 1$, then $L \cong \Delta_{\lambda}$ for some $\lambda \in \Lambda$ and the conclusion holds clearly. Suppose that the conclusion is true for all objects in $\mathcal{F} (_{\preccurlyeq} \Delta)$ with filtration length at most $s$ and assume $l(L) = s+1$. Then we have an exact sequence $0 \rightarrow L' \rightarrow L \rightarrow \Delta_{\lambda} \rightarrow 0$ for some $\lambda \in \Lambda$ and $l(L') = s$. Then we have a commutative diagram by the Snake Lemma:

\begin{equation*}
\xymatrix{ & 0 \ar[r] & L' \ar[d] \ar[r] & L \ar[r] \ar[d] & \Delta_{\lambda} \ar[r] & 0\\
 & & M \ar[d] \ar@{=}[r] & M \ar[d]\\
0 \ar[r] & \Delta_{\lambda} \ar[r] & N' \ar[r] & N \ar[r] & 0.}
\end{equation*}

Consider the first column. By induction hypothesis, $N' \in \mathcal{F} (_{\preccurlyeq} \Delta)$. By considering the bottom row we conclude that $N \in \mathcal{F} (_{\preccurlyeq} \Delta)$ from the given condition.
\end{proof}

The following theorem contains a partial answer to the question of whether $\mathcal{F} (\Delta)$ is closed under cokernels of monomorphisms.

\begin{theorem}
Let $A$ be a finite-dimensional basic algebra standardly stratified for a linear order $\preccurlyeq$. Then:
\begin{enumerate}
\item $\mathcal{F} (_{\preccurlyeq} \Delta)$ is closed under cokernels of monomorphisms if and only if the cokernel of every monomorphism $\iota: \Delta_{\lambda} \rightarrow P$ is contained in $\mathcal{F} (_{\preccurlyeq} \Delta)$, where $P$ is an arbitrary projective module and $\lambda \in \Lambda$.
\item If $\mathcal{F} (_{\preccurlyeq} \Delta)$ is closed under cokernels of monomorphisms and $A$ is standardly stratified for another linear order $\preccurlyeq'$, then $\mathcal{F} (_{\preccurlyeq'} \Delta) \subseteq \mathcal{F} (_{\preccurlyeq} \Delta)$.
\item If $A$ is quasi-hereditary, then $\mathcal{F} (_{\preccurlyeq} \Delta)$ is closed under cokernels of monomorphisms if and only if $A$ is a quotient of a finite-dimensional hereditary algebra and all standard modules are simple.
\end{enumerate}
\end{theorem}

\begin{proof}
The first statement follows immediately from the above lemmas.

Now we prove the second statement. The claim is clear if $|\Lambda|=1$. So we assume $\Lambda$ has more than one elements. Take an element $\lambda \in \Lambda$ maximal with respect to $\preccurlyeq '$. Clearly, $\Delta'_{\lambda} \cong P_{\lambda} \in \mathcal{F} (_{\preccurlyeq} \Delta)$. Consider the quotient algebra $\bar{A} = A / Ae_{\lambda}A$. It is standardly stratified with respect to the restricted linear order $\preccurlyeq '$ on $\Lambda \setminus \{ \lambda\}$. We claim $\bar{A} \in \mathcal{F} (_{\preccurlyeq} \Delta)$ as well. Indeed, consider the exact sequence
\begin{equation*}
\xymatrix { 0 \ar[r] & Ae_{\lambda}A \ar[r] & A \ar[r] & \bar{A} \ar[r] & 0}.
\end{equation*}
Note that $A e_{\lambda}A \in \mathcal{F} (_{\preccurlyeq} \Delta)$ as a direct sum of $P_{\lambda}$. Therefore, $\bar{A} \in \mathcal{F} (_{\preccurlyeq} \Delta)$ as well since $\mathcal{F} (_{\preccurlyeq} \Delta)$ is closed under cokernels of monomorphisms. Taking a maximal element $\nu \in \Lambda \setminus \{\lambda \}$ with respect to $\preccurlyeq'$ and repeating the above procedure, we conclude that $_{\preccurlyeq'} \Delta_{\nu} \cong \bar{A} e_{\nu} \in \mathcal{F} (_{\preccurlyeq} \Delta)$ and $\bar{\bar{A}} = \bar{A} / \bar{A} e_{\nu} \bar{A} \in \mathcal{F} (_{\preccurlyeq} \Delta)$. Recursively, we proved that $_{\preccurlyeq'} \Delta_{\lambda} \in \mathcal{F} (\Delta)$ for all $\lambda \in \Lambda$, so $\mathcal{F} (_{\preccurlyeq'} \Delta) \subseteq \mathcal{F} (_{\preccurlyeq} \Delta)$.

The third statement can be proved by induction on $|\Lambda|$ as well. If $|\Lambda| = 1$, the claim is clear since $A \cong k$. Suppose that the conclusion holds for $|\Lambda| = s$ and let $\Lambda$ be a linear ordered set with $s+1$ elements. Take a maximal element $\lambda$ in $\Lambda$. Then $\bar{A} = A / Ae_{\lambda}A$ is also a quasi-hereditary algebra. Moreover, $\mathcal{F} (_{\bar{A}} \Delta)$ is still closed under cokernels of monomorphisms. By the induction hypothesis, $\bar{A}$ is the quotient of a finite-dimensional hereditary algebra, and all standard modules are simple. Note that these standard $\bar {A}$-modules can be viewed as standard $A$-modules.

Choose a composition series of $\Delta_{\lambda} \cong P_{\lambda}$: $0 = M_0 \subseteq M_1 \subseteq \ldots \subseteq M_t = \Delta_{\lambda}$. It is clear that $M_i / M_{i-1} \cong S_{\lambda}$ if and only if $i=t$ since $A$ is quasi-hereditary. But $\mathcal{F} (_{\preccurlyeq} \Delta)$ is closed under cokernels of monomorphisms and $\Delta_{\mu} \cong S_{\mu} \cong P_{\mu} / \rad P_{\mu}$ for all $\mu \in \Lambda \setminus \{ \lambda \}$, we deduce that $S_{\lambda} \in \mathcal{F} (_{\preccurlyeq} \Delta)$. Thus $S_{\lambda} \cong \Delta_{\lambda}$.

It remains to show that the ordinary quiver of $A$ has no oriented cycles. For $P_{\lambda} = Ae_{\lambda}$ and $P_{\mu} = Ae_{\mu}$ with $\lambda \nsucceq \mu$, we claim $e_{\lambda} A e_{\mu}  \cong \Hom _{\mathcal{A}} (P_{\lambda}, P_{\mu}) = 0$. Indeed, since $P_{\mu} \in \mathcal{F} (_{\preccurlyeq} \Delta)$ and all standard modules are simple, the composition factors of $P_{\mu}$ are those $S_{\nu}$ with $\nu \in \Lambda$ and $\nu \succcurlyeq \mu$. Thus $P_{\mu}$ has no composition factors isomorphic to $S_{\lambda}$. Therefore, $\Hom _{\mathcal{A}} (P_{\lambda}, P_{\mu}) =0$.
\end{proof}

An immediate corollary is:

\begin{corollary}
If $A$ is standardly stratified with respect to two different linear orders $\preccurlyeq$ and $\preccurlyeq'$ such that both $\mathcal{F} (_{\preccurlyeq} \Delta)$ and $\mathcal{F} (_{\preccurlyeq} \Delta)$ are closed under cokernels of monomorphisms, then $\mathcal{F} (_{\preccurlyeq} \Delta) = \mathcal{F} (_{\preccurlyeq'} \Delta)$.
\end{corollary}

\begin{proof}
It is straightforward from the second statement of the previous theorem.
\end{proof}

In general it is possible that there are more than one linear orders with respect to which $A$ is standardly stratified. However, among these categories $\mathcal{F} (_{\preccurlyeq} \Delta)$, there is at most one which is closed under cokernels of monomorphisms. If this category exists, it is the unique maximal category and contains all other ones as subcategories. But the converse of this statement is not true by the next example. That is, if there is a linear order $\preccurlyeq$ with respect to which $A$ is standardly stratified and $\mathcal{F} (_{\preccurlyeq} \Delta)$ is the unique maximal category, $\mathcal{F} (_{\preccurlyeq} \Delta)$ might not be closed under cokernels of monomorphisms.

\begin{example}
Let $A$ be the path algebra of the following quiver with relations: $\gamma^2 = \rho^2 = \delta \gamma = \rho \delta =0$, $\gamma \alpha = \alpha'$, $\gamma \beta = \beta'$, and $\delta \alpha = \delta \beta$.
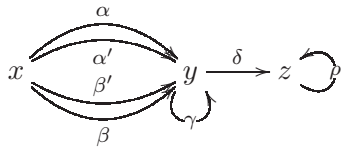
\begin{equation*}
\xymatrix {x \ar@/^1.5pc/[rr] ^{\alpha} \ar@/^1pc/[rr] _{\alpha'} \ar@/_1.5pc/[rr] _{\beta} \ar@/_1pc/[rr] ^{\beta'} & & y \ar@(dl, dr)[]|{\gamma} \ar[r]^{\delta} & z \ar@(rd,ru)[]|{\rho} }
\end{equation*}

The structures of indecomposable projective $A$-modules are described as follows:

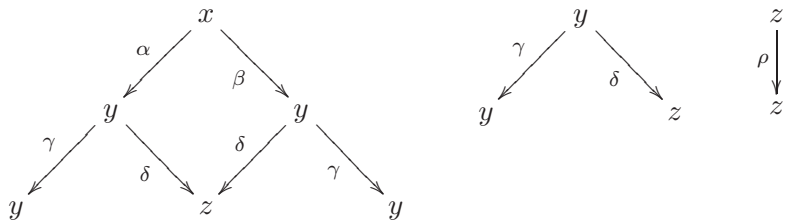
\begin{equation*}
\xymatrix{ & & x \ar[dl] _{\alpha} \ar[dr] _{\beta} & & \\ & y \ar[dl] _{\gamma} \ar[dr] _{\delta} & & y \ar[dl] _{\delta} \ar[dr] _{\gamma}& \\ y & & z & & y } \qquad
\xymatrix{ & y \ar[dl] _{\gamma} \ar[dr] _{\delta} & \\ y & & z} \qquad
\xymatrix{ z \ar[d] _{\rho} \\ z}
\end{equation*}

It is not hard to check that if $A$ is standardly stratified with respect to some linear order $\preccurlyeq$, then all standard modules coincide with indecomposable projective modules. Therefore, $\mathcal{F} (_ {\preccurlyeq} \Delta)$ is the category of all projective modules and is the unique maximal category. However, we get an exact sequence as follows, showing that $\mathcal{F} (\Delta)$ is not closed under cokernels of monomorphisms:
\begin{equation*}
\xymatrix{ 0 \ar[r] & P_y \ar[r] & P_x \ar[r] & M \ar[r] & 0}
\end{equation*}
where $M \notin \mathcal{F} (\Delta)$ has the following structure:
\begin{equation*}
\xymatrix{ x \ar[d] _{\alpha} \\ y\ar[d] _{\gamma} \\ y}
\end{equation*}
\end{example}

The last statement of Theorem 6.3.4 is incorrect if $A$ is supposed to be standardly stratified, as shown by the next example.

\begin{example}
Let $A$ be the path algebra of the following quiver with relations $\delta^2 = \delta \alpha = \beta \delta = \beta \alpha = \gamma \beta =0$. Let $x \succ z \succ y$.
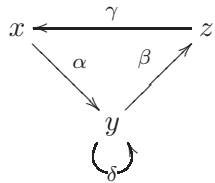
\begin{equation*}
\xymatrix{ x \ar[dr] ^{\alpha} & & z \ar[ll] _{\gamma}\\
& y \ar[ur] ^{\beta} \ar@(dl,dr)[]|{\delta} }
\end{equation*}
Indecomposable projective modules of $A$ are described below:
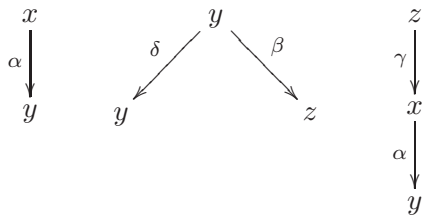
\begin{equation*}
\xymatrix{ x \ar[d] _{\alpha} \\ y} \qquad
\xymatrix{ & y \ar[dl] _{\delta} \ar[dr] ^{\beta} & \\ y & & z}
\qquad \xymatrix{ z \ar[d] _{\gamma} \\ x \ar[d] _{\alpha} \\ y}
\end{equation*}
The standard modules are:
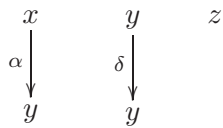
\begin{equation*}
\xymatrix{ x \ar[d] _{\alpha} \\ y} \qquad
\xymatrix{ y \ar[d] _{\delta} \\ y}
\qquad z
\end{equation*}
It is clear that $A$ is standardly stratified. Moreover, by (1) of Theorem 0.3, $\mathcal{F} (\Delta)$ is closed under cokernels of monomorphisms. But $\mathcal{A}$ is not a directed category.
\end{example}

\section{An algorithm}

In this section we describe an algorithm to determine whether there is a linear order $\preccurlyeq$ with respect to which $A$ is standardly stratified and $\mathcal{F} (_{\preccurlyeq} \Delta)$ is closed under cokernels, as well as several examples.

Given an arbitrary algebra $A$, we want to check whether there exists a linear order $\preccurlyeq$ for which $A$ is standardly stratified and the corresponding category $\mathcal{F} (_{\preccurlyeq} \Delta)$ is closed under cokernels. It is certainly not an ideal way to check all linear orders. In the rest of this chapter we will describe an algorithm to construct a set $\mathcal{L}$ of linear orders for $A$ satisfying the following property: $A$ is standardly stratified for every linear order in $\mathcal{L}$; moreover, if there is a linear order $\preccurlyeq$ for which $A$ is standardly stratified and $\mathcal{F} (_{\preccurlyeq} \Delta)$ is closed under cokernels, then $\preccurlyeq \in \mathcal{L}$.

As before, choose a set $\{e_{\lambda} \} _{\lambda \in \Lambda}$ of orthogonal primitive idempotents in $A$ such that $\sum _{\lambda \in \Lambda} e_{\lambda} = 1$ and let $P_{\lambda} = A e_{\lambda}$. The algorithm is as follows:

\begin{enumerate}
\item Define $O_1 = \{ \lambda \in \Lambda \mid \forall \mu \in \Lambda, \text{tr} _{P_{\lambda}} (P_{\mu}) \cong P_{\lambda}^{m_{\mu}} \}$. If $O_1 = \emptyset$, the algorithm ends at this step. Otherwise, continue to the second step.
\item Define a partial order $\leqslant '$ on $O_1$: $\lambda \leqslant ' \mu$ if and only if tr$_{P_{\mu}} (P_{\lambda}) \neq 0$ for $\lambda, \mu \in O_1$. We can check that this partial order $\leqslant'$ is well defined.
\item Take $e_{s_1} \in O_1$ which is maximal with respect to $\leqslant'$. Let $\bar {A} = A / A e_{s_1} A$.
\item Repeat the above steps for $\bar{A}$ recursively until the algorithm ends. Thus we get a chain of $t$ idempotents $e_{s_1}, e_{s_2}, \ldots, e_{s_t}$ and define $e_{s_1} \succ e_{s_2} \succ \ldots \succ e_{s_t}$.
\end{enumerate}
Let $\tilde{ \mathcal{L} }$ be the set of all linear orders obtained from the above algorithm, and let $\mathcal{L} \subseteq \tilde{ \mathcal{L} }$ be the set of linear orders with length $n = |\Lambda|$.

The following example illustrates our algorithm.

\begin{example}
Let $\mathcal{A}$ be the following algebra with relation $\alpha^2 = \delta^2 = \beta \alpha = \delta \gamma = \rho^2 = \rho \varphi = 0$.
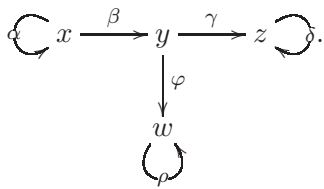
\begin{equation*}
\xymatrix{ x \ar[r] ^{\beta} \ar@(lu, ld)[]|{\alpha} & y \ar[r] ^{\gamma} \ar[d] ^{\varphi} & z \ar@(ru, rd)[]|{\delta} \\
 & w \ar@(dl, dr)[]|{\rho}}.
\end{equation*}
The projective $\mathcal{A}$-modules are:
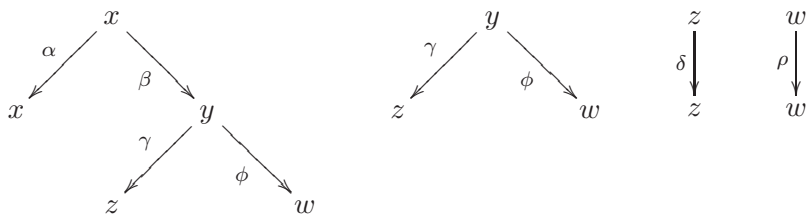
\begin{equation*}
\xymatrix{ & x \ar[dl] _{\alpha} \ar[dr] _{\beta} & & \\ x & & y \ar[dl] _{\gamma} \ar[dr] _{\phi} & \\ & z & & w}
\qquad \xymatrix{ & y \ar[dl] _{\gamma} \ar[dr] _{\phi} & \\ z & & w}
\qquad \xymatrix{ z \ar[d] _{\delta} \\ z}
\qquad \xymatrix{ w \ar[d] _{\rho} \\ w}
\end{equation*}

Then by the above algorithm, $O_1 = \{ x, y \}$ and $x \leqslant' y$ in $O_1$, we should take $y$ as the maximal element. But then $O_2 = \{ x, z, w \}$ and all these elements are maximal in $O_2$ with respect to $\leqslant '$. Thus we get three choices for $O_3$. Similarly, the two elements in each $O_3$ are maximal with respect to $\leqslant'$. In conclusion, 6 linear orders are contained in $\mathcal{L}$: $y \succ x \succ z \succ w$, $y \succ x \succ w \succ z$, $y \succ z \succ x \succ w$, $y \succ z \succ w \succ x$, $y \succ w \succ z \succ x$, and $y \succ w \succ x \succ z$. For all these six linear orders $\mathcal{A}$ is standardly stratified and has the same standard modules. Moreover, the category $\mathcal{F} (_{\preccurlyeq} \Delta)$ is closed under cokernels of monomorphisms.
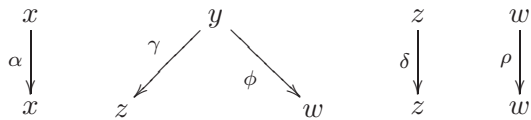
\begin{equation*}
\xymatrix{ x \ar[d] _{\alpha} \\ x}
\qquad \xymatrix{ & y \ar[dl] _{\gamma} \ar[dr] _{\phi} & \\ z & & w}
\qquad \xymatrix{ z \ar[d] _{\delta} \\ z}
\qquad \xymatrix{ w \ar[d] _{\rho} \\ w}
\end{equation*}
\end{example}

In general, for different linear orders in $\mathcal{L}$ the corresponding standard modules are different.

\begin{example}
Let $\mathcal{A}$ be the following category with relation: $\delta^2 = \delta \alpha =0$, $\beta \delta = \beta'$.
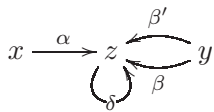
\begin{equation*}
\xymatrix{ x \ar[r] ^{\alpha} & z \ar@(dl,dr)[]|{\delta} & y \ar@/^/[l] ^{\beta} \ar@/_/[l] _{\beta'} }
\end{equation*}

The reader can check that the following two linear orders are contained in $\mathcal{L}$: $x \succ z \succ y$ and $y \succ x \succ z$. The corresponding standard modules are:
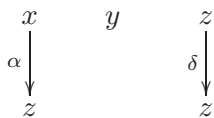
\begin{equation*}
\xymatrix{ x \ar[d] _{\alpha} \\ z} \qquad y \qquad \xymatrix{ z \ar[d] _{\delta} \\ z}
\end{equation*}
and
\begin{equation*}
\xymatrix{ x \ar[d] _{\alpha} \\ z} \qquad \xymatrix{y \ar[d] _{\beta} \\ z \ar[d] _{\delta} \\ z} \qquad \xymatrix{ z \ar[d] _{\delta} \\ z}
\end{equation*}
It is easy to see that $\mathcal{F} (\Delta)$ corresponding to the first order is closed under cokernels of monomorphisms.
\end{example}

\begin{proposition}
The algebra $\mathcal{A}$ is standardly stratified for every $\preccurlyeq \in \mathcal{L}$.
\end{proposition}

\begin{proof}
We use induction on $|\Lambda|$. The conclusion is clearly true if $|\Lambda| = 1$. Suppose that it holds for $|\Lambda| \leqslant n$ and consider the case that $|\Lambda| = n+1$.

Let $\preccurlyeq$ be an arbitrary linear order in $\mathcal{L}$ and take the unique maximal element $\lambda \in \Lambda$ with respect to $\preccurlyeq$. Consider the quotient algebra $\bar{A} = A / A e_{\lambda} A$. Then $\bar{A}$, by the induction hypothesis and our algorithm, is standardly stratified with respect to the restricted order on $\Lambda \setminus \{ \lambda \}$. It is clear from our definition of $\preccurlyeq$ that $A e_{\lambda} A = \bigoplus _{\mu \in \Lambda} \text{tr} _{P_{\lambda}} (P_{\mu})$ is projective. Thus $A$ is standardly stratified for $\preccurlyeq$.
\end{proof}

The next proposition tells us that it is enough to check linear orders in $\mathcal{L}$ to determine whether there exists a linear order $\preccurlyeq$ for which $\mathcal{A}$ is standardly stratified and the corresponding category $\mathcal{F} (_{\preccurlyeq} \Delta)$ is closed under cokernels of monomorphisms.

\begin{proposition}
Let $\preccurlyeq$ be a linear order on $\Lambda$ such that $A$ is standardly stratified and the corresponding category $\mathcal{F} (_{\preccurlyeq} \Delta)$ is closed under cokernels of monomorphisms. Then $\preccurlyeq \in \mathcal{L}$.
\end{proposition}

\begin{proof}
The proof relies on induction on $|\Lambda|$. The claim is clear if $|\Lambda| = 1$. Suppose that the conclusion is true for $| \Lambda | \leqslant l$ and let $\Lambda$ be a set with $n=l+1$ elements. Note that our algorithm is defined recursively. Furthermore, the quotient algebra $\bar {A} = A / A e_{\lambda} A$ is also standardly stratified for the restricted linear order $\preccurlyeq$ on $\Lambda \setminus \{ \lambda \}$, where $\lambda$ is maximal with respect to $\preccurlyeq$; $\bar{\Delta} = \bigoplus _{\mu \in \Lambda \setminus \{\lambda \} } \Delta_{\mu}$; and $\mathcal{F} (\bar {\Delta})$ is also closed under cokernels of monomorphisms. Thus by induction it suffices to show that $\lambda \in O_1$, and is maximal in $O_1$ with respect to $\leqslant'$ (see the second step of the algorithm).

Consider $P_{\lambda} \cong \Delta_{\lambda}$. Since $A$ is standardly stratified for $\preccurlyeq$,  for each $\mu \in \Lambda$, tr$_{P_{\lambda}} (P_{\mu})$ is a projective module. Thus $\lambda \in O_1$.

If $\lambda$ is not maximal in $O_1$ with respect to $\leqslant '$, then we can choose some $\mu \in O_1$ such that $\mu >' \lambda$, i.e., tr$ _{P_{\mu}} (P_{\lambda}) \neq 0$. Since $\mu \in O_1$, by definition, tr$ _{P_{\mu}} (P_{\lambda}) \cong P_{\mu} ^m$ for some $m \geqslant 1$. Now consider the exact sequence:
\begin{equation*}
\xymatrix{ 0 \ar[r] & \text{tr} _{P_{\mu}} (P_{\lambda}) \ar[r] & P_{\lambda} \ar[r] & P_{\lambda} / \text{tr} _{P_{\mu}} (P_{\lambda}) \ar[r] & 0.}
\end{equation*}
Since tr$ _{P_{\mu}} (P_{\lambda}) \cong P_{\mu} ^m \in \mathcal{F} (_{\preccurlyeq} \Delta)$, $P_{\lambda} \in \mathcal{F} (_{\preccurlyeq} \Delta)$, and $\mathcal{F} (_{\preccurlyeq} \Delta)$ is closed under cokernels of monomorphism, we conclude that $P_{\lambda} / \text{tr} _{P_{\mu}} (P_{\lambda}) \in \mathcal{F} (_{\preccurlyeq} \Delta)$. This is impossible. Indeed, since $P_{\lambda} / \text{tr} _{P_{\mu}} (P_{\lambda})$ has a simple top $S_{\lambda} \cong P_{\lambda} / \rad P_{\lambda}$, if it is contained in $\mathcal{F} (_{\preccurlyeq} \Delta)$, then it has a filtration factor $\Delta_{\lambda} \cong P_{\lambda}$. This is absurd.

We have proved by contradiction that ${\lambda} \in O_1$ and is maximal in $O_1$ with respect to $\leqslant'$. The conclusion follows from induction.
\end{proof}

We reminder the reader that although $A$ is standardly stratified for all linear orders in $\mathcal{L}$, it does not imply that all linear orders for which $A$ is standardly stratified are contained in $\mathcal{L}$. Moreover, it is also wrong that for every linear order $\preccurlyeq \notin \mathcal{L}$, there exists some $\tilde{ \preccurlyeq } \in \mathcal{L}$ such that $\mathcal{F} (_{\preccurlyeq} \Delta) \subseteq \mathcal{F} (_{\tilde {\preccurlyeq}} {\Delta})$. Consider the following example.

\begin{example}
Let $A$ be the path algebra of the following quiver with relations: $\delta^2 = \delta \gamma = 0$, $\delta \beta = \beta'$.
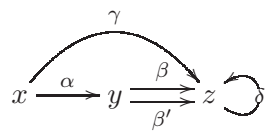
\begin{equation*}
\xymatrix{ x \ar[r]^{\alpha} \ar@/^2pc/[rr] ^{\gamma} & y \ar@<0.5ex>[r] ^{\beta} \ar@<-0.5ex>[r] _{\beta'} & z \ar@(rd,ru)[]|{\delta}}
\end{equation*}
The structures of indecomposable projective $A$-modules are:
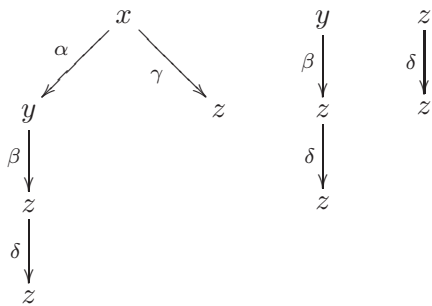
\begin{equation*}
\xymatrix{ & x \ar[dl] _{\alpha} \ar[dr] _{\gamma} & \\ y \ar[d] _{\beta} & & z \\ z \ar[d] _{\delta} & & \\ z & &} \qquad
\xymatrix{y \ar[d] _{\beta} \\ z \ar[d] _{\delta} \\ z} \qquad
\xymatrix{ z \ar[d] _{\delta} \\ z}
\end{equation*}
Applying the algorithm, we get a unique linear order $y \succ x \succ z$ contained in $\mathcal{L}$, and the corresponding standard modules are:
\begin{equation*}
\Delta_x = \xymatrix{ x \ar[d] _{\gamma} \\ z} \qquad \Delta_y \cong P_y \qquad \Delta_z \cong P_z.
\end{equation*}
But there is another linear order $x \succ z \succ y$ for which $A$ is also standardly stratified with standard modules
\begin{equation*}
\Delta_x' \cong P_x \quad \Delta_y' = y \quad \Delta_z' \cong P_z.
\end{equation*}
Both $\mathcal{F} (_{\preccurlyeq} \Delta)$ and $\mathcal{F} (_{\preccurlyeq'} \Delta)$ are not closed under cokernels of monomorphisms. Moreover, each of them is not contained in the other one.
\end{example} 




\end{document}